 \documentclass[10pt]{amsart}
\usepackage{amssymb}
\usepackage{amscd}
\usepackage[all]{xy}
\usepackage{color}
\usepackage{relsize}
\numberwithin{equation}{section}

\def\today{\number \day \space \ifcase \month \or January
 \or February \or March \or April \or May \or June \or
 July \or August \or September \or
 October \or November \or December \fi \   \number \year}

\newcounter{TmpEnumi}

\theoremstyle{definition}
\newtheorem{thm}{Theorem}[section]
\newtheorem{lem}[thm]{Lemma}
\newtheorem{prp}[thm]{Proposition}
\newtheorem{dfn}[thm]{Definition}
\newtheorem{cor}[thm]{Corollary}

\newtheorem{rmk}[thm]{Remark}
\newtheorem{ntn}[thm]{Notation}
\newtheorem{exa}[thm]{Example}

\newcommand{\beq}{\begin{equation}}
\newcommand{\eeq}{\end{equation}}
\newcommand{\beqr}{\begin{eqnarray*}}
\newcommand{\eeqr}{\end{eqnarray*}}
\newcommand{\bal}{\begin{align*}}
\newcommand{\eal}{\end{align*}}
\newcommand{\bei}{\begin{itemize}}
\newcommand{\eei}{\end{itemize}}
\newcommand{\limi}[1]{\lim_{{#1} \to \infty}}

\newcommand{\af}{\alpha}
\newcommand{\bt}{\beta}
\newcommand{\gm}{\gamma}

\newcommand{\ep}{\varepsilon}

\newcommand{\et}{\eta}
\newcommand{\ch}{\chi}
\newcommand{\io}{\iota}
\newcommand{\te}{\theta}
\newcommand{\ld}{\lambda}
\newcommand{\sm}{\sigma}
\newcommand{\kp}{\kappa}
\newcommand{\ph}{\varphi}
\newcommand{\ps}{\psi}
\newcommand{\rh}{\rho}
\newcommand{\om}{\omega}
\newcommand{\ta}{\tau}

\newcommand{\Gm}{\Gamma}

\newcommand{\Sm}{\Sigma}

\newcommand{\Z}{{\mathbb{Z}}}
\newcommand{\R}{{\mathbb{R}}}
\newcommand{\C}{{\mathbb{C}}}
\newcommand{\N}{{\mathbb{Z}}_{> 0}}
\newcommand{\Nz}{{\mathbb{Z}}_{\geq 0}}

\newcommand{\id}{{\operatorname{id}}}

\newcommand{\spec}{{\operatorname{sp}}}

\newcommand{\diag}{{\operatorname{diag}}}

\newcommand{\rank}{{\operatorname{rank}}}

\newcommand{\spn}{{\operatorname{span}}}
\newcommand{\card}{{\operatorname{card}}}

\newcommand{\dirlim}{\varinjlim}

\newcommand{\andeqn}{\, \, \, \, \, \, {\mbox{and}} \, \, \, \, \, \, }
\newcommand{\QED}{\rule{0.4em}{2ex}}

\newcommand{\wolog}{without loss of generality}

\newcommand{\tfae}{the following are equivalent}
\newcommand{\ifo}{if and only if}

\newcommand{\ca}{C*-algebra}

\newcommand{\hm}{homomorphism}

\newcommand{\pj}{projection}

\newcommand{\ct}{continuous}
\newcommand{\cfn}{continuous function}

\newcommand{\rpn}{representation}

\newcommand{\mb}{measurable}
\newcommand{\msp}{measure space}
\newcommand{\sft}{$\sm$-finite}
\newcommand{\sfm}{$\sm$-finite measure space}

\newcommand{\XBM}{ (X, {\mathcal{B}}, \mu)}
\newcommand{\YCN}{ (Y, {\mathcal{C}}, \nu)}

\newcommand{\LLp}{L (L^p (X, \mu))}
\newcommand{\LLpy}{L (L^p (Y, \nu))}

\newcommand{\MP}[2]{M_{#1}^{#2}}

\newcommand{\textdisp}{\displaystyle}

\renewcommand{\S}{\subset}
\renewcommand{\c}{\colon}

\newcommand{\I}{\infty}

\newcommand{\Lem}[1]{Lemma~\ref{#1}}

\title[Classification
 of $L^p$~AF algebras]{Classification of $L^p$~AF algebras}

\author{N.~Christopher Phillips}
\author{Maria Grazia Viola}

\address{Department of Mathematics, University of Oregon,
      Eugene OR 97403-1222, USA,
      and Department of Mathematics, University of Toronto,
      Room 6290, 40 St.\  George St., Toronto ON M5S 2E4, Canada.}

\email[]{ncp@darkwing.uoregon.edu}

\address{Lakehead University Orillia, Orillia ON L3V 0B9, Canada,
and Fields Institute, 222 College Street, Toronto, ON M5T 3J1, Canada}

\email[]{mviola@lakeheadu.ca}

\thanks{The first author was partially supported by the
  US National Science Foundation under
  Grants DMS-1101742 and DMS-1501144.
The second author was supported by a Natural Sciences and
Engineering Research Council Discovery Grant.}

\date{31~July 2017}

\begin{document}

\begin{abstract}
We define spatial $L^p$~AF algebras
for $p \in [1, \infty) \setminus \{ 2 \}$,
and prove the following analog
of the Elliott AF algebra classification theorem.
If $A$ and $B$ are spatial $L^p$~AF algebras,
then the following are equivalent:
\begin{itemize}
\item
$A$ and $B$ have isomorphic scaled preordered $K_0$-groups.
\item
$A \cong B$ as rings.
\item
$A \cong B$ (not necessarily isometrically)
as Banach algebras.
\item
$A$ is isometrically isomorphic to $B$
as Banach algebras.
\item
$A$ is completely isometrically isomorphic to $B$
as matrix normed Banach algebras.
\end{itemize}
As background,
we develop the theory of matrix normed $L^p$~operator algebras,
and show that there is a unique way to make
a spatial $L^p$~AF algebra into a
matrix normed $L^p$~operator algebra.
We also show that any countable scaled
Riesz group can be realized as the scaled preordered
$K_0$-group of a spatial $L^p$~AF algebra.
\end{abstract}

\maketitle

\section{Introduction}\label{Sec_Intro}

In a well known paper~\cite{Ell} of 1976, Elliott gave a
complete classification of approximately finite
dimensional (AF) C*-algebras.
He showed that two AF C*-algebras
$A_1$ and $A_2$ are isomorphic if and only if
their scaled preordered $K_0$-groups
$\big( K_0 (A_1), K_0 (A_1)_{+}, \Sigma (A_1) \big)$
and $\big( K_0 (A_2), K_0 (A_2)_{+}, \Sigma (A_2) \big)$
are isomorphic.
Moreover, the work of
Effros, Handelman, and Shen showed
(see~\cite{EE} and~\cite{EHS}) that any
countable scaled Riesz group $(G, G_{+},\Sigma)$
can be realized as the scaled preordered $K_0$-group
of an AF C*-algebra.

In a series of papers (see~\cite{PhLp1},
\cite{PhLp2}, \cite{PhLp4}, and~\cite{PhLp3}),
the first author introduced and studied
$L^p$ analogs of the uniformly hyperfinite
(UHF) algebras and $L^p$ analogs of the
Cuntz algebras.
One result of~\cite{PhLp2} is that
two spatial $L^p$~UHF algebras are isomorphic
if and only if they have the same supernatural
number.
This result is analogous to the result
of Glimm~\cite{Gli},
that two UHF C*-algebras are
isomorphic if and only if they have the
same supernatural number.
(This is a special case, done earlier,
of Elliott's AF classification theorem.)

It is therefore natural to ask if there
are $L^p$ analogs of AF~algebras
which can be
classified by their scaled preordered $K_0$~groups.
In this paper,
we show that the algebras that we call the spatial $L^p$~AF algebras
provide a positive answer to this question.
In Theorem~\ref{T_6324_Classification_Uniq},
we show that two
spatial $L^p$~AF algebras are completely isometrically isomorphic
(as matricial $L^p$~operator algebras) if and only if
their scaled ordered $K_0$ groups are isomorphic.
We further show that, as in the C*-algebra case, given
any scaled countable Riesz group
$(G, G_{+}, \Sigma)$, there exists a spatial
$L^p$~AF algebra $A$ such that the scaled
preordered $K_0$ group
$\big( K_0 (A), \, K_0 (A)_{+}, \, \Sigma (A) \big)$
is isomorphic to $(G, G_{+}, \Sigma)$. We also show that
spatial $L^p$~AF algebras have unique $L^p$~matrix
norms.

We don't list examples.
Theorem~\ref{T_6324_Classification_Exist}
and Theorem~\ref{T_6417_FullClass}
show that for each $p \in [1, \infty)$
there is a one to one correspondence between
isomorphism classes of AF C*-algebras
and spatial $L^p$~AF algebras,
so the examples are ``the same''.
Although we don't address this issue here,
constructions like the \ca{}
of a locally finite discrete abelian group,
which give AF C*-algebras,
give $L^p$~operator algebras
which are AF in some sense but
are not spatial $L^p$~AF algebras.

In a forthcoming paper we will prove that
the ideal structure of a spatial $L^p$~AF algebra
is determined by K-theory in the same way as for an AF \ca.
We will prove that,
like a \ca,
a spatial $L^p$~AF algebra is incompressible
in the
sense that any contractive \hm{} to some other Banach
algebra can be factored as a quotient map followed by
an isometric \hm.
(In particular, contractive injective
\hm{s} from spatial $L^p$~AF algebras are isometric.)
We will also study
the isometries and automorphisms of a spatial $L^p$~AF algebra.
The results will be quite different from what
happens with AF \ca{s}.

A spatial $L^p$~AF algebra is the direct limit of a direct system
of semisimple finite dimensional $L^p$~operator
algebras in which the connecting maps are contractive
homomorphisms having the property that the image
of the identity is a spatial partial isometry in
the sense of Definition~6.4 of~\cite{PhLp1}.
In the context of $L^p$~operator algebras, where in
general we do not require the homomorphisms between
$L^p$~operator algebras to be unital, that is
the best possible form for our maps.

To make sense of uniqueness of $L^p$~matrix norms
and completely isometric isomorphism,
we develop the basics of
the theory of matrix normed Banach algebras
and matricial $L^p$~operator algebras.

The arguments used for the classification of spatial
$L^p$~AF algebras are similar to the ones used for
the classification of AF \ca{s}.
However, to be able to carry out these arguments,
background material needs
to be developed.
Much of it is fairly elementary, and for this part the
novelty is putting it together in the right way. There
are several somewhat more substantial ingredients,
including a structure theorem for contractive
representations of $C (X)$ on an $L^p$~space
(Theorem~\ref{T_4116_CXLp}), the recognition
that, in connection with nonunital maps between
unital algebras, idempotents must be required to
be hermitian (contractivity is not good enough;
see Section~\ref{Sec_HIdemp}), and what to
require of approximate identities of idempotents
in order to get a unique suitable norm on the unitization
(see Proposition~\ref{P_5209_NormOnUnit}).
We also have to prove that the direct limit of
$L^p$~operator algebras is again an $L^p$~operator
algebra.

The paper is organized as follows.
In Section~\ref{Sec_Pre} we recall
$L^p$~operator algebras and give some preliminary
results on their representations on $L^p$-spaces.
In Section~\ref{Sec_MatLp} we introduce
matricial (matrix normed) $L^p$~operator algebras
and discuss their representations on $L^p$~spaces.
This material
is needed to define $L^p$~operator algebras that have
unique $L^p$ matrix norms, which we examine in
Section~\ref{Sec_UniqM}.
Most of the $L^p$~operator algebras in this article
have unique $L^p$ matrix norms, including the matrix algebra
$M_n^p$ and the algebra $C (X)$ for a compact
metric space~$X$.

Sections~\ref{Sec_DSum} and~\ref{Sec_DLim} deal with
direct sums and direct limits of (matricial) $L^p$~operator
algebras, while Section~\ref{Sec_HIdemp} contains
material on hermitian idempotents, including a
characterization of hermitian idempotents in an
$L^p$~operator algebra in terms of multiplication
operators.

In Section~\ref{Sec_ssfd} we introduce our building
blocks (the spatial semisimple finite dimensional
$L^p$~operator algebras), and the appropriate \hm{s}
between them, the spatial homomorphisms. We
characterize
spatial homomorphisms in terms of block diagonal
homomorphisms.
In Section~\ref{Sec_Init} we define spatial $L^p$~AF
algebras, show that every spatial $L^p$~AF
algebra is an $L^p$~operator algebra as in~\cite{PhLp1},
and that it has unique $L^p$ matrix norms.

Section~\ref{Sec_ClassN} contains our main result.
We give a complete classification of spatial $L^p$~AF
algebras using the scaled preordered $K_0$ group,
and show that,
as in the C*-algebra case, any countable scaled Riesz
group can be realized as the scaled preordered $K_0$ group
of a spatial $L^p$~AF algebra.

Shortly after posting this paper on the airXiv,
E.~Gardella informed us of his work with Lupini on the uniqueness
of the matricial norm structure for
$L^p$~analogs of groupoid C*-algebras.
(See \cite{GL}.)
Spatial $L^p$~AF algebras are
examples of such algebras;
see Subsection 7.2 in~\cite{GL}.
We were unaware of the work of Gardella
and Lupini while preparing this manuscript and we
refer the reader to their paper for a different proof
of the uniqueness of the $L^p$ matrix norms.

We use the following standard notation
throughout the paper.

\begin{ntn}\label{algebraboundedoperators}
If $E$ is a Banach space, then $L (E)$ denotes
the Banach algebra of all bounded linear operators
on~$E$, with the operator norm.
\end{ntn}

\begin{ntn}\label{N_7627_ResMeas}
If $\XBM$ is a \msp,
and $E \S X$ is measurable,
then $\mu |_E$ denotes the measure on~$E$
gotten by restricting $\mu$ to the $\sm$-algebra
of measurable subsets of~$E$.
\end{ntn}

We also recall that an idempotent in a ring is an element~$e$
satisfying $e^2 = e$.

\section{$L^p$~operator algebras}\label{Sec_Pre}

\indent
In this section we define $L^p$~operator algebras, and
state some of the standard results about $L^p$~operator
algebras and their representations.
These results are basic for the rest of the paper.

The following definitions are based on Definition~1.1
and Definition~1.17 of~\cite{PhLp3}.

\begin{dfn}\label{D_4310_LpOpAlg}
Let $p \in[1, \infty)$.
An $L^p$~operator algebra is a
Banach algebra
such that there exists a measure space
$(X, \mathcal{B}, \mu)$ and an isometric
isomorphism from $A$ to a norm closed
subalgebra of $L (L^p (X, \mu))$.
\end{dfn}

\begin{dfn}\label{D_7627_LpOpAlgConds}
Let $p \in[1, \infty)$.
\begin{enumerate}
\item\label{D_3914_LpRep_Rep}
A {\emph{representation}} of an $L^p$~operator
algebra~$A$ (on $L^p (Y, \nu)$) is a \ct{}
\hm{}  $\pi \colon A \to  L(L^p (Y, \nu))$
for some measure space $(Y, \mathcal{C}, \nu)$.
\item\label{D_3914_LpRep_Contr}
The \rpn{} $\pi$ is {\emph{contractive}}
if $\| \pi (a) \| \leq \| a \|$ for all $a \in A$,
and {\emph{isometric}}
if $\| \pi (a) \| = \| a \|$ for all $a \in A$.
\item\label{D_3914_LpRep_Spb}
We say that
the \rpn{}
$\pi \colon A \to L(L^p (Y, \nu))$
is {\emph{separable}} if $L^p (Y, \nu)$ is separable,
and that $A$ is {\emph{separably representable}}
if it has a separable isometric representation.
\item\label{D_3914_LpRep_Sft}
We say that $\pi$ is {\emph{$\sigma$-finite}}
if $\nu$ is \sft,
and that $A$ is {\emph{$\sigma$-finitely representable}}
if it has a $\sigma$-finite isometric representation.
\item\label{D_3914_LpRep_Ndg}
We say that $\pi$ is {\emph{nondegenerate}}
if
\[
\pi (A) L^p (Y, \nu) = \spn \big( \big\{ \pi (a) \xi \colon
     {\mbox{$a \in A$ and $\xi \in L^p (Y, \nu)$}} \big\} \big)
\]
is dense in~$L^p (Y, \nu)$.
We say that $A$ is {\emph{nondegenerately (separably)
representable}} if it has a nondegenerate (separable)
isometric representation, and {\emph{nondegenerately
$\sm$-finitely representable}} if it has a nondegenerate
$\sm$-finite isometric representation.
\end{enumerate}
\end{dfn}

The following fact about the restriction of an operator
looks obvious (and the proof is easy),
but it is the sort of statement
that should not be taken for granted
outside of the context of \ca{s}. The
condition $\| f \| = 1$ is necessary;
see Example~\ref{E_6807_NotNorm1}.

\begin{lem}\label{L_5125_NormCut}
Let $E$ be a Banach space,
let $a \in L (E)$,
and let $f \in L (E)$ be an idempotent with $\| f \| = 1$
such that $a f = a$.
Then $\| a |_{f E} \| = \| a \|$.
\end{lem}

\begin{proof}
It is obvious that $\| a |_{f E} \| \leq \| a \|$.
For the reverse inequality,
let $\ep > 0$,
choose $\xi \in E$ such that $\| \xi \| \leq 1$
and $\| a \xi \| > \| a \| - \ep$,
and set $\et = f \xi$.
Then $\et \in f E$,
$\| \et \| \leq 1$,
and $\| a \et \| = \| a \xi \| > \| a \| - \ep$.
\end{proof}

The main application of Lemma~\ref{L_5125_NormCut}
is the next result, which we will use
repeatedly in the following sections.

\begin{cor}\label{C_5125_CutAlg}
Let $A$ be a unital Banach algebra in which $\| 1 \| = 1$.
Let $E$ be a Banach space,
and let $\pi \colon A \to L (E)$
be a nonzero \rpn.
Set $F = \pi (1) E$.
Then there is a unital \rpn{} $\pi_0 \colon A \to L (F)$
such that $\pi_0 (a) \xi = \pi (a) \xi$
for all $a \in A$ and $\xi \in F$.
If $\| \pi (1) \| = 1$,
then $\pi_0$ is contractive
\ifo{} $\pi$ is contractive
and $\pi_0$ is isometric \ifo{} $\pi$ is isometric.
\end{cor}

\begin{proof}
The existence of $\pi_0$ follows from the
equation $\pi (1) \pi (a) \pi (1) = \pi (a)$
for all $a \in A$.
If $\pi$ is contractive then
$\| \pi (1) \| \leq 1$.
If $\| \pi (1) \| = 1$,
taking $f = \pi (1)$ in Lemma~\ref{L_5125_NormCut}
gives $\| \pi_0 (a) \| = \| \pi (a) \|$ for all $a \in A$.
\end{proof}

\begin{exa}\label{E_6807_NotNorm1}
Lemma~\ref{L_5125_NormCut} fails without
$\| f \| = 1$ and Corollary~\ref{C_5125_CutAlg}
fails without $\| 1 \| = 1$.
Take $f \in L (E)$ to
be any idempotent with $\| f \| > 1$.
For example, take $p \in (1, \infty)$,
take $E$ to be $\C^2$ with $\| \cdot \|_p$,
and take
$f = \left( \begin{smallmatrix}
  1     &  1        \\
  0     &  0
\end{smallmatrix} \right)$.
Then $\| f \| = 2^{1 - 1/p}$.
For Lemma~\ref{L_5125_NormCut} take $a = f$.
Then $a |_{f E}$ is the identity operator,
so $\| a |_{f E} \| = 1 < \| a \|$.
For Corollary~\ref{C_5125_CutAlg} take $A = \C f \subset L (E)$
with the operator norm
and take $\pi$ to be the identity representation.
Then $\pi$ is isometric but $\pi_0$ is not.
\end{exa}

\begin{prp}[Proposition~1.25 of~\cite{PhLp3}]\label{P_4Y19_SepImpSepRep}
Let $p \in [1, \I)$,
and let $A$ be a separable $L^p$~operator algebra.
Then $A$ is separably representable.
If $A$ is nondegenerately representable,
then $A$ is separably nondegenerately representable.
\end{prp}

Since we will only use separable $L^p$~operator
algebras in this paper, we need only deal with separable
$L^p$~spaces.
Lemma~\ref{L_6416_SepSFT} implies that we
can always assume that the measures are \sft.

\begin{lem}\label{L_6416_SepSFT}
Let $p \in [1, \infty)$.
Let $\XBM$ be a \msp{}
such that $L^p (X, \mu)$
is separable.
Then there exists a \sfm{} $\YCN$
such that $L^p (X, \mu)$
is isometrically isomorphic to $L^p (Y, \nu)$.
\end{lem}

\begin{proof}
See the Corollary to Theorem~3
in Section~15 of~\cite{Lcy}.
\end{proof}

The following result will be used often enough that we restate it here.

\begin{prp}\label{P_5204_ContrPj}
Let $p \in [1, \infty) \setminus \{ 2 \}$.
Let $\XBM$ be a measure space and let $e \in \LLp$
be an idempotent.
Then $\| e \| \leq 1$
\ifo{} there exists a \msp{} $\YCN$
and an isometric bijection from $L^p (Y, \nu)$
to the range of~$e$.
Moreover,
if $L^p (X, \mu)$ is separable,
then $\nu$ can be chosen to be \sft.
\end{prp}

\begin{proof}
This is part of Theorem~3 in Section~17 of~\cite{Lcy}.
\end{proof}

\begin{prp}\label{P_4Y19_Unital}
Let $A$ be a unital $L^p$~operator algebra in
which $\| 1 \| = 1$.
Then $A$ has an isometric
unital representation on an $L^p$~space.
If $A$ is separable then the $L^p$~space can be
chosen to be separable
and to be the $L^p$~space of a \sfm.
\end{prp}

\begin{proof}
Let $\XBM$ be a measure space such that
there is an isometric representation
$\rh \colon A \to \LLp$.
Then
$e = \rh (1)$ is an idempotent in $\LLp$
with $\| e \| = 1$.
Set $E = {\operatorname{ran}} (e)$.
Then $\rh$ induces an isometric unital \hm{}
$\rh_0 \colon A \to L (E)$ by Corollary~\ref{C_5125_CutAlg}.
By Proposition~\ref{P_5204_ContrPj}, there is a \msp{} $\YCN$
such that $E$ is isometrically isomorphic to $L^p (Y, \nu)$.
The first part of the conclusion follows.
For the second part,
if $A$ is separable then we may require that $L^p (X, \mu)$
be separable by Proposition~\ref{P_4Y19_SepImpSepRep}.
Then $E$ must be separable.
Proposition~\ref{P_5204_ContrPj}
implies that $E$ is isometrically isomorphic to
the $L^p$~space of a \sfm.
\end{proof}

\section{Matrix normed algebras and matricial $L^p$~operator
 algebras}\label{Sec_MatLp}

\indent
We will mostly work with ordinary $L^p$~operator algebras,
but for some results we will need the matrix normed version
introduced here (Definition~\ref{D_4Y19_MNLpOpAlg}).
We also need the analogs of
Proposition~\ref{P_4Y19_SepImpSepRep}
and Proposition~\ref{P_4Y19_Unital}
for matricial $L^p$~operator algebras;
see Proposition~\ref{P_4Y19_CISepImpCISepRep} and
Proposition~\ref{P_4Y19_CIUnital}.

Matrix normed spaces
(operator spaces of various kinds)
are well known,
but we have not seen a general definition of a matrix normed algebra.
We therefore give one here
(Definition~\ref{D_4Y19_MatNAlg}).
The conditions on the matrix norms
seem to be the minimal ``reasonable'' conditions.
Condition~(\ref{D_4Y19_MatNAlg_SbMat})
essentially says that submatrices have smaller norm.
We first describe our (fairly standard) notation for matrices.

\begin{ntn}\label{N_6808_Mn}
Let $n \in \N$.
Then $M_n$ denotes the algebra of $n \times n$ complex matrices
(without any specific norm being assumed).
For $j, k \in \{ 1, 2, \ldots, n \}$,
we let $e_{j, k}$ denote the corresponding standard
matrix unit of $M_n$.
For any complex algebra~$A$,
we identify the algebra $M_n (A)$ with $M_n \otimes A$
via
${\textdisp{(a_{j, k})_{1 \leq j, k \leq n}
 \mapsto \sum_{j, k = 1}^n e_{j, k} \otimes a_{j, k} }}$.
For $x \in M_n$ and $a \in M_n (A)$,
the products $x a$ and $a x$ are defined in the obvious way,
so that $x (y \otimes b) = x y \otimes b$
and $(y \otimes b) x = y x \otimes b$
for $y \in M_n$ and $b \in A$.
\end{ntn}

\begin{dfn}\label{D_4Y19_MatNAlg}
A {\emph{matrix normed algebra}}
is a complex algebra $A$ equipped with
algebra norms $\| \cdot \|_n$
on $M_n (A)$ for all $n \in \N$,
satisfying the following:
\begin{enumerate}
\item\label{D_4Y19_MatNAlg_SbMat}
For any $m, n \in \N$
with $m \leq n$,
any injective functions
\[
\sm, \ta \colon \{ 1, 2, \ldots, m \} \to \{ 1, 2, \ldots, n \},
\]
and any $a = (a_{j, k})_{1 \leq j, k \leq n} \in M_n (A)$,
we have
\[
\big\| (a_{\sm (j), \ta (k)})_{1 \leq j, k \leq m} \big\|_m
 \leq \| a \|_n.
\]
\item\label{D_4Y19_MatNAlg_PrdDiag}
For any $n \in \N$,
any $a \in M_n (A)$,
and any
$\ld_1, \ld_2, \ldots, \ld_n \in \C$,
if we set $s = \diag (\ld_1, \ld_2, \ldots, \ld_n) \in M_n$,
then
\[
\| a s \|_n, \,
\| s a \|_n
 \leq \max \big( |\ld_1|, |\ld_2|, \ldots, |\ld_n| \big) \| a \|_n.
\]
\item\label{D_4Y19_MatNAlg_DSum}
For any $m, n \in \N$,
$a \in M_m (A)$,
and $b \in M_n (A)$,
we have
\[
\| \diag (a, b) \|_{m + n} = \max (\| a \|_m, \, \| b \|_n ).
\]
\end{enumerate}
We abbreviate $\| \cdot \|_1$ to $\| \cdot \|$.
If $A$ is complete in $\| \cdot \|$, we call $A$ a
{\emph{matrix normed Banach algebra}}.
\end{dfn}

\begin{rmk}\label{R_5205_Redundant}
In Definition~\ref{D_4Y19_MatNAlg},
if $A$ is unital and $\| 1 \| = 1$, or even if $A$ has
an approximate identity which is bounded by~$1$,
condition~(\ref{D_4Y19_MatNAlg_PrdDiag})
follows from condition~(\ref{D_4Y19_MatNAlg_DSum})
and submultiplicativity of $\| \cdot \|_n$.
\end{rmk}

\begin{rmk}\label{R_6408_IneqRed}
In Definition~\ref{D_4Y19_MatNAlg},
the inequality
$\| \diag (a, b) \|_{m + n} \geq \max (\| a \|_m, \, \| b \|_n )$
in condition~(\ref{D_4Y19_MatNAlg_DSum})
follows from condition~(\ref{D_4Y19_MatNAlg_SbMat}).
\end{rmk}

\begin{lem}\label{L_6408_matrixnorms}
Let $A$ be a matrix normed algebra,
let $n \in \N$,
and let $a = (a_{j, k})_{1 \leq j, k \leq n} \in M_n (A)$.
Then
\[
\max_{1 \leq j, k \leq n} \| a_{j, k} \|
\leq \| a \|_n
\leq \sum_{j, k = 1}^n \| a_{j, k} \|.
\]
\end{lem}

\begin{proof}
The first inequality follows from
Definition \ref{D_4Y19_MatNAlg}(\ref{D_4Y19_MatNAlg_SbMat}).
We prove the second inequality.
First, two applications of condition~(\ref{D_4Y19_MatNAlg_SbMat}),
taking $\sm$ and $\ta$ there to be permutations,
show that permuting the rows and also permuting the columns
of a matrix does not change its norm.
Using this fact at the second step
and condition~(\ref{D_4Y19_MatNAlg_DSum}) at the first step,
we get
\[
\|a_{j, k}\|
 = \|\diag (a_{j, k}, 0, \ldots, 0) \|_n
 = \| e_{j, j} a e_{k, k} \|_n.
\]
Apply this to the relation
$a = {\textdisp{ \sum_{j, k = 1}^m e_{j, j} a e_{k, k} }}$
to complete the proof.
\end{proof}

\begin{cor}\label{C_6808_MnComplete}
Let $A$ be a matrix normed Banach algebra.
Then $M_n (A)$ is complete for all $n \in \N$.
\end{cor}

\begin{proof}
This is immediate from Lemma~\ref{L_6408_matrixnorms}.
\end{proof}

For clarity, we state the standard definitions related to
completely bounded maps.

\begin{dfn}\label{D_4Y19_CB}
Let $A$ and $B$ be matrix normed algebras,
and let $\ph \colon A \to B$ be a linear map.
For $n \in \N$,
write $\ph^{(n)}$ or $\id_{M_n} \otimes \ph$
for the map $M_n (A) \to M_n (B)$
determined by
$(a_{j, k})_{1 \leq j, k \leq n} \mapsto
  ( \ph (a_{j, k}))_{1 \leq j, k \leq n}$.
Then:
\begin{enumerate}
\item\label{D_4Y19_CB_CB}
We set
${\textdisp{\| \ph \|_{\mathrm{cb}}
 = \sup_{n \in \N} \| \ph^{(n)} \| }}$.
If $\| \ph \|_{\mathrm{cb}} < \I$,
we say that $\ph$ is {\emph{completely bounded}}.
\item\label{D_4Y19_CB_CCntr}
We say that $\ph$ is {\emph{completely contractive}}
if $\| \ph \|_{\mathrm{cb}} \leq 1$.
\item\label{D_4Y19_CB_CIso}
We say that $\ph$ is {\emph{completely isometric}}
if $\ph^{(n)}$ is isometric (not necessarily surjective)
for all $n \in \N$.
\item\label{D_4Y19_CB_CII}
We say that $\ph$ is a {\emph{completely isometric isomorphism}}
if $\ph$ is completely isometric and bijective.
\end{enumerate}
\end{dfn}

\begin{dfn}\label{D_6408_Ops}
Let $A$ be a matrix normed algebra.
\begin{enumerate}
\item\label{D_6408_Ops_Sub}
Let $B$ be a subalgebra of~$A$.
For $n \in \N$,
we define the norm $\| \cdot \|_n$ on $M_n (B)$
to be the restriction to $M_n (B)$
of the given norm on $M_n (A)$.
\item\label{D_6408_Ops_Quot}
Let $J \subset A$ be a closed ideal.
For $n \in \N$,
we define the norm $\| \cdot \|_n$ on $M_n (A / J)$
to be the quotient norm coming from
the obvious identification of $M_n (A / J)$
with $M_n (A) / M_n (J)$.
\end{enumerate}
\end{dfn}

\begin{lem}\label{L_6408_OpsAre}
Let $A$ be a matrix normed algebra.
\begin{enumerate}
\item\label{L_6408_OpsAre_Sub}
Let $B \subset A$ be a subalgebra.
Then the norms in Definition \ref{D_6408_Ops}(\ref{D_6408_Ops_Sub})
make $B$ a matrix normed algebra,
and the inclusion map is completely isometric.
\item\label{L_6408_OpsAre_Quot}
Let $J \subset A$ be a closed ideal.
Then the norms in Definition \ref{D_6408_Ops}(\ref{D_6408_Ops_Quot})
make $A / J$ a matrix normed algebra,
and the quotient map is completely contractive.
\end{enumerate}
\end{lem}

\begin{proof}
Part~(\ref{L_6408_OpsAre_Sub}) is immediate.

We prove part~(\ref{L_6408_OpsAre_Quot}).
Let $\pi \colon A \to A / J$ be the quotient map.
Complete contractivity of $\pi$ is immediate.
For Definition \ref{D_4Y19_MatNAlg}(\ref{D_4Y19_MatNAlg_SbMat}),
let $m, n \in \N$
with $m \leq n$,
let
\[
\sm, \ta \colon \{ 1, 2, \ldots, m \} \to \{ 1, 2, \ldots, n \}
\]
be injective functions,
and let $x = (x_{j, k})_{1 \leq j, k \leq n} \in M_n (A / J)$.
Let $\ep > 0$.
Choose $a = (a_{j, k})_{1 \leq j, k \leq n} \in M_n (A)$
such that $\pi^{(n)} (a) = x$ and $\| a \|_n < \| x \|_n + \ep$.
Since $\pi^{(m)}$ is contractive, we have
\[
\big\| (x_{\sm (j), \ta (k)})_{1 \leq j, k \leq m} \big\|_m
 \leq \big\| (a_{\sm (j), \ta (k)})_{1 \leq j, k \leq m} \big\|_m
 \leq \| a \|_n
 < \| x \|_n + \ep.
\]

The proofs of
Definition \ref{D_4Y19_MatNAlg}(\ref{D_4Y19_MatNAlg_PrdDiag})
and the inequality
\[
\| \diag (a, b) \|_{m + n} \leq \max (\| a \|_m, \, \| b \|_n )
\]
in Definition~\ref{D_4Y19_MatNAlg}(\ref{D_4Y19_MatNAlg_DSum})
are similar.
Equality in Definition \ref{D_4Y19_MatNAlg}(\ref{D_4Y19_MatNAlg_DSum})
now follows from Remark~\ref{R_6408_IneqRed}.
\end{proof}

\begin{dfn}\label{D_5205_CxPerm}
Let $n \in \N$.
A matrix $s \in M_n$ is a {\emph{permutation matrix}}
if there exists a bijection
$\sm \colon \{ 1, 2, \ldots, n \} \to \{ 1, 2, \ldots, n \}$
such that ${\textdisp{s = \sum_{j = 1}^n e_{ \sm (j), \, j} }}$.
The matrix $s$ is a {\emph{complex permutation matrix}}
if there exist a bijection
$\sm \colon \{ 1, 2, \ldots, n \} \to \{ 1, 2, \ldots, n \}$
and
$\ld_1, \ld_2, \ldots, \ld_n
 \in S^1 = \{z \in \C \colon |z| = 1 \}$
such that
${\textdisp{s = \sum_{j = 1}^n \ld_j e_{ \sm (j), \, j} }}$.
\end{dfn}

The complex permutation matrices form a group.

\begin{lem}\label{L_5205_MbyCxP}
Let $A$ be a matrix normed algebra and fix $n \in \N$.
Let $a \in M_n (A)$,
and let $s \in M_n$ be a complex permutation matrix.
Interpret
$a s$ and $s a$
as in Notation~\ref{N_6808_Mn}.
Then $\| a s \|_n = \| s a \|_n = \| a \|_n$.
\end{lem}

\begin{proof}
Since $s^{-1}$ is also a complex permutation matrix,
it suffices to prove that
$\| a s \|_n \leq \| a \|_n$ and $\| s a \|_n \leq \| a \|_n$.
Since a complex permutation matrix is a product of a
permutation matrix and a diagonal matrix with diagonal
entries in~$S^1$, it suffices to prove these inequalities
for these two kinds of matrices separately.
For the first kind,
apply Definition~\ref{D_4Y19_MatNAlg}(\ref{D_4Y19_MatNAlg_SbMat}).
For the second kind,
apply Definition~\ref{D_4Y19_MatNAlg}(\ref{D_4Y19_MatNAlg_PrdDiag}).
\end{proof}

\begin{dfn}\label{D_5205_MmMnMmn}
Let $m, n \in \N$ and let
\[
\sm \colon \{ 1, 2, \ldots, m \} \times \{ 1, 2, \ldots, n \}
  \to \{ 1, 2, \ldots, m n \}
\]
be a bijection.
We let
$\te_{\sm} \colon M_m \otimes M_n \to M_{m n}$
be the unique algebra isomorphism such that for
$i, j \in  \{ 1, 2, \ldots, m \}$ and $k, l \in \{ 1, 2, \ldots, n \}$,
we have
$\te_{\sm} (e_{i, j} \otimes e_{k, l})= e_{\sm (i, k), \, \sm (j, l)}$.

The {\emph{standard choice}} of bijection
is the one given by
$\sm (j, l) = j + m (l - 1)$
for $j = 1, 2, \ldots, m$ and $l = 1, 2, \ldots, n$.
\end{dfn}

\begin{dfn}\label{D_5205_MatNormMm}
Let $A$ be a matrix normed algebra
and let $m \in \N$.
We define matrix norms on $M_m (A)$ as follows.
For $n \in \N$,
choose some bijection
\[
\sm_n \colon \{ 1, 2, \ldots, n \} \times \{ 1, 2, \ldots, m \}
  \to \{ 1, 2, \ldots, n m \},
\]
and use it (and Notation~\ref{N_6808_Mn}) to
get the isomorphism
$\te_{\sm_n} \otimes \id_A \colon M_n (M_m (A)) \to M_{n m} (A)$.
For $a \in M_n (M_m (A))$,
we then use the matrix norms on~$A$
to define
\[
\| a \|_n = \| (\te_{\sm_n} \otimes \id_A) (a) \|_{n m}.
\]
\end{dfn}

\begin{lem}\label{L_5205_MNormIndep}
In Definition~\ref{D_5205_MatNormMm},
the matrix norms are independent of the choice of
$(\sm_n)_{n \in \N}$, and make $M_m (A)$ a matrix
normed algebra.
\end{lem}

\begin{proof}
Independence of $(\sm_n)_{n \in \N}$ follows from
Lemma~\ref{L_5205_MbyCxP},
and the fact that one gets a matrix normed algebra
follows easily from Definition~\ref{D_4Y19_MatNAlg}.
\end{proof}

\begin{dfn}\label{D_4Y19_LpMN}
Let $\XBM$ be a \msp,
and let $A \subset \LLp$ be a closed subalgebra. We
equip $A$ with the matrix norms coming from the
identification of $M_n (A)$ with a closed subalgebra of
$L \big( L^p ( \{ 1, 2, \ldots, n \} \times X,
 \, \nu \times \mu) \big)$,
in which $\nu$ is counting measure on $\{ 1, 2, \ldots, n \}$.
\end{dfn}

\begin{ntn}\label{N_5124_FDP}
For any set~$S$ and any $p\in [1,\infty]$, we give $l^p (S)$ the usual
meaning (using counting measure on~$S$),
and we set (as usual) $l^p = l^p (\N)$.
For $d \in \N$ we let $l^p_d = l^p \big( \{1, 2, \ldots, d \} \big)$.
We further let $\MP{d}{p} = L \big( l_d^p \big)$
with the usual operator norm,
and we algebraically identify $\MP{d}{p}$ with $M_d$
in the standard way.
For $a \in M_d^p$, we write the norm as $\| a \|_p$.
We equip $\MP{d}{p}$ with the matrix norms
as in Definition~\ref{D_4Y19_LpMN}.
\end{ntn}

\begin{lem}\label{L_5205_MatNLp}
Let $\XBM$ be a \msp,
and let $A \subset \LLp$ be a closed subalgebra.
Then $A$ is a matrix normed algebra with the matrix
norms of Definition~\ref{D_4Y19_LpMN}.
Moreover for $n \in \N$,
$a \in M_n (A)$, and $x \in M_n$,
with products as in Notation~\ref{N_6808_Mn},
we have $\| a x \|_n, \, \| x a \|_n \leq \| x \|_p \| a \|_n$.
Furthermore,
for $m \in \N$
the matrix norms on $M_m (A)$ from
Definition~\ref{D_5205_MatNormMm}
agree with those gotten from
the obvious inclusion
\[
M_m (A)
\to L \big( L^p ( \{ 1, 2, \ldots, m \} \times X,
 \, \nu \times \mu) \big),
\]
in which $\nu$ is counting measure on $\{ 1, 2, \ldots, m \}$.
\end{lem}

\begin{proof}
All parts are easy.
\end{proof}

\begin{dfn}\label{D_4Y19_MNLpOpAlg}
Let $p \in [1, \infty)$.
A matricial $L^p$~operator algebra is a matrix normed
Banach algebra~$A$ such that there exists a measure
space $(X, \mathcal{B}, \mu)$ and a completely isometric
isomorphism from $A$ to a norm closed subalgebra of
$L (L^p (X, \mu))$.

Using the terminology from Definition~\ref{D_7627_LpOpAlgConds},
we say that a matricial $L^p$~operator algebra $A$ is
{\emph{separably representable} }if it has a separable
completely isometric representation.
We say that $A$ is
{\emph{$\sigma$-finitely representable}} if it has a
$\sigma$-finite completely isometric representation.
We say that $A$ is {\emph{nondegenerately (separably)
representable}} if it has a nondegenerate (separable)
completely isometric representation, and
{\emph{nondegenerately $\sm$-finitely representable}}
if it has a nondegenerate $\sm$-finite completely
isometric representation.
\end{dfn}

\begin{prp}\label{P_4Y19_CISepImpCISepRep}
Let $p \in [1, \I)$, and let $A$ be a separable
matricial $L^p$~operator algebra.
Then $A$ is
separably representable.
If $A$ is nondegenerately
representable, then $A$ is separably nondegenerately
representable.
\end{prp}

\begin{proof}
For $n \in \N$
let $\nu_n$ be counting measure on $\{ 1, 2, \ldots, n \}$.
Let $S$ be a countable dense subset of~$A$,
and for $n \in \N$ define
\[
S_n = \big\{ b \in M_n (A) \colon
 {\mbox{$b_{j, k} \in S$ for $j, k = 1, 2, \ldots, n$}} \big\},
\]
which is a countable dense subset of~$M_n (A)$.

By hypothesis, there exist a measure space
$(X, \mathcal{B}, \mu)$ and a completely isometric
representation $\rho \colon A \to L(L^p (X, \mu))$, which
we can take to be
nondegenerate when $A$ is nondegenerately representable.
For any $m, n \in \N$ and $b \in S_n$,
choose
\[
\xi_{n, b, m}
 = \big( \xi_{n, b, m}^{(j)} \big)_{1 \leq j \leq n}
 \in L^p \big( \{ 1, 2, \ldots, n \} \times X, \, \nu_n \times \mu \big)
\]
such that
\[
\| \xi_{n, b, m} \|_p = 1
\andeqn
\big\| (\id_{M_n} \otimes \rho) (b) \xi_{n, b, m} \big\|
 > \| b \| - \frac{1}{m}.
\]
By the
argument used in the proof of Proposition~1.25 of~\cite{PhLp3}
there exists a separable closed sublattice $F_{n, b, m}$ of
$L^p (X, \mu)$ containing
$\xi_{n, b, m}^{(1)}, \, \xi_{n, b, m}^{(2)}, \, \ldots,
  \, \xi_{n, b, m}^{(2)}$
and such that
$\rho (A) F_{n, b, m} \subset F_{n, b, m}$.
Moreover, $F_{n, b, m}$ is
isomorphic to $L^p (Y_{n, b, m}, \nu_{n, b, m})$ for some measure
space $(Y_{n, b, m}, \nu_{n, b, m})$.
Furthermore, if $\rho$ is
nondegenerate then $F_{n, b, m}$ can be chosen to satisfy
$\overline{\spn} (\rho (A) F_{n, b, m}) = F_{n, b, m}$.
The map defined by $\pi_{n, b, m} (a) = \rho (a)|_{F_{n, b, m}}$ is a
completely contractive representation of $A$ on a separable
$L^p$-space, which is nondegenerate if $\rho$
is nondegenerate.
Since $F_{n, b, m}$ contains
$\xi_{n, b, m}^{(1)}, \, \xi_{n, b, m}^{(2)}, \, \ldots,
  \, \xi_{n, b, m}^{(2)}$,
we get
$\| (\id_{M_n} \otimes \pi_{n, b, m}) (b) \| > \| b \| - \frac{1}{m}$
for every $m \in \N$.
Now let $\pi$ be the $L^p$ direct sum
of the representations
$\pi_{n, b, m}$ for $m, n \in \N$ and $b \in S_n$,
as in Definition 1.23 of~\cite{PhLp3}.
Then $\pi$ is a completely
contractive representation on a separable $L^p$~space.
We have $\| (\id_{M_n} \otimes \pi) (b) \| = \| b \|$ for all
$n \in \N$ and $b \in S_n$,
so density of $S_n$ in $M_n (A)$
implies that
$\pi$ is completely isometric.
Moreover, by Lemma 1.24
in~\cite{PhLp3}, $\pi$ is nondegenerate if
$\rho$ is nondegenerate.
\end{proof}

\begin{prp}\label{P_4Y19_CIUnital}
Let $p \in [1, \I)$,
and let $A$ be a unital matricial $L^p$~operator algebra
in which $\| 1 \| = 1$.
Then $A$ has a completely isometric unital representation
on an $L^p$~space.
If $A$ is separable then the $L^p$~space can be chosen to
be separable and to come from a \sfm.
\end{prp}

\begin{proof}
Let $(X, \mathcal{B}, \mu)$ be a measure space such that
there is completely isometric representation
$\rho_0 \colon A \to L(L^p (X, \mu))$.
Then $e = \rho_0 (1)$ is an
idempotent in $L(L^p (X, \mu))$, and $\|e \| = 1$.
By Proposition~\ref{P_5204_ContrPj} there exists a measure
space $(Y, \mathcal{C}, \nu)$ such that $\operatorname{ran} (e)$ is
isometrically isomorphic to $L^p (Y, \nu)$.
Thus, $\rho_0$
gives a completely isometric unital homomorphism
$\rho \colon A \to L(L^p (Y, \nu))$.
Moreover, if $A$ is
separable, then $L^p (X, \mu)$ can be chosen to be
separable, which implies that $\operatorname{ran} (e)$
is also separable.
To get $\sm$-finiteness,
use Lemma~\ref{L_6416_SepSFT}.
\end{proof}

\section{Unique matrix norms}\label{Sec_UniqM}

\indent
We consider uniqueness of matrix norms
on $L^p$~operator algebras.
Most of the $L^p$~operator algebras we deal with will
have unique $L^p$~operator matrix norms,
in the sense of Definition~\ref{D_4Y23_Uniq} below.
The basic examples are $M_d^p$ and $C (X)$.
We will show in Corollary~\ref{C_6416_GenUniqMN} below
that all spatial $L^p$~AF
algebras have unique $L^p$~operator matrix norms.
The proof that $C (X)$
has unique $L^p$~operator matrix norms uses
a structure theorem
(Theorem~\ref{T_4116_CXLp})
for contractive unital representations of
$C (X)$ on $L^p$~spaces,
which also plays a key role later.
To avoid technical issues,
we restrict our discussion to the separable case.

\begin{dfn}\label{D_4Y23_Uniq}
Let $p \in [1, \infty) \setminus \{ 2 \}$.
Let $A$ be a separable $L^p$~operator algebra.
We say that $A$ has
{\emph{unique $L^p$~operator matrix norms}}
if whenever $\XBM$ and $\YCN$
are \sfm{s}
such that $L^p (X, \mu)$ and $L^p (Y, \nu)$ are
separable, $\pi \colon A \to \LLp$
and $\sm \colon A \to L (L^p (Y, \nu))$
are isometric representations,
and $\pi (A)$ and $\sm (A)$
are given the matrix normed structures of
Definition~\ref{D_4Y19_LpMN},
then $\sm \circ \pi^{-1} \colon \pi (A) \to \sm (A)$
is completely isometric.
\end{dfn}

When $A$ is unital and $\| 1 \| = 1$,
in Definition~\ref{D_4Y23_Uniq} we need only
consider unital isometric representations.

\begin{lem}\label{L_5125_UnitUniq}
Let $p \in [1, \infty) \setminus \{ 2 \}$.
Let $A$ be a unital separable $L^p$~operator algebra
in which $\| 1 \| = 1$.
Assume that whenever $\XBM$ and $\YCN$
are \sfm{s}
such that $L^p (X, \mu)$ and $L^p (Y, \nu)$ are separable,
$\pi \colon A \to \LLp$
and $\sm \colon A \to L (L^p (Y, \nu))$
are unital isometric representations,
and $\pi (A)$ and $\sm (A)$
are given the matrix normed structures of
Definition~\ref{D_4Y19_LpMN},
then $\sm \circ \pi^{-1} \colon \pi (A) \to \sm (A)$
is completely isometric.
It follows that $A$ has unique $L^p$~operator matrix norms.
\end{lem}

\begin{proof}
Let $\XBM$ and $\YCN$
be \sfm{s}
such that $L^p (X, \mu)$ and $L^p (Y, \nu)$ are separable,
and let $\pi \colon A \to \LLp$
and $\sm \colon A \to L (L^p (Y, \nu))$
be isometric representations.

The operator $e = \pi (1)$ is an idempotent in $\LLp$
with $\| e \| = 1$.
Set $E = {\operatorname{ran}} (e)$.
Then $\pi$ induces a unital \hm{}
$\pi_0 \colon A \to L (E)$,
which is isometric by Corollary~\ref{C_5125_CutAlg}.
By Proposition~\ref{P_5204_ContrPj},
there is a \msp{} $(X_0, {\mathcal{B}}_0, \mu_0)$
such that $E$ is isometrically isomorphic to $L^p (X_0, \mu_0)$.
Since $E$ is separable,
we may require that $\mu_0$ be \sft.
Similarly,
${\operatorname{ran}} ( \sm (1) )$ is
isometrically isomorphic to a separable $L^p$~space
$L^p (Y_0, \nu_0)$ in which $\nu_0$ is \sft,
and $\sm$ induces an isometric unital \hm{}
$\sm_0 \colon A \to L ( L^p (Y_0, \nu_0) )$.
In particular,
$\sm_0 \circ \pi_0^{-1} \colon \pi_0 (A) \to \sm_0 (A)$
is isometric.

Let $n \in \N$.
We take the norms on
$M_n (\pi (A))$, $M_n (\pi_0 (A))$, $M_n (\sm (A))$,
and $M_n (\sm_0 (A))$
to be as in Definition~\ref{D_4Y19_LpMN}.
Define
\[
\rh_0 = \sm_0 \circ \pi_0^{-1} \colon
 \pi_0 (A) \to \sm_0 (A) \subset L ( L^p (Y_0, \nu_0) )
\]
and
\[
\rh = \sm \circ \pi_0^{-1} \colon
 \pi_0 (A) \to \sm (A) \subset L ( L^p (Y, \nu) ).
\]
Then
\[
(\id_{M_n} \otimes \rh) (1)
 = 1_{M_n} \otimes \sm (1)
 \in L \big( l_n^p \otimes_p L^p (Y, \nu) \big),
\]
so $\| (\id_{M_n} \otimes \rh) (1) \| = 1$.
The hypothesis implies that
$\id_{M_n} \otimes \rh_0$ is isometric.
So Corollary~\ref{C_5125_CutAlg} implies that
$\id_{M_n} \otimes \rh$ is isometric.

Similarly,
the map
$\id_{M_n} \otimes (\pi \circ \pi_0^{-1}) \colon
  M_n (\pi_0 (A)) \to M_n (\pi (A))$
is isometric.
Therefore
$\id_{M_n} \otimes (\sm \circ \pi^{-1}) \colon
 M_n (\pi (A)) \to M_n (\sm (A))$ is isometric.
This completes the proof.
\end{proof}

\begin{prp}\label{P_5124_MnIncmp}
Let $p \in [1, \infty) \setminus \{ 2 \}$.
Let $\MP{d}{p}$ be as in Notation~\ref{N_5124_FDP}.
Then every nonzero contractive
unital representation of $\MP{d}{p}$
on a separable $L^p$~space is completely isometric.
\end{prp}

\begin{proof}
Let $\XBM$ be a measure space such that
$L^p (X, \mu)$ is separable,
and let $\rh \colon \MP{d}{p} \to \LLp$
be a contractive unital representation.
By Lemma~\ref{L_6416_SepSFT},
we can assume that $\XBM$ is \sft.
Theorem~7.2 of~\cite{PhLp1}
provides a \sfm\  $(Z, {\mathcal{C}}, \ld)$
and a bijective isometry
\[
u \colon l_d^p \otimes L^p (Z, \ld) \to L^p (X, \mu)
\]
such that
for all $a \in M_d^p$ we have $\rh (a) = u (a \otimes 1) u^{-1}$.
For $n \in \N$,
it is easy to see that
\[
1_{M_n} \otimes u \colon l_n^p \otimes l_d^p \otimes L^p (Z, \ld)
 \to l_n^p \otimes L^p (X, \mu)
\]
is a bijective isometry such that
\[
(1_{M_n} \otimes \rh) (b)
 = (1_{M_n} \otimes u) (b \otimes 1) (1_{M_n} \otimes u)^{-1}
\]
for all $b \in M_n (\MP{d}{p})$.
It is now immediate that $\rh_n$ is isometric.
\end{proof}

\begin{cor}\label{P_5124_MnUniq}
Let $p \in [1, \infty) \setminus \{ 2 \}$.
The algebra $\MP{d}{p}$ of Notation~\ref{N_5124_FDP}
has unique $L^p$~operator matrix norms.
\end{cor}

\begin{proof}
Combine Lemma~\ref{L_5125_UnitUniq}
and Proposition~\ref{P_5124_MnIncmp}.
\end{proof}

Next, we give a structure theorem for any contractive unital
representation of $C (X)$ on an $L^p$~space.

\begin{thm}\label{T_4116_CXLp}
Let $p \in [1, \I) \setminus \{ 2 \}$.
Let $X$ be a compact metrizable space,
let $\YCN$ be a \sfm,
and let $\pi \colon C (X) \to \LLpy$
be a contractive unital \hm.
Let $\mu \colon L^{\infty} (Y, \nu) \to \LLpy$
be the representation of $L^{\infty} (Y, \nu)$
on $L^p (Y, \nu)$ by multiplication operators.
Then there exists a unital \ca{} \hm{}
$\ph \colon C (X) \to L^{\infty} (Y, \nu)$
such that $\pi = \mu \circ \ph$.
\end{thm}

\begin{proof}
We claim that the range of $\pi$ is contained in the range
of~$\mu$.
It suffices to prove that if $f \in C (X)$ is real
valued and satisfies $\| f \| < \pi$,
then $\pi (f)$ is in the range of~$\mu$.
Let $f$ be such a function.
For $\ld \in \R$,
the function $w_{\ld} = \exp (i \ld f)$ is invertible in $C (X)$
and satisfies $\| w_{\ld} \| = \| w_{\ld}^{- 1} \| = 1$.
Therefore $\pi (w_{\ld})$ is a bijective isometry
in $\LLpy$.
By Lemma 6.16 of~\cite{PhLp1},
the operator $\pi (w_{\ld})$ is a spatial isometry
in the sense of Definition~6.4 of~\cite{PhLp1}.
In particular,
it has a spatial system $(E_{\ld}, F_{\ld}, S_{\ld}, g_{\ld})$
as there.
By Lemma 6.22 of~\cite{PhLp1},
we have $S_{\ld} = S_0$ for all $\ld \in \R$.
Now $\pi (w_0) = 1$,
so, in the notation of
Definition~6.3 of~\cite{PhLp1} and Definition~5.4 of~\cite{PhLp1},
the operator $\pi (w_0)$ has the spatial system
$(Y, Y, \id_{ {\mathcal{C}} / {\mathcal{N}} ( \nu)}, 1)$.
The uniqueness statement in Lemma 6.6 of~\cite{PhLp1}
now implies that $S_{\ld} = \id_{ {\mathcal{C}} / {\mathcal{N}} ( \nu)}$
for all $\ld \in \R$.
Therefore $\pi (w_{\ld})$ is a multiplication operator;
in fact, $\pi (w_{\ld}) = \mu (g_{\ld})$ for all $\ld \in \R$.
Now let $\log$ be the holomorphic branch which is real on $(0, \I)$
and defined on $\C \setminus ( - \I, \, 0]$.
We have
\[
\spec (g_1)
 = \spec ( \pi (w_1))
 \subset \spec (w_1)
 \subset \C \setminus ( - \I, \, 0]
\]
and
\[
\pi (f)
 = \pi ( - i \log (w_1))
 = - i \log ( \pi (w_1))
 = - i \log ( \mu (g_1))
 = \mu ( - i \log (g_1)).
\]
Thus $\pi (f)$ is in the range of~$\mu$,
as claimed.

It follows that there is a contractive \hm{}
$\ph \colon C (X) \to L^{\infty} (Y, \nu)$
such that $\pi = \mu \circ \ph$.
Obviously $\ph$ is unital.
It follows from Proposition A.5.8 of~\cite{BM}
that $\ph$ is a \ca{} \hm.
\end{proof}

We don't need the following proposition,
but it is an interesting result
which follows from the machinery we have developed.

\begin{prp}\label{C_5124_CXUniq}
Let $X$ be a compact metrizable space,
and let $p \in [1, \I) \setminus \{ 2 \}$.
Then $C (X)$ has unique $L^p$~operator matrix norms.
They are given as follows.
Let $a \in M_n (C (X))$.
Interpret $a$ as a \cfn{} $a \colon X \to M_n$.
Equip $M_n = M_n^p$
with the norm $\| \cdot \|_p$ from Notation~\ref{N_5124_FDP}.
Then ${\textdisp{\| a \|_n = \sup_{x \in X} \| a (x) \|_p }}$.
\end{prp}

\begin{proof}
Choose a \sft{} Borel measure $\mu$ on $X$
such that $\mu (U) > 0$ for every nonempty
open set~$U \subset X$.
Represent $C (X)$ on
$L^p (X, \mu)$ as multiplication operators.
It is then easy to check that the matrix norms
from Definition~\ref{D_4Y19_LpMN}
are equal to the matrix norms in the statement
of the proposition.

In view of Lemma~\ref{L_5125_UnitUniq},
it remains to show that if $\YCN$ is a \sfm{}
such that $\LLpy$ is separable,
and $\pi \colon C (X) \to \LLpy$
is an isometric unital \hm,
then $\pi$ is completely isometric.
Let $\rh \colon L^{\infty} (Y, \nu) \to \LLpy$
be the representation given by multiplication operators.
Then $\rh$ is isometric,
so we can identify $L^{\infty} (Y, \nu)$
with its image under~$\rh$,
and thus make $L^{\infty} (Y, \nu)$
a matricial $L^p$~operator algebra
using the matrix norms of Definition~\ref{D_4Y19_LpMN}.
For $n \in \N$,
identify $M_n ( L^{\infty} (Y, \nu) )$ with
the algebra of $L^{\infty}$ functions from $Y$ to~$M_n$.
It is easy to check that the norms on $M_n ( L^{\infty} (Y, \nu) )$
are given by
${\textdisp{ \| a \|_n
 = {\operatorname{ess}} \, \sup_{y \in Y} \| a (y) \|_p}}$.

Now let $Z$ be the maximal ideal space of $L^{\infty} (Y, \nu)$,
and let $\gm \colon L^{\infty} (Y, \nu) \to C (Z)$
be the Gelfand transform,
which is an isomorphism.
Define matrix norms on $C (Z)$ in the same way
as on $C (X)$ in the statement of the proposition.
For every $f \in L^{\infty} (Y, \nu)$,
the essential range of $f$ is the range of $\gm (f)$.
It follows that for every $a \in M_n ( L^{\infty} (Y, \nu) )$,
the essential range of $a$
is equal to the range of
$(\id_{M_n} \otimes \gm) (a) \in M_n (C (Z)) = C (Z, M_n)$.
Therefore $\gm$ is completely isometric.

Apply Theorem~\ref{T_4116_CXLp} to~$\pi$.
We get a unital \ca{} \hm{}
$\ph \colon C (X) \to L^{\infty} (Y, \nu)$
such that $\pi = \rh \circ \ph$.
Moreover, $\ph$ is injective.
There is a \cfn{} $h \colon Z \to X$ such that
$(\gm \circ \ph) (f) = f \circ h$ for all $f \in C (X)$.
Injectivity of $\gm \circ \ph$ implies surjectivity of~$h$.
It is now immediate that $\gm \circ \ph$ is completely
isometric.
Since $\gm$ and $\rh$ are completely
isometric, we conclude that $\pi$ is completely isometric.
\end{proof}

\section{Direct sums}\label{Sec_DSum}

In this section we show that direct sum of a family of
(matricial) $L^p$~operator algebras is also a (matricial)
$L^p$~operator algebra.

\begin{dfn}\label{Notdisjointunion}
If $\big( (X_i, {\mathcal{B}}_i, \mu_i) \big)_{i \in I}$ is a family
of \msp{s},
then the \msp{}
${\textdisp{
 \XBM = \coprod_{i \in I} (X_i, {\mathcal{B}}_i, \mu_i) }}$
is determined by taking
${\textdisp{ X = \coprod_{i \in I} X_i }}$,
\[
{\mathcal{B}}
 = \big\{ E \S X \colon
    {\mbox{$E \cap X_i \in {\mathcal{B}}_i$ for all $i \in I$}} \big\},
\]
and
${\textdisp{ \mu (E) = \sum_{i \in I} \mu_i (E \cap X_i) }}$
for $E \in {\mathcal{B}}$.
\end{dfn}

\begin{dfn}\label{D_4326_DSumNorm}
Whenever $N \in \N$ and $A_1, A_2, \ldots, A_N$
are Banach algebras,
we make ${\textdisp{\bigoplus_{k = 1}^N A_k }}$
a Banach algebra by giving it the obvious
algebra structure
and the norm
\[
\| (a_1, a_2, \ldots, a_N) \|
 = \max \big( \| a_1 \|, \, \| a_2 \|, \, \ldots, \, \| a_N \| \big)
\]
for $a_1 \in A_1, \, a_2 \in A_2, \, \ldots, \, a_N \in A_N$.
If $A_1, A_2, \ldots, A_N$
are matrix normed Banach algebras,
we define matrix norms on
${\textdisp{ \bigoplus_{k = 1}^N A_k }}$ by
\[
\| (a_1, a_2, \ldots, a_N) \|_n
 = \max
  \big( \| a_1 \|_n, \, \| a_2 \|_n, \, \ldots, \, \| a_N \|_n \big)
\]
for $n \in \N$
and
$a_1 \in M_n (A_1), \, a_2 \in M_n (A_2), \, \ldots,
  \, a_N \in M_n (A_N)$.
\end{dfn}

\begin{lem}\label{L_5128_DSumMat}
Let $N \in \N$.
Let $A_1, A_2, \ldots, A_N$
be matrix normed Banach algebras.
Then ${\textdisp{\bigoplus_{k = 1}^N A_k }}$,
as in Definition~\ref{D_4326_DSumNorm},
is a matrix normed Banach algebra.
\end{lem}

\begin{proof}
The proof is easy, and is omitted.
\end{proof}

\begin{lem}\label{L_5128_SupIn}
Let the notation be as in Definition~\ref{D_4326_DSumNorm}.
Let $B$ be a Banach algebra, and
for $k = 1, 2, \ldots, N$
let $\ph_k \colon B \to A_k$
be a \hm.
Define ${\textdisp{ \ph \colon B \to \bigoplus_{k = 1}^N A_k }}$
by $\ph (b) = \big( \ph_1 (b), \, \ph_2 (b), \, \ldots, \ph_N (b) \big)$
for $b \in B$.
Then $\ph$ is contractive \ifo{}
$\ph_k$ is contractive for $k = 1, 2, \ldots, N$.
If $A_1, A_2, \ldots, A_N$ are
matrix normed Banach algebras,
then $\ph$ is completely contractive \ifo{}
$\ph_k$ is completely contractive for
$k = 1, 2, \ldots, N$.
\end{lem}

\begin{proof}
The proof is immediate.
\end{proof}

\begin{lem}\label{L_6408_QtOfSum}
Let the notation be as in Definition~\ref{D_4326_DSumNorm}.
Let $S \subset \{ 1, 2, \ldots, N \}$.
Then ${\textdisp{ \bigoplus_{k \in S} A_k }}$
is an ideal in ${\textdisp{ \bigoplus_{k = 1}^N A_k }}$,
and the obvious map
\[
\bigoplus_{k = 1}^N A_k \Big/ \bigoplus_{k \in S} A_k
  \to \bigoplus_{k \not\in S} A_k
\]
is completely isometric when the quotient
is given the matrix norms of
Definition \ref{D_6408_Ops}(\ref{D_6408_Ops_Quot}).
\end{lem}

\begin{proof}
The proof is easy, and is omitted.
\end{proof}

\begin{lem}\label{L_5128_ToMat}
Let $A$ be a matrix normed Banach algebra.
Let $n \in \N$, and let
${\textdisp{ \ph \colon \bigoplus_{k = 1}^n A \to M_n (A) }}$
be the map
$\ph (a_1, a_2, \ldots, a_n) = \diag (a_1, a_2, \ldots, a_n)$
for $a_1, a_2,  \ldots, a_n \in A$.
Then $\ph$ is completely isometric.
\end{lem}

\begin{proof}
Let $r \in \N$.
Let $\sm$ be the standard bijection of Definition~\ref{D_5205_MmMnMmn},
with $r$ in place of~$n$,
and let $\te_{\sm}$ be as there.
For ${\textdisp{ a \in \bigoplus_{k = 1}^n A }}$ the matrix
$\big[ (\te_{\sm} \otimes \id_A)
  \circ (\id_{M_r} \otimes \ph) \big] (a)$
is block diagonal.
So iteration of condition~(\ref{D_4Y19_MatNAlg_DSum})
in Definition~\ref{D_4Y19_MatNAlg}
shows that $(\te_{\sm} \otimes \id_A) \circ (\id_{M_r} \otimes \ph)$
is isometric.
Lemma~\ref{L_5205_MNormIndep}
implies that $\te_{\sm} \otimes \id_A$ is isometric.
So $\id_{M_r} \otimes \ph$ is isometric.
\end{proof}

\begin{lem}\label{L_4Y17_SumIsLp}
Let $p \in [1, \I)$.
In Definition~\ref{D_4326_DSumNorm},
if $A_1, A_2, \ldots, A_N$ are $L^p$~operator algebras,
then so is ${\textdisp{ A = \bigoplus_{k = 1}^N A_k }}$.
If $A_1, A_2, \ldots, A_N$ are matricial $L^p$~operator algebras,
then $A$ is a matricial $L^p$~operator algebra.
\end{lem}

\begin{proof}
We give the proof for $L^p$~operator algebras;
the matricial case is essentially the same.
Suppose that $\rh_k \colon A_k \to L (L^p (X_k, \mu_k))$
is an isometric representation for $k = 1, 2, \ldots, N$.
Let ${\textdisp{ X = \coprod_{k = 1}^N X_k }}$
and $\mu$ be as in Definition~\ref{Notdisjointunion}.
Then $L^p (X, \mu)$ is the $L^p$~direct sum
${\textdisp{\bigoplus_{k = 1}^N L^p (X_k, \mu_k) }}$.
Define
${\textdisp{ \rh \colon \bigoplus_{k = 1}^N A_k \to \LLp }}$
by
\[
\rh (a_1, a_2, \ldots, a_N)
 = \rh_1 (a_1) \oplus \rh_2 (a_2) \oplus \cdots \oplus \rh_N (a_N)
\]
for
$a_1 \in A_1, \, \, a_2 \in A_2, \, \, \ldots, \, \, a_N \in A_N$.
Clearly $\rho$ is an isometric representation
of ${\textdisp{ \bigoplus_{k = 1}^N A_k }}$.
\end{proof}

\section{Hermitian idempotents}\label{Sec_HIdemp}

\indent
The right kind of idempotent to consider in an $L^p$~operator algebra
for $p \neq 2$
is what might be called a ``spatial idempotent'',
that is, one which is a spatial partial isometry
in the sense of Definition~6.4 of~\cite{PhLp1}.
We develop some of the basic theory in this section.
Such idempotents can be characterized as those which
are hermitian in the sense of Definition~\ref{D_6608_Herm} below,
a much older notion (see \cite{Vid}).
Although we will not need the general theory of
hermitian elements of a Banach algebra,
it seems appropriate to make the connection with the older concept.

We formalize the following terminology for idempotents.

\begin{dfn}[Definition 4.1.1 of~\cite{Bl3}]\label{D_4Y17_Idemp}
Let $A$ be a ring (not necessarily unital),
and let $e, f \in A$ be idempotents.
We say that
{\emph{$f$ dominates~$e$}},
written $f \geq e$ or $e \leq f$,
if $f e = e f = e$.
We say that $e$ and $f$ are
{\emph{orthogonal}}
if $e f = f e = 0$.
\end{dfn}

Even if $A$ is a \ca,
the notation $e \leq f$ need not agree
with the usual C*-algebraic order.
Among other things, $e$ and $f$ need not be selfadjoint.
If $e$ and $f$ happen to be \pj{s} in a \ca,
then our notation does agree with the usual
C*-algebraic order.
Orthogonality need not be
the same as the version of orthogonality for \ca{s}
implicit in the remark after Definition 4.1.1
of~\cite{Bl3}.

\begin{dfn}[Definition 2.6.1 of~\cite{Plm2}]\label{D_6608_NumRan}
Let $A$ be a unital Banach algebra in which $\| 1 \| = 1$.
Let $a \in A$.
Then the {\emph{numerical range}} $W (a)$
is the set of all numbers $\om (a) \in \C$ for linear
functionals $\om$ on $A$ such that $\| \om \| = \om (1) = 1$.
\end{dfn}

\begin{thm}[Theorem 2.6.7 of~\cite{Plm2}]\label{T_6608_HCnd}
Let $A$ be a unital Banach algebra in which $\| 1 \| = 1$,
and let $a \in A$.
Then \tfae:
\begin{enumerate}
\item\label{T_6608_HCnd_W}
$W (a) \S \R$.
\item\label{T_6608_HCnd_Exp}
$\| \exp (i \ld a) \| = 1$ for all $\ld \in \R$.
\item\label{T_6608_HCnd_SubExp}
$\| \exp (i \ld a) \| \leq 1$ for all $\ld \in \R$.
\item\label{T_6608_HCnd_Lim}
With the limit being taken over $\ld \in \R$,
${\textdisp{
\lim_{\ld \to 0} | \ld |^{-1} \big( \| 1 - i \ld a \| - 1 \big) = 0 }}$.
\end{enumerate}
\end{thm}

\begin{proof}
The equivalence of conditions
(\ref{T_6608_HCnd_W}),
(\ref{T_6608_HCnd_Exp}),
and~(\ref{T_6608_HCnd_Lim})
is in Theorem 2.6.7 of~\cite{Plm2}.
That~(\ref{T_6608_HCnd_Exp})
implies~(\ref{T_6608_HCnd_SubExp})
is trivial.
That~(\ref{T_6608_HCnd_SubExp})
implies~(\ref{T_6608_HCnd_Exp})
follows from $\| 1 \| = 1$
and $\exp (i \ld a)^{-1} = \exp (- i \ld a)$.
\end{proof}

\begin{dfn}[see Definition 2.6.5 of~\cite{Plm2} and the
  preceding discussion]\label{D_6608_Herm}\label{Hermitianidempotent}
Let $A$ be a unital Banach algebra in which $\| 1 \| = 1$,
and let $a \in A$.
We say that $a$ is {\emph{hermitian}}
if $a$ satisfies the equivalent conditions of Theorem~\ref{T_6608_HCnd}.
If $a$ is also an idempotent,
we call it a {\emph{hermitian idempotent}}.
\end{dfn}

\begin{rmk}\label{R_4Y18_01}
Let $A$ be a unital Banach algebra in
which $\| 1 \| = 1$.
Then clearly $0$ and $1$ are
hermitian idempotents.
Also, in any unital
Banach algebra, if $e$ is a hermitian idempotent,
then so is $1 - e$.
Indeed, if $\ld \in \R$ then
\[
\| \exp (i \ld (1 - e)) \|
 = \| e + \exp (i \ld) (1 - e) \|
 = \| \exp (i \ld) \exp (- i \ld e) \|
 = \| \exp (- i \ld e) \|
 = 1,
\]
as desired.
\end{rmk}

A hermitian idempotent in a \ca{} is simply a projection.
(See Proposition 3.3.3 in~\cite{JC}, observing that a
nonzero hermitian idempotent has norm 1
by \Lem{L_6612_HI} below.)

The following result gives the characterization we use most often.

\begin{lem}\label{L_6612_HI}
Let $A$ be a unital Banach algebra in which $\| 1 \| = 1$.
Let $e \in A$ be an idempotent.
Define a \hm{} $\bt_e \colon \C \oplus \C \to A$
by $\bt_e (\ld_1, \ld_2) = \ld_1 e + \ld_2 (1 - e)$
for $\ld_1, \ld_2 \in \C$.
Then $e$ is hermitian in the sense of Definition~\ref{D_6608_Herm}
\ifo,
when $\C \oplus \C$ is normed as in Definition~\ref{D_4326_DSumNorm},
the \hm{} $\bt_e$ is contractive.
\end{lem}

\begin{proof}
We use the characterization~(\ref{T_6608_HCnd_SubExp})
of Theorem~\ref{T_6608_HCnd}.

First suppose that $\bt_e$ is contractive.
Then for $\ld \in \R$ we have
\[
\| \exp (i \ld e) \|
 = \| \bt ( (\exp (i \ld), \, 1) ) \|
 \leq \| (\exp (i \ld), \, 1) \|
 = 1.
\]

For the converse,
suppose $e$ is hermitian,
and let $\ld_1, \ld_2 \in \C$.
We need to prove
\begin{equation}\label{Eq_6807_HEst}
\| \bt_e ( (\ld_1, \ld_2) ) \| \leq \max ( |\ld_1|, |\ld_2|).
\end{equation}

This relation is trivial if $\ld_1 = \ld_2 = 0$.

Next, suppose $| \ld_1 | \leq | \ld_2 |$ and $\ld_2 \neq 0$.
Multiplying by $\ld_2^{-1}$,
we reduce to the case $\ld_2 = 1$.
Write $\ld_1 = \rh \exp (i \te)$
with $\te \in \R$
and $0 \leq \rh \leq 1$.
Define
\[
\af_1 = \te + \arccos (\rh)
\andeqn
\af_2 = \te - \arccos (\rh).
\]
Then one checks that
$(\ld_1, 1)
 = \frac{1}{2}
    \big[ (\exp (i \af_1), \, 1) + (\exp (i \af_2), \, 1) \big]$.
So
\[
\| \bt_e ((\ld_1, 1)) \|
 = \left\| \frac{1}{2}
    \big[ \exp (i \af_1 e ) + \exp (i \af_2 e ) \big]\right\|
 \leq \frac{1}{2}
    \big( \| \exp (i \af_1 e ) \| + \| \exp (i \af_2 e ) \| \big)
 \leq 1,
\]
which is~(\ref{Eq_6807_HEst}).

Finally, suppose $| \ld_2 | \leq | \ld_1 |$ and $\ld_1 \neq 0$.
Using Remark~\ref{R_4Y18_01},
we can apply the case of~(\ref{Eq_6807_HEst}) already done
to $1 - e$,
with $(\ld_2, \ld_1)$ in place of $(\ld_1, \ld_2)$.
This gives~(\ref{Eq_6807_HEst}) for $e$ and $(\ld_1, \ld_2)$.
\end{proof}

\begin{lem}\label{L_5128_UnitalHm}
Let $A$ and $B$ be unital Banach algebras such that
$\|1_A\| = 1$ and $\|1_B\| = 1$.
Let $\ph \colon A \to B$ be
a contractive unital \hm, and let $e \in A$ be a hermitian
idempotent. Then $\ph (e) \in B$ is a hermitian idempotent.
\end{lem}

\begin{proof}
The proof is immediate from Lemma~\ref{L_6612_HI}.
\end{proof}

\begin{lem}\label{L_5128_SIDSum}
Let $N \in \N$, and
let $A_1, A_2, \ldots, A_N$ be unital Banach algebras whose
identities have norm one.
Set ${\textdisp{ A = \bigoplus_{k = 1}^N A_k }}$, equipped with
the norm in Definition~\ref{D_4326_DSumNorm},
and for $k = 1, 2, \ldots, N$ let $e_k$
be a hermitian idempotent in $A_k$.
Then $(e_1, e_2, \ldots, e_N)$ is a hermitian idempotent in~$A$.
\end{lem}

\begin{proof}
The proof is immediate from Lemma~\ref{L_6612_HI}.
\end{proof}

We are interested in hermitian idempotents in
$L^p$~operator algebras.
Given $p \in [1, \infty) \setminus \{ 2 \}$,
the following lemma gives a characterization
of hermitian idempotents in $\LLp$, for
a \sfm{} $\XBM$.

\begin{lem}\label{L_4326_CharSpI}
Let $p \in [1, \infty) \setminus \{ 2 \}$.
Let $\XBM$ be a \sfm,
and let $e \in \LLp$ be an idempotent.
Then \tfae:
\begin{enumerate}
\item\label{L_4326_SpIdem}
$e$ is a hermitian idempotent.
\item\label{L_4326_SpPI}
$e$ is a spatial partial isometry
in the sense of Definition~6.4 of~\cite{PhLp1}.
\item\label{L_4326_Mult}
There is a \mb{} subset $E \subset X$
such that $e$ is multiplication by~$\ch_E$.
\end{enumerate}
\end{lem}

\begin{proof}
Lemma~6.18 in~\cite{PhLp1}
shows that
(\ref{L_4326_SpPI}) and~(\ref{L_4326_Mult})
are equivalent.

It is obvious that
(\ref{L_4326_Mult}) implies~(\ref{L_4326_SpIdem}).
For the
converse, assume that $e$ is a hermitian idempotent.
Thus the homomorphism
$\bt_e \colon \mathbb{C} \oplus \mathbb{C} \to \LLp$
of Lemma~\ref{L_6612_HI}
is unital and contractive.
Set $Y = \{0, 1 \}$,
and define $f \in C (Y)$ by
$f (0) = 1$ and $f (1) = 0$.
Let $\rh$ be the representation of $L^{\infty} (X, \mu)$
on $L^{p} (X, \mu)$ by multiplication operators.
By Theorem~\ref{T_4116_CXLp},
there exists a unital *-homomorphism
$\ph \colon C (Y) \to L^{\infty} (X, \mu)$
such that $\bt_e = \rh \circ \ph$.
Since $\ph (f)$ is an idempotent in $L^{\infty} (X, \mu)$,
there is a measurable set $E \subset X$
such that $\ph (f) = \chi_E$.
\end{proof}

\begin{cor}\label{C_5128_SpIsDiag}
Let $p \in [1, \infty) \setminus \{ 2 \}$,
let $d \in \N$,
and let $e \in \MP{d}{p}$ be an idempotent.
Then $e$ is hermitian \ifo{} $e$
is a diagonal matrix with entries in $\{ 0, 1 \}$.
\end{cor}

\begin{proof}
Identify $\MP{d}{p} = L (l_d^p)$.
Then the statement is immediate from
the equivalence of (\ref{L_4326_SpIdem})
and~(\ref{L_4326_Mult})
in Lemma~\ref{L_4326_CharSpI}.
\end{proof}

We now give several counterexamples.
It isn't enough
to require that $\| e \| \leq 1$ and $\| 1 - e \| \leq 1$ to
get a hermitian idempotent, even if $A$ is
a $\sm$-finitely representable unital $L^p$~operator
algebra.
There is an example in~$\MP{2}{p}$.

\begin{lem}\label{L_5224_NmIdemp}
Let $p \in [1, \I)$.
Define $e \in \MP{2}{p}$
by
$e = \frac{1}{2} \left( \begin{smallmatrix}
  1     &  1        \\
  1     &  1
\end{smallmatrix} \right)$.
Then $\| e \|_p = \| 1 - e \|_p = 1$,
but if $p \neq 2$ then $e$ is not a hermitian idempotent.
\end{lem}

\begin{proof}
Let $\af, \bt \in \C$.
Let $q > 1$ be such that
$\frac{1}{p} + \frac{1}{q} = 1$.
Applying H\"{o}lder's inequality
to $(1, 1) \in l_2^q$ and $(\af, \bt) \in l_2^p$,
we get
\[
| \af + \bt | \leq 2^{1 - 1 / p} (| \af |^p + | \bt |^p )^{1/p}.
\]
Use this inequality at the third step, to get
\[
\| e (\af, \bt) \|_p
 = \left\| \left( \tfrac{1}{2} (\af + \bt),
    \, \tfrac{1}{2} (\af + \bt) \right) \right\|_p
 = 2^{1/p} \left| \frac{\af + \bt}{2} \right|
    \leq (| \af |^p + | \bt |^p )^{1/p}
 = \| (\af, \bt) \|_p.
\]
Since $\af, \bt \in \C$ are arbitrary,
this shows that $\| e \|_p \leq 1$.
Obviously,
we have $\| e \|_p \geq 1$ because $e^2 = e$ and $e \neq 0$.

The same argument applies to $1 - e$.
(Or else take
$s = \left( \begin{smallmatrix}
  1     &  \; 0        \\
  0     &  -1
\end{smallmatrix} \right)$
and use the fact that $s$ is an invertible isometry with
$s e s^{-1} = 1 - e$.)

It follows from Corollary~\ref{C_5128_SpIsDiag}
that $e$ is not hermitian.
\end{proof}

One can also explicitly show that $e$ is not hermitian.
For example, suppose $p < 2$.
Letting $\bt_e$ be as in Lemma~\ref{L_6612_HI},
one can explicitly show that
$\| \bt_e ((1, i)) (1, 0) \|_p = 2^{1/p - 1/2} > 1 = \| (1, 0) \|_p$.

We don't define hermitian idempotents in a nonunital
Banach algebra, since whether an idempotent
is hermitian depends on the norm used on the unitization,
even for $L^p$~operator algebras,
as is shown by the following example.

\begin{exa}\label{E_unitization}
Let $p \in [1, \infty) \setminus \{ 2 \}$.
Let $X$ be a second countable locally compact
Hausdorff space with a nontrivial compact open
set $K$.
Let $\mu$ a measure on $X$ with full
support.
Denote by $\ps$ the representation of
$C_0 (X)$ on $L^p (X, \mu)$ given by multiplication operators.
Let $e$ be the idempotent in
Lemma~\ref{L_5224_NmIdemp}.
Define
$\rho \colon C_0 (X) \to L \big( l^p_2 \otimes_p L^p (X, \mu) \big)$
by $\rho (f) = e \otimes \ps (f)$.
Then $\rho$ is isometric.
Take as unitization the subagebra
$\rho (C_0 (X)) \oplus \mathbb{C}1$ of
$L \big( l^p_2 \otimes_p L^p (X, \mu) \big)$.
Let $\chi_{K}\in C_0 (X)$
be the characteristic function of $K$.
Then $\ps (\chi_{K})$ is
a hermitian idempotent in $L(L^p (X, \mu))$, but $\rho (\chi_{K})$
is not a hermitian idempotent because $e$ is not a
hermitian idempotent.
\end{exa}

\begin{lem}\label{L_4Y18_Funct}
Let $p \in [1, \infty) \setminus \{ 2 \}$.
Let $A$ and $B$ be unital
$\sm$-finitely representable $L^p$~operator algebras
with $\| 1_A \| = 1$ and $\| 1_B \| = 1$,
and let $\ps \colon A \to B$ be a contractive \hm{}
such that $\ps (1)$ is a hermitian idempotent in~$B$.
Let $e \in A$ be a hermitian idempotent.
Then $\ps (e)$ is a hermitian idempotent in~$B$.
\end{lem}

We don't know to what extent the hypotheses can be
weakened.
But some hypothesis is necessary.
Let $p \in [1, \infty) \setminus \{ 2 \}$.
By Lemma~\ref{L_5224_NmIdemp} there is a
nonhermitian idempotent $e \in \MP{2}{p}$ such that
$\| e \|_p = 1$.
The homomorphism $\C \to \MP{2}{p}$ defined by
$\ld \mapsto \ld e$ is contractive
but sends the hermitian idempotent~$1$
to the nonhermitian idempotent~$e$.

\begin{proof}[Proof of Lemma~\ref{L_4Y18_Funct}.]
We may assume that there is a \sfm{} $\YCN$
such that $B$ is a unital subalgebra of $L (L^p (Y, \nu))$.
Lemma~\ref{L_4326_CharSpI} provides
a \mb{} subset $E \subset X$ such that
$\ps (1)$ is multiplication by~$\ch_E$.
By Corollary~\ref{C_5125_CutAlg}, we may
view $\ps$ as a unital contractive \hm{}
from $A$ to $L (L^p (E, \nu |_{E}))$.

Let $\bt_e \colon \mathbb{C} \oplus \mathbb{C} \to A$
be as in Lemma~\ref{L_6612_HI}.
Then $\bt_e$ is contractive, so $\ps \circ \bt_e$ is contractive.
By Lemma~\ref{L_5128_UnitalHm}, it follows that
$\ps (e)$ is a hermitian idempotent
in $L (L^p (E, \nu |_{E}))$.
Lemma~\ref{L_4326_CharSpI}
provides a \mb{} subset $F \subset E$ such that
$\ps (e)$ is multiplication by~$\ch_F$.
Another application of Lemma~\ref{L_4326_CharSpI}
implies that $\ps (e)$ is a hermitian idempotent
in $L (L^p (Y, \nu))$.
So $\ps (e)$ is hermitian in~$B$ by Lemma~\ref{L_6612_HI}.
\end{proof}

\begin{cor}\label{C_4Y17_Orth}
Let $p \in [1, \infty) \setminus \{ 2 \}$.
Let $\XBM$ be a \sfm,
let $N \in \N$,
and let $e_1, e_2, \ldots, e_N \in \LLp$ be
orthogonal hermitian idempotents.
Then:
\begin{enumerate}
\item\label{C_4Y17_Orth_Disj}
There exist disjoint measurable sets
$E_1, E_2, \ldots, E_N \subset X$
such that $e_k$ is multiplication by $\ch_{E_k}$
for $k = 1, 2, \ldots, N$.
\item\label{C_4Y17_Orth_Sum}
${\textdisp{ \sum_{k = 1}^N e_k }}$ is a hermitian
idempotent.
\item\label{C_4Y17_Orth_LpNorm}
For every $\xi \in L^p (X, \mu)$,
we have
${\textdisp{ \| \xi \|_p^p = \sum_{k = 1}^N \| e_k \xi \|_p^p }}$.
\item\label{C_4Y17_Orth_Norm}
The map $\bt \colon \C^N \to \LLp$,
given by
${\textdisp{ \bt ( \ld_1, \ld_2, \ldots, \ld_N )
  = \sum_{k = 1}^N \ld_k e_k }}$
for $\ld_1, \ld_2, \ldots, \ld_N \in \C$,
is a contractive \hm.
\end{enumerate}
\end{cor}

\begin{proof}
Let $\rh \colon L^{\infty} (X, \mu) \to \LLp$
be the representation by multiplication operators.
Lemma~\ref{L_4326_CharSpI}
provides measurable sets
$F_1, F_2, \ldots, F_N \subset X$
such that $e_k = \rh (\ch_{F_k})$
for $k = 1, 2, \ldots, N$.
Since $e_j e_k = 0$ for $j \neq k$,
we have $\mu (F_j \cap F_k) = 0$ for $j \neq k$.
So there exist disjoint measurable sets
$E_k \subset F_k$ for $k = 1, 2, \ldots, N$
such that $\mu (F_k \setminus E_k) = 0$.
Then $\rh (\ch_{E_k}) = \rh (\ch_{F_k})$.
This proves~(\ref{C_4Y17_Orth_Disj}).
Part~(\ref{C_4Y17_Orth_Sum})
follows from Lemma~\ref{L_4326_CharSpI}
because setting ${\textdisp{ E = \bigcup_{k = 1}^N E_k }}$
gives ${\textdisp{ \sum_{k = 1}^N e_k = \rh (\ch_{E}) }}$.
Part~(\ref{C_4Y17_Orth_LpNorm}) is immediate.
For~(\ref{C_4Y17_Orth_Norm}),
define $\bt_0 \colon \C^N \to L^{\infty} (X, \mu)$
by
${\textdisp{ \bt_0 ( \ld_1, \ld_2, \ldots, \ld_N ) =
 \sum_{k = 1}^N \ld_k \ch_{E_k} }}$
for $\ld_1, \ld_2, \ldots, \ld_N \in \C$.
Then $\bt_0$ and $\rh$ are contractive \hm{s},
and $\bt = \rh \circ \bt_0$.
\end{proof}

\begin{lem}\label{L_5209_CommuteSp}
Let $p \in [1, \infty) \setminus \{ 2 \}$.
Let $A$ be a separable unital $L^p$~operator algebra,
let $N \in \N$,
and let $e_1, e_2, \ldots, e_N \in A$
be orthogonal hermitian idempotents
such that ${\textdisp{ \sum_{k = 1}^N e_k = 1 }}$.
Assume $a \in A$ satisfies $a e_k = e_k a$
for $k = 1, 2, \ldots, N$.
Then
\[
\| a \| = \max \big( \| e_1 a e_1 \|, \, \| e_2 a e_2 \|, \,
          \ldots, \, \| e_N a e_N \| \big).
\]
\end{lem}

\begin{proof}
We prove this when $N = 2$,
$e_1 = e$,
and $e_2 = 1 - e$.
The general case follows by induction using
Corollary \ref{C_4Y17_Orth}.
We may assume that $e \neq 0$.

It is immediate from
Lemma~\ref{L_6612_HI} that $\| e \| = 1$.
Proposition~\ref{P_4Y19_Unital}
provides an isometric unital representation $\pi$ of $A$
on a separable $L^p$~space $L^p (X, \mu)$.
By Lemma
\ref{L_6416_SepSFT} we may assume that $\mu$ is \sft.
Lemma~\ref{L_5128_UnitalHm} implies that $\pi (e)$ is
a hermitian idempotent,
so Lemma~\ref{L_4326_CharSpI} provides a measurable set $E \subset X$
such that $\pi (e)$ is multiplication by $\ch_E$.
Thus $\pi (a)$ commutes with multiplication by $\ch_E$.
With respect to the $L^p$~direct sum decomposition
$L^p (X, \mu) = L^p (E, \mu) \oplus_p L^p (X \setminus E, \mu)$,
we get $\pi (a) = \pi (e a e) \oplus \pi \big( (1 - e) a (1 - e) \big)$.
So
\begin{align*}
\| a \|
& = \| \pi (a) \|
\\
& = \max \big( \| \pi (e a e) \|, \,
                \| \pi ( (1 - e) a (1 - e) ) \| \big)
  = \max \big( \| e a e \|, \, \| (1 - e) a (1 - e) \| \big).
\end{align*}
This completes the proof.
\end{proof}

It follows that if $A$ is a separable unital $L^p$~operator
algebra and $e \in A$ is a central hermitian idempotent,
then $A = e A e \oplus (1 - e) A (1 - e)$,
normed as in Definition~\ref{D_4326_DSumNorm}.
We don't know whether this is true for more general
Banach algebras.

\begin{lem}\label{P_5124_UniqMatDSum}
In Definition~\ref{D_4326_DSumNorm},
suppose that $A_1, A_2, \ldots, A_N$ are
separable unital $L^p$~operator algebras
whose identities have norm~$1$,
and that they have unique $L^p$~operator matrix norms.
Then ${\textdisp{ A = \bigoplus_{k = 1}^N A_k }}$,
normed as in Definition~\ref{D_4326_DSumNorm},
has unique $L^p$~operator matrix norms.
\end{lem}

\begin{proof}
For $k = 1, 2, \ldots, N$ choose some unital isometric \rpn{}
of $A_k$ on a separable $L^p$~space of a \sft{} measure,
and equip $A_k$ with the matrix
norms on its image under this \rpn{}
as in Definition~\ref{D_4Y19_LpMN}.
Then make $A$ a matrix normed algebra as
in Definition~\ref{D_4326_DSumNorm}. In view
of Lemma~\ref{L_5125_UnitUniq},
it suffices to prove that if
$\XBM$ is a \sfm{}
with $L^p (X, \mu)$ separable,
$n \in \N$,
and $\pi \colon A \to \LLp$
is a unital isometric \rpn,
then
$\id_{M_n} \otimes \pi \colon
M_n (A) \to L \big( l_n^p \otimes_p L^p (X, \mu) \big)$
is  isometric.

For $k = 1, 2, \ldots, N$,
we identify $A_k$ with its image in $A$,
and we let $f_k$ be the identity of $A_k$.
Set $Z = \{ 1, 2, \ldots, N \}$.
Let $\ph \colon C (Z) \to A$
be the unital \hm{}
determined by $\ph ({\ch}_{ \{ k \} }) = f_k$
for $k = 1, 2, \ldots, N$.
It is obvious from the definition of the norm on~$A$
that $\ph$ is isometric.
Then $\pi \circ \ph$
is isometric,
from which it easily follows that
$(\pi \circ \ph) ({{\ch}}_{ \{ k \} })$
is a hermitian idempotent for $k = 1, 2, \ldots, N$.
Let $E_1, E_2, \ldots, E_N \subset X$
be the disjoint sets corresponding to the idempotents
$\big ((\pi \circ \ph) ({{\ch}}_{ \{ k \} }) \big )_{k = 1}^N$
as in Corollary \ref{C_4Y17_Orth}(\ref{C_4Y17_Orth_Disj}).
Since $\pi \circ \ph$ is unital, we can assume
\wolog{} that ${\textdisp{ \bigcup_{k = 1}^N E_k = X }}$.
For $k = 1, 2, \ldots, N$,
let $\pi_k \colon A_k \to L(L^p (E_k, \mu))$
be the unital \rpn{}
gotten as in Corollary~\ref{C_5125_CutAlg}
from $\pi |_{A_k}$.
Corollary~\ref{C_5125_CutAlg} and the definition
of the norm on~$A$
imply that $\pi_k$ is isometric.

It is immediate that
\mbox{$l_n^p \otimes_p L^p (X, \mu)$}
is the $L^p$~direct sum of the spaces
\mbox{$l_n^p \otimes_p L^p (E_k, \mu)$} for $k = 1, 2, \ldots, N$.
It follows that if
${\textdisp{ a = (a_1, a_2, \ldots, a_N)
  \in \bigoplus_{k = 1}^N M_n (A_k) = M_n (A) }}$,
then
\begin{align*}
& \| (\id_{M_n} \otimes \pi) (a) \|
\\
& \hspace*{3em} \mbox{}
= \max \big( \| (\id_{M_n} \otimes \pi_1) (a_1) \|, \,
     \| (\id_{M_n} \otimes \pi_2) (a_2) \|, \,
      \ldots, \, \| (\id_{M_n} \otimes \pi_N) (a_N) \| \big).
\end{align*}
The hypotheses imply that $\id_{M_n} \otimes \pi_k$
is isometric for $k = 1, 2, \ldots, N$.
Definition~\ref{D_4326_DSumNorm}
therefore implies that $\id_{M_n} \otimes \pi$ is isometric.
\end{proof}

The following lemma will be used in connection with
representations of
nonunital spatial $L^p$~AF algebras.

\begin{lem}\label{L_6808_IncIdem}
Let $p \in [1, \infty) \setminus \{ 2 \}$.
Let $\XBM$ be a \sfm{}
such that $L^p (X, \mu)$ is separable.
Let $e_1, e_2, \ldots \in \LLp$
be idempotents,
and take $e_0 = 0$.
Assume that,
for all $n \in \N$,
$\| e_n \| \leq 1$
and
$e_{n - 1}$ is a hermitian idempotent in
$e_n \LLp e_n$.
Then there are an idempotent $e \in \LLp$,
a \sfm{} $\YCN$
such that $L^p (Y, \nu)$ is separable,
an isometric linear map $s \colon L^p (Y, \nu) \to L^p (X, \mu)$,
and measurable subsets $Y_1, Y_2, \ldots \S Y$,
such that:
\begin{enumerate}
\item\label{L_6808_IncIdem_Norm}
$\| e \| \leq 1$.
\item\label{L_6808_IncIdem_Union}
${\textdisp{ e L^p (X, \mu)
 = {\overline{\bigcup_{n = 1}^{\I} e_n L^p (X, \mu)}} }}$.
\item\label{L_6808_IncIdem_Herm}
For every $n \in \N$,
$e_n$ is a hermitian idempotent in
$e \LLp e$.
\item\label{L_6808_IncIdem_Ran}
${\operatorname{ran}} (s) = e L^p (X, \mu)$.
\item\label{L_6808_IncIdem_Disj}
${\textdisp{ Y = \coprod_{n = 1}^{\I} Y_n }}$.
\item\label{L_6808_IncIdem_s}
For every $n \in \N$,
$s L^p (Y_n, \nu |_{Y_n}) = (e_n-e_{n-1}) L^p (X, \mu)$.
\end{enumerate}
\end{lem}

We do not assume that $e_n$ is a hermitian idempotent in $\LLp$,
and the conclusion does not claim that
$e$ is a hermitian idempotent in $\LLp$.

\begin{proof}[Proof of Lemma~\ref{L_6808_IncIdem}]
For $n \in \N$ define $E_n = e_n L^p (X, \mu) \S  L^p (X, \mu)$.
Set ${\textdisp{ E = {\overline{\bigcup_{n = 1}^{\I} E_n}} }}$.
For $n \in \N$, use Proposition~\ref{P_5204_ContrPj}
to find a \sfm{} $(Z_n, {\mathcal{D}}_n, \ld_n)$
such that $L^p (Z_n, \ld_n)$ is isometrically isomorphic to~$E_n$.
Also,
define $\pi_n \c L (E_n) \to \LLp$
by $\pi_n (a) \xi = a e_n \xi$
for $a \in L (E_n)$ and $\xi \in L^p (X, \mu)$.
Since $e_n a e_n \xi = a e_n \xi$ for $a\in L(E_n)$,
one checks easily that $\pi_n$ is a (nonunital) \hm.
An application of Corollary~\ref{C_5125_CutAlg}
shows that $\pi_n$ is isometric.
Thus,
$L \big( L^p (Z_n, \ld_n) \big)$
is isometrically isomorphic to $e_n \LLp e_n$.

For $m, n \in \N$ with $m \leq n$,
a similar argument shows that
the analogous map $\pi_{n, m} \c L (E_m) \to L (E_n)$
is an isometric \hm.
It follows from Remark~\ref{R_4Y18_01}
that $e_m - e_{m - 1}$ is a hermitian idempotent in
$e_m \LLp e_m$,
and then it follows from Lemma~\ref{L_4Y18_Funct}
and induction on~$n$
that $e_m - e_{m - 1}$ is a hermitian idempotent in
$e_n \LLp e_n$.

For $n \in \N$,
Corollary \ref{C_4Y17_Orth}(\ref{C_4Y17_Orth_LpNorm}),
applied to $e_1 - e_0, \, e_2 - e_1, \, \ldots, \, e_n - e_{n - 1}$,
shows that for every $\xi \in E_n$
we have
\begin{equation}\label{Eq_6808_pNormSum}
\| \xi \|_p^p = \sum_{k = 1}^n \| (e_k - e_{k - 1}) \xi \|_p^p.
\end{equation}

For $n \in \N$,
use Proposition~\ref{P_5204_ContrPj}
to find a \sfm{} $(Y_n, {\mathcal{C}}_n, \nu_n)$
and an isometric isomorphism
$s_n \c L^p (Y_n, \ld_n) \to (e_n - e_{n - 1}) L^p (X, \mu)$.
Following the notation of Definition~\ref{Notdisjointunion},
set
$\YCN
 = {\textdisp{ \coprod_{n = 1}^{\I} (Y_n, {\mathcal{C}}_n, \nu_n) }}$.
By~(\ref{Eq_6808_pNormSum}),
the map from
${\textdisp{
L^p \left( \coprod_{k = 1}^{n} (Y_k, {\mathcal{C}}_k, \nu_k) \right) }}$
to $e_nL^p(X,\mu)$, given by
\[
(\et_1, \et_2, \ldots, \et_n)
  \mapsto \sum_{k = 1}^n s_k (\et_k),
\]
is an isometric isomorphism.
Combining these for $n \in \N$ and extending by continuity,
we get an isometric linear map
$s \colon L^p (Y, \nu) \to L^p (X, \mu)$,
whose range must be equal to~$E$.

We now have the objects $s$, $\YCN$, and $Y_1, Y_2, \ldots \S Y$
of the conclusion,
as well as parts (\ref{L_6808_IncIdem_Disj})
and~(\ref{L_6808_IncIdem_s}).
If we use $E$ in place of $e L^p (X, \mu)$,
we also have parts (\ref{L_6808_IncIdem_Union})
and~(\ref{L_6808_IncIdem_Ran}) of the conclusion.

Now let $\xi \in L^p (X, \mu)$.
For $n \in \N$ we use $\| e_n \| \leq 1$ and~(\ref{Eq_6808_pNormSum})
to get
\[
\| \xi \|_p^p
 \geq \| e_n \xi \|_p^p
 = \sum_{k = 1}^n \| (e_k - e_{k - 1}) \xi \|_p^p.
\]
Therefore
${\textdisp{ \sum_{k = 1}^{\I} \| (e_k - e_{k - 1}) \xi \|_p^p }}$
converges.
For $m, n \in \N$ with $m \leq n$,
we get
${\textdisp{ \| (e_n - e_m) \xi \|_p^p
  = \sum_{k = m + 1}^{n} \| (e_k - e_{k - 1}) \xi \|_p^p }}$,
so $(e_n \xi )_{n \in \N}$ is a Cauchy sequence in $L^p (X, \mu)$.
Thus ${\textdisp{\lim_{n \to \I} e_n \xi}}$ exists.
Call this limit $e \xi$.
By taking suitable limits,
one checks that $e$ is linear,
$\| e \| \leq 1$,
$e^2 = e$,
and ${\operatorname{ran}} (e) = E$.
We now have all the required objects for the conclusion,
and all the conditions except~(\ref{L_6808_IncIdem_Herm}).

For~(\ref{L_6808_IncIdem_Herm}),
use continuity to get
$\| e_n \xi \|_p^p + \| (e - e_n) \xi \|_p^p = \| \xi \|_p^p$
for all $\xi \in E$ and $n \in \N$.
{}From this,
it is easy to see that
the map $(\ld_1, \ld_2) \mapsto \ld_1 e_n + \ld_2 (e - e_n)$
is a contraction from $\C \oplus \C$ to $e \LLp e$.
Apply Lemma~\ref{L_6612_HI}.
\end{proof}

\section{Direct limits}\label{Sec_DLim}

The main result in this section is that the direct limit
of matricial $L^p$~operator algebras is also a matricial
$L^p$~operator algebra.
Moreover, if each algebra
in the system has unique $L^p$~operator matrix norms,
and the connecting maps of the direct system are isometric,
then the direct limit also has unique $L^p$ operator matrix norms.

\begin{dfn}\label{D_4310_DirSys}
Let $I$ be an infinite directed set.
A {\emph{(completely) contractive direct system of Banach
algebras indexed by~$I$}} is a pair
$\big( (A_i)_{i \in I}, \, ( \ph_{j, i} )_{i \leq j} \big)$
consisting of a family $(A_i)_{i \in I}$
of (matrix normed) Banach algebras
and a family $( \ph_{j, i} )_{i \leq j}$ of (completely) contractive
\hm{s} $\ph_{j, i} \colon A_i \to A_j$
for $i, j \in I$ with $i \leq j$,
such that $\ph_{i, i} = \id_{A_i}$ for all $i \in I$
and $\ph_{k, j} \circ \ph_{j, i} = \ph_{k, i}$
whenever $i, j, k \in I$ satisfy $i \leq j \leq k$.
We say that the system is {\emph{unital}}
if $A_i$ is unital for all $i \in I$
and $\ph_{j, i}$ is unital for all $i, j \in I$ with $i \leq j$.

In the contractive case,
the {\emph{direct limit}} ${\textdisp{ \dirlim_i A_i }}$ of this direct
system is the Banach algebra direct (``inductive'')
limit, as constructed in Section~3.3 of~\cite{Bl3}.

In the completely contractive case,
for $n \in \N$
we use Lemma~\ref{L_6408_matrixnorms}
to identify ${\textdisp{ M_n \Big( \dirlim_i A_i \Big) }}$
with ${\textdisp{ \dirlim_i M_n (A_i) }}$
up to isomorphism of topological algebras.
Then we equip ${\textdisp{ M_n \Big( \dirlim_i A_i \Big) }}$
with the norm obtained by applying
the contractive case to ${\textdisp{ \dirlim_i M_n (A_i) }}$.
Lemma~\ref{L_5125_LimIsMatN} below
shows that we do indeed get a matrix normed Banach algebra
this way.
\end{dfn}

\begin{lem}\label{L_5125_LimIsMatN}
Let $\big( (A_i)_{i \in I}, \, ( \ph_{j, i} )_{i \leq j} \big)$
be a
completely contractive direct system of matrix normed
Banach algebras.
Then ${\textdisp{ \dirlim_i A_i }}$
is a matrix normed Banach algebra,
and for every $j \in I$
the standard \hm{}
${\textdisp{ \ph_j \colon A_j \to \dirlim_i A_i }}$
is completely contractive.
\end{lem}

\begin{proof}
The statement about complete contractivity follows from
the identification of the matrix norms.

For every
$n \in \mathbb{Z}_{>0}$ and $i \in I$ identify
$B_i^{(n)} = M_n (A_i)$ with $M_n \otimes A_i$, and let
$\ph_{j, i}^{(n)}
 = \id_{M_n} \otimes \ph_{j, i} \colon B_i^{(n)} \to B_j^{(n)}$
and
$\ph_i^{(n)}
 = \id_{M_n} \otimes \ph_i \colon B_i^{(n)} \to M_n (A)$
be the maps induced by $\ph_{j, i}$ and $\ph_i$.
Set
${\textdisp{ B^{(n)}
 = \bigcup_{i \in I} \ph_i^{(n)} \big( B_i^{(n)} \big) }}$.
We claim that the norms on $B^{(n)}$, for $n \in \N$,
obtained by viewing $B^{(n)}$ as a subalgebra of the
Banach algebra direct limit ${\textdisp{ \dirlim_i B_i^{(n)} }}$,
are a system of matrix norms
as in Definition~\ref{D_4Y19_MatNAlg}.

Let $b \in B^{(n)}$,
and choose $i_0 \in I$
and $a \in B_{i_0}^{(n)}$
such that $\ph_{i_0}^{(n)} (a) = b$.
Let $\sigma$ and $\tau$ be injective functions as in
Definition \ref{D_4Y19_MatNAlg}(\ref{D_4Y19_MatNAlg_SbMat}).
Since
${\textdisp{ M_m \Big( \dirlim_i A_i \Big) = \dirlim_i B_i^{(m)} }}$,
using Definition \ref{D_4Y19_MatNAlg}(\ref{D_4Y19_MatNAlg_SbMat})
in $M_n (A_i) = B_i^{(n)}$ we have
\begin{align*}
\big\| ( b_{\sigma (j), \tau (k)} )_{1 \leq j, k \leq m} \big\|_m
 & = \lim_{i}
    \big\| \ph_{i, i_0}^{(m)}
       \big( (a_{\sigma (j), \tau (k)})_{1 \leq j, k \leq m} \big)
                       \big\|_m
  \\
 & \leq \lim_{i}
     \big\| \ph_{i, i_0}^{(n)}
              \big( (a_{j, k})_{1 \leq j, k \leq n} \big)
             \big\|_n
   = \| b \|_n.
\end{align*}

Moreover, given
$\lambda_1, \lambda_2, \ldots, \lambda_n \in \mathbb{C}$,
if we set $s = \diag (\lambda_1, \lambda_2, \ldots, \lambda_n)$,
we have
$\ph_{i, i_0} (s a) = s \ph_{i, i_0} (a)$, so,
using Definition \ref{D_4Y19_MatNAlg}(\ref{D_4Y19_MatNAlg_PrdDiag})
on $M_n (A_i)$,
\begin{align*}
\| s b \|_n
& = \lim_{i} \| \ph_{i, i_0} (s a) \|_n
  = \lim_{i} \| s \ph_{i, i_0} (a) \|_n
\\
& \leq  \max \big( |\lambda_1|, |\lambda_2|, \ldots, |\lambda_n| \big)
        \lim_{i} \|\ph_{i, i_0} (a) \|_n
 = \max \big(| \lambda_1|, |\lambda_2|, \ldots, |\lambda_n| \big)
   \| b \|_n.
\end{align*}
Similarly
$\| b s \|_n
 \leq
 \max \big( |\lambda_1|, |\lambda_2|, \ldots, |\lambda_n| \big)
   \| b \|_n$.

Lastly, given $b_1 \in B^{(m)}$ and $b_2 \in B^{(n)}$,
there exist $i_0 \in I$, $a_1 \in B_{i_0}^{(m)}$, and
$a_2 \in B_{i_0}^{(n)}$ such that $b_1 = \ph_{i_0}^{(m)} (a_1)$
and $b_2 = \ph_{i_0}^{(n)} (a_2)$.
Therefore,
$\diag ( b_1, b_2 ) = \ph_{i_0}^{(m + n)} (\diag (a_1, a_2))$ and
\begin{align*}
\big\| \diag (b_1, b_2) \big\|_{m + n}
& = \lim_{i}
  \big\| \ph_{i, i_0}^{(m + n)} (\diag (a_1, a_2)) \big\|_{m + n}
\\
& = \lim_{i} \big\| \diag \big( \ph_{i, i_0}^{(m)} (a_1), \,
   \ph_{i, i_0}^{(n)} (a_2) \big) \big\|_{m + n}
\\
& = \lim_{i} \max \big( \big\| \ph_{i, i_0}^{(m)} (a_1) \big\|_m,
        \, \big\| \ph_{i, i_0}^{(n)} (a_2) \big\|_n \big)
  = \max ( \|b_1 \|_m, \|b_2 \|_n).
\end{align*}

This completes the proof of the claim.

Since $B^{(n)}$ is dense in
${\textdisp{ M_n \Big( \dirlim_i A_i \Big) }}$
for all $n \in \N$,
the conditions of Definition~\ref{D_4Y19_MatNAlg}
for ${\textdisp{ \dirlim_i A_i }}$ follow by continuity.
\end{proof}

\begin{thm}\label{T_4310_LpDLim}
Let $p \in [1, \infty)$.
Let $\big( (A_i)_{i \in I}, \, ( \ph_{j, i} )_{i \leq j} \big)$
be a contractive direct system of $L^p$~operator algebras.
Then ${\textdisp{ \dirlim_i A_i }}$ is an $L^p$~operator algebra.
\end{thm}

\begin{thm}\label{T_5125_LpMatDLim}
Let $p \in [1, \infty)$.
Let $\big( (A_i)_{i \in I}, \, ( \ph_{j, i} )_{i \leq j} \big)$
be a completely contractive direct system of matricial $L^p$~operator
algebras.
Then ${\textdisp{ \dirlim_i A_i }}$ is a matricial $L^p$~operator
algebra.
\end{thm}

The proofs are essentially the same.
We prove Theorem~\ref{T_4310_LpDLim} here.
We describe the modifications for
the proof of Theorem~\ref{T_5125_LpMatDLim} afterwards.

\begin{proof}[Proof of Theorem~\ref{T_4310_LpDLim}]
The statement is trivial if $I$ has a largest element.
Otherwise,
let ${\mathcal{U}}_0$ be the collection of all subsets
of~$I$ of the form
$\big\{ i \in I \colon i \geq i_0 \big\}$ for $i_0 \in I$.
These sets are nonempty.
Since $I$ is directed,
the intersection of any finite collection of them
contains another one.
Since $I$ has no largest element,
for every $i \in I$ there is $S \in {\mathcal{U}}_0$
such that $i \not \in S$.
Therefore there is a free ultrafilter ${\mathcal{U}}$ on~$I$
which contains ${\mathcal{U}}_0$.

By definition,
for every $i \in I$
there exists a measure space
$(X_i, {\mathcal{B}}_i, \mu_i)$
and an isometric representation
$\rh_i \colon A_i \to L(L^p (X_i, \mu_i))$.
Let $M$ be the Banach space ultraproduct
${\textdisp{
 \Big( \prod_{i \in I} L^p (X_i, \mu_i) \Big) \Big/ {\mathcal{U}} }}$
(Definition~2.1 of~\cite{Hn}).
By Theorem 3.3(ii) of~\cite{Hn},
there exists a measure space
$(X, {\mathcal{B}}, \mu)$
such that $M$ is isometrically isomorphic to $L^p (X, \mu)$.
So it suffices to find an isometric representation
of ${\textdisp{ \dirlim_i A_i }}$ on~$M$.

Let $B$ be the algebraic direct limit of the algebras~$A_i$,
and for $i \in I$ let $\ph_i \colon A_i \to B$
be the \hm{} associated to the direct system.
Equip $B$ with the direct limit seminorm,
and let ${\textdisp{ A = \dirlim_i A_i }}$ be the completion
of $B / \{ b \in B \colon \| b \| = 0 \}$,
with the obvious isometric map $\kp \colon B \to A$.

We will construct an isometric representation $\gm$
of $B$ on~$M$.
(It will not be injective;
rather, its kernel will be $\{ b \in B \colon \| b \| = 0 \}$.)
Let $x \in B$.
Choose $i \in I$ and
$a \in A_i$ such that $\ph_i (a) = x$.
We give an associated operator $y_l \in L (L^p (X_l, \mu_l))$
for each $l \in I$.
If $l \geq i$,
set $y_l = \rh_l ( \ph_{l, i} (a))$.
Otherwise,
set $y_l = 0$.
Clearly $\| y_l \| \leq \| a \|$ for all $l \in I$,
so the ultraproduct of operators
(Definition~2.2 of~\cite{Hn})
gives an operator
$y = ( y_l )_{\mathcal{U}} \in L (M)$
such that
${\textdisp{\| y \| = \lim_{\mathcal{U}} \| y_l \|}}$.
Since $J = \{ l \in I \colon l \geq i \}$
is cofinal in~$I$,
we have
\[
 \lim_{l \in I} \| y_l \|
 = \lim_{l \in J} \| \rh_l ( \ph_{l, i} (a)) \|
 = \lim_{l \in J} \| \ph_{l, i} (a) \|
 = \| x \|.
\]
The choice of ${\mathcal{U}}$ ensures that
${\textdisp{ \lim_{\mathcal{U}} \| y_l \|
  = \lim_{l \in I} \| y_l \| }}$.
Therefore $\| y \| = \| x \|$.

We claim that $y$ does not depend on the choices of
$i$ and $a \in A_i$.
To prove this,
suppose that $j \in I$ and $b \in A_j$
also satisfy $\ph_j (b) = x$.
Let $z_l \in L (L^p (X_l, \mu_i))$ for $l \in I$,
and $z = \big( z_l \big)_{\mathcal{U}} \in L (M)$,
be defined in the same way as $y_l$ above,
but using $j$ and $b$ in place of $i$ and~$a$.
Choose $k \in I$ such that $k \geq i$, $k \geq j$,
and $\ph_{k, i} (a) = \ph_{k, j} (b)$.
Then $z_l = y_l$ for all $l \in I$ with $l \geq k$,
and $\{ l \in I \colon l \geq k \} \in {\mathcal{U}}$,
so $z = y$.
The claim is proved.

It follows that there is a well defined isometric map
$\gm \colon B \to L (M)$
such that if $x \in B$ and $i \in I$ and $a \in A_i$
satisfy $\ph_i (a) = x$,
then $\gm (x)$ is the element~$y$ constructed above.
Using directedness of~$I$,
it is easy to prove that $\gm$ is a \hm.
Since $\gm$ is isometric,
we have
$\gm (x) = 0$ whenever $\| x \| = 0$,
and
there exists a unique isometric \hm{}
$\rh \colon A \to L (M)$
such that $\rh ( \kp (x)) = \gm (x)$
for all $x \in B$.
The existence of $\rh$ shows that $A$ is an
$L^p$~operator algebra.
\end{proof}

\begin{proof}[Proof of Theorem~\ref{T_5125_LpMatDLim}]
We describe the differences from
the proof of Theorem~\ref{T_4310_LpDLim}.
We choose the maps $\rh_i$ in the proof of Theorem~\ref{T_4310_LpDLim}
to be completely isometric,
not just isometric.
Let $m \in \N$ and let
$\nu$ be counting measure on $\{ 1, 2, \ldots, m \}$.
One can check that the obvious map
gives an isometric isomorphism
\[
 \left( \prod_{i \in I} L^p ( \{ 1, 2, \ldots, m \} \times X_i,
 \, \nu \times \mu_i) \big) \right) \Big/ {\mathcal{U}}
\to L^p ( \{ 1, 2, \ldots, m \} \times X, \, \nu \times \mu) \big).
\]
(This is a direct computation from the definitions.)
Using the standard isomorphism
${\textdisp{ M_m \Big( \dirlim_i A_i \Big)
  \cong \dirlim_i M_m (A_i) }}$,
the argument used in the proof of Theorem~\ref{T_4310_LpDLim}
to show that $\gm$ is isometric
now shows that $\id_{M_m} \otimes \gm$ is isometric.
Since this is true for all $m \in \N$,
we conclude that~$\gm$,
hence also~$\rh$,
is completely isometric.
\end{proof}

\begin{prp}\label{P_5125_LimHasUniq}
Let $p \in [1, \infty)$, and let
$\big( (A_i)_{i \in I}, \, ( \ph_{j, i} )_{i \leq j} \big)$
be as in Theorem~\ref{T_4310_LpDLim}.
Suppose further that $I$ is countable, that for all $i \in I$
the algebra $A_i$ is separable and has unique
$L^p$~operator matrix norms, and that for all
$i, j \in I$ with $i \leq j$, the map $\ph_{j, i}$ is
isometric.
Then ${\textdisp{ A = \dirlim_i A_i }}$
has unique $L^p$~operator matrix norms.
They are obtained by equipping $A_i$
with its unique $L^p$~operator matrix norms
for $i \in I$ and,
for each $n \in \N$,
giving $M_n (A)$
the norm coming from the contractive case of
Definition~\ref{D_4310_DirSys}
applied to ${\textdisp{ \dirlim_i  M_n (A_i) }}$.
\end{prp}

\begin{proof}
The hypotheses imply that $A$ is separable.
For $i \in I$, equip $A_i$ with its unique $L^p$~operator matrix norms.
The hypotheses imply that if $j \in I$ and $j \geq i$,
then $\ph_{j, i}$ is completely isometric.
Equip $A$ with the matrix norms in the statement.
Then $A$ is an
$L^p$~operator algebra by Theorem~\ref{T_5125_LpMatDLim}.

Let $\XBM$ be a \sfm{} with $L^p (X, \mu)$ separable,
and let $\pi \colon A \to \LLp$ be isometric.
We show that $\pi$ is completely isometric.
Let $n \in \N$.
For $i \in I$, let $\ph_i \colon A_i \to A$
be the \hm{} coming from the direct system.
Then $\ph_i$ is isometric,
so $\pi \circ \ph_i$ is isometric.
The hypothesis on $A_i$ implies that
$\id_{M_n} \otimes (\pi \circ \ph_i)$
is isometric.
Since ${\textdisp{ \bigcup_{i \in I} \ph_i (A_i) }}$ is
dense in~$A$, it follows that $\id_{M_n} \otimes \pi$
is isometric.
\end{proof}

\section{Spatial semisimple finite dimensional $L^p$~operator
 algebras}\label{Sec_ssfd}

In this section we introduce our setup by giving the definitions
of spatial semisimple finite dimensional $L^p$~operator algebras
and spatial homomorphisms.
We also give a characterization of spatial homomorphisms
between spatial semisimple finite dimensional $L^p$~operator algebras
in terms of block diagonal homomorphisms
(Lemma~\ref{L_4326_BlkDiagHm}), and show that any spatial
semisimple finite dimensional $L^p$~operator algebra
has unique $L^p$~operator matrix norms.

\begin{dfn}\label{D_5916_IS}
Let $B$ be a unital Banach algebra,
and let $b, c \in B$.
We say that $b$ and $c$ are {\emph{isometrically similar}}
if there is an invertible isometry $s \in B$
such that $c = s b s^{-1}$.

If $A$ is also a Banach algebra,
and $\ph, \ps \colon A \to B$ are linear maps,
we say that $\ph$ and $\ps$ are {\emph{isometrically similar}}
if there is an invertible isometry $s \in B$
such that $\ps (a) = s \ph (a) s^{-1}$ for all $a \in A$.
\end{dfn}

We state some immediate properties.

\begin{prp}\label{P_5916_PropOfISHm}
Let $A$ and $B$ be Banach algebras,
with $B$ unital,
and let $\ph, \ps \colon A \to B$
be isometrically similar linear maps.
\begin{enumerate}
\item\label{P_5916_PropOfISHm_Ct}
If $\ph$ is contractive then so is $\ps$.
\item\label{P_5916_PropOfISHm_Iso}
If $\ph$ is isometric then so is $\ps$.
\item\label{P_5916_PropOfISHm_Mn}
If $A$ and $B$ are matrix normed
(Definition~\ref{D_4Y19_MatNAlg})
then $\id_{M_n} \otimes \ph$ and $\id_{M_n} \otimes \ps$
are isometrically similar.
\item\label{P_5916_PropOfISHm_CmpBd}
If $A$ and $B$ are matrix normed
and $\ph$ is completely bounded,
then so is~$\ps$.
\item\label{P_5916_PropOfISHm_Cmp}
If $A$ and $B$ are matrix normed
and $\ph$ is completely contractive,
then so is~$\ps$.
\item\label{P_5916_PropOfISHm_CmpIs}
If $A$ and $B$ are matrix normed
and $\ph$ is completely isometric,
then so is~$\ps$.
\end{enumerate}
\end{prp}

\begin{proof}
The only part requiring proof is~(\ref{P_5916_PropOfISHm_Mn}).
For this part,
we use Definition~\ref{D_4Y19_MatNAlg}(\ref{D_4Y19_MatNAlg_DSum})
to see that if $s \in B$ is an invertible isometry,
then so is $1 \otimes s \in M_n \otimes B$ for any $n \in \N$.
\end{proof}

\begin{prp}\label{P_5916_PropOfIS}
Let $B$ be a unital Banach algebra,
and let $e, f \in B$ be isometrically similar idempotents.
If $e$ is hermitian (Definition~\ref{D_6608_Herm}),
so is~$f$.
\end{prp}

\begin{proof}
The corresponding \hm{s} in Lemma~\ref{L_6612_HI}
are isometrically similar.
Apply Proposition~\ref{P_5916_PropOfISHm}(\ref{P_5916_PropOfISHm_Ct}).
\end{proof}

\begin{dfn}\label{D_4326_SpHmMd}
Let $p \in [1, \infty) \setminus \{ 2 \}$.
Let $A$ be a unital $\sm$-finitely representable
$L^p$~operator algebra, let $d \in \N$,
and let $\ph \colon \MP{d}{p} \to A$ be a \hm{}
(not necessarily unital).
We say that $\ph$ is {\emph{spatial}}
if $\ph (1)$ is a hermitian idempotent
(Definition~\ref{D_6608_Herm})
and $\ph$ is contractive.
The zero \hm{} is allowed
as a choice of $\ph$.
\end{dfn}

It isn't enough to merely require that $\ph$
be contractive.

\begin{exa}\label{E_4402_BadIdem}
Let $p \in [1, \infty) \setminus \{ 2 \}$.
Then there is a contractive \hm{}
$\ph \colon \C \to \MP{2}{p}$ which is not spatial.
To construct one,
let $e$ be the nonhermitian idempotent from
Lemma~\ref{L_5224_NmIdemp}.
Define
$\ph (\lambda) = \lambda e$ for $\lambda \in \C$.
Then $\ph$
is clearly contractive, but not spatial because $e$ is not hermitian.
\end{exa}

\begin{lem}\label{L_5917_MnSpIs}
Let $p \in [1, \infty) \setminus \{ 2 \}$,
let $d \in \N$,
and let $s \in \MP{d}{p}$.
Then $s$ is an invertible isometry \ifo{}
$s$ is a complex permutation matrix
(Definition~\ref{D_5205_CxPerm}).
\end{lem}

\begin{proof}
It is obvious that complex permutation matrices
are invertible isometries.
Conversely, assume that $s$ is an invertible isometry.
It follows from Lemma 6.16 of~\cite{PhLp1}
that $s$ is spatial,
and it is easily seen from the definitions
(Definition 6.3 and Definition 6.4 of~\cite{PhLp1})
that $s$ is a complex permutation matrix.
\end{proof}

The next lemma shows that every spatial homomorphism
between two matrix algebras is isometrically similar
to a block diagonal \hm.

\begin{lem}\label{L_4326_WhenSp}
Let $p \in [1, \infty) \setminus \{ 2 \}$.
Let $d, m \in \N$,
and let $\psi \colon \MP{d}{p} \to \MP{m}{p}$ be a \hm{}
(not necessarily unital).
Then \tfae:
\begin{enumerate}
\item\label{L_4326_WhenSp_Sp}
$\psi$ is spatial.
\item\label{L_4326_WhenSp_BlockD}
There exist $k \in \N$ with $0 \leq k d \leq m$
such that $\ps$ is isometrically similar to
the \hm{} $a \mapsto \diag (a, a, \ldots, a, 0)$,
the block diagonal matrix
in which $a$ occurs $k$ times and $0$ is the zero
element of $\MP{m - k d}{p}$.
\end{enumerate}
\end{lem}

\begin{proof}
It is easy to check that (\ref{L_4326_WhenSp_BlockD})
implies~(\ref{L_4326_WhenSp_Sp}).
So assume~(\ref{L_4326_WhenSp_Sp}).

First assume that $\psi$ is unital.
Then $m = k d$
for some $k \in \mathbb{Z}_{>0}$.
The implication from (4) to~(8) in Theorem~7.2 in~\cite{PhLp1}
provides
a \sfm{} $\YCN$ and a bijective isometry
\[
u \colon l^p_d \otimes_p L^p (Y, \nu) \to l^p_{k d}
\]
such that
for all $a \in M_d^p$ we have $\psi (a) = u (a \otimes 1) u^{-1}$.
The space $L^p (Y, \nu)$ must have dimension~$k$,
so it is isometrically isomorphic to $l^p_k$.
There is a bijection
\[
\{ 1, 2, \ldots, d \} \times \{ 1, 2, \ldots, k \}
\to \{ 1, 2, \ldots, m \}
\]
such that the corresponding
isomorphism of $L(l^p_d \otimes_p l^p_k)$ with $L(l^p_m)$
sends $a \otimes 1$ to $\diag (a, a, \ldots, a)$.
This allows us to identify $s = u^{-1}$
with an invertible isometry in $L (l^p_m) \cong \MP{m}{p}$
such that
$s \psi (a) s^{- 1} = \diag (a, a, \ldots, a)$
for all $a \in \MP{d}{p}$.

Now consider the general case.
Assume that $\psi$ is spatial.
So $\ps$ is contractive and $\psi (1)$ is spatial.
By Lemma~\ref{L_4326_CharSpI},
there exists a measurable subset $E \subset \{1, 2, \ldots, m \}$
such that $\psi (1)$ is multiplication by $\chi_E$ on $l_m^p$.
By conjugating by a permutation matrix,
we can assume that
$E = \{1, 2, \ldots, n \}$ for some $n \in \{ 1, 2, \ldots, m \}$.
Identify $l^p_n$ with $l^p (E) \subset l^p_m$,
and let $\io \colon L (l^p_n) \to L (l^p_m)$
be $\io (b) = b \oplus 0$ for $b \in L (l^p_n)$.
(In matrix form, this is $\io (b) = \diag (b, 0)$.)
There is a homomorphism $\ph \colon L (l_d^p) \to L (l_n^p)$
such that $\ps (a) = \io ( \ph (a)) = \ph (a) \oplus 0$
for all $a \in L (l_d^p)$,
and $\ph$ is a unital homomorphism from
$M_d^p$ to $M_n^p$ which is contractive
by Corollary~\ref{C_5125_CutAlg}.
By the case done above, there exist
$k \in \mathbb{Z}_{>0}$ such that $n = k d$
and an invertible isometry
$s_0 \in M_n^p$ such that
$s_0 \ph (a) s_0^{- 1} = \diag (a, a, \ldots, a)$
for all $a \in L (l_d^p)$.
Then $s = \diag (s_0, 1)$,
with $1$ being the identity of $M_{m - n}^p$,
is an invertible isometry
such that for all $a \in L (l_d^p)$ we have
$s \psi (a) s^{- 1} = \diag (a, a, \ldots, a, 0)$.
\end{proof}

To define a spatial $L^p$~AF algebra, we
need to first define its building blocks, the
spatial semisimple finite dimensional
$L^p$~operator algebras.

\begin{dfn}\label{D_4326_SFDLp}
Let $p \in [1, \infty) \setminus \{ 2 \}$.
A Banach algebra $A$ is called a
{\emph{spatial semisimple finite dimensional
$L^p$~operator algebra}}
if there are $N \in \N$ and
$d_1, d_2, \ldots, d_N \in \N$ such
that $A$ is isometrically isomorphic to
${\textdisp{ \bigoplus_{k = 1}^N \MP{d_k}{p} }}$, endowed
with the norm as in Definition~\ref{D_4326_DSumNorm}.
\end{dfn}

\begin{rmk}\label{R_4326_SSFD}
Let $p \in [1, \infty) \setminus \{ 2 \}$.
To simplify the notation in our proofs,
if $A$ is a
spatial semisimple finite dimensional $L^p$~operator algebra,
we will omit the isometric isomorphism and simply write
${\textdisp{A = \bigoplus_{k = 1}^N \MP{d_k}{p}}}$
with $N, d_1, d_2, \ldots, d_N \in \mathbb{Z}_{>0}$.
\end{rmk}

\begin{lem}\label{L_7629_Quot}
Let $p \in [1, \infty) \setminus \{ 2 \}$.
Let $A$ be
a spatial semisimple finite dimensional $L^p$~operator algebra,
and let $J \S A$ be an ideal.
Then $A / J$,
with the quotient norm,
is a spatial semisimple finite dimensional $L^p$~operator algebra.
\end{lem}

\begin{proof}
Use the notation in Remark~\ref{R_4326_SSFD}.
Then there is a subset $S \S \{ 1, 2, \ldots, N \}$
such that,
as an algebra,
${\textdisp{A / J = \bigoplus_{k \in S} \MP{d_k}{p}}}$.
The quotient norm
agrees with the norm on ${\textdisp{ \bigoplus_{k \in S} \MP{d_k}{p}}}$
by Lemma~\ref{L_6408_QtOfSum}.
\end{proof}

\begin{lem}\label{L_5917_SSFDSpIs}
Let $p \in [1, \infty) \setminus \{ 2 \}$,
and let ${\textdisp{A = \bigoplus_{k = 1}^N \MP{d_k}{p}}}$
be a spatial semisimple finite dimensional $L^p$~operator algebra.
Let $s = (s_1, s_2, \ldots, s_N) \in A$.
Then $s$ is an invertible isometry \ifo{}
$s_k$ is a complex permutation matrix
for $k = 1, 2, \ldots, N$.
\end{lem}

\begin{proof}
This is immediate from Lemma~\ref{L_5917_MnSpIs}
and the definition of the norm on~$A$.
\end{proof}

\begin{lem}\label{L_4326_FDIsLp}
Let $p \in [1, \infty) \setminus \{ 2 \}$,
and let $A$ be a spatial semisimple finite
dimensional $L^p$~operator algebra.
Then $A$ is an $L^p$~operator algebra
with unique $L^p$~operator matrix norms,
obtained by combining
Definition~\ref{D_4326_DSumNorm} and Definition~\ref{D_4Y19_LpMN}.
\end{lem}

\begin{proof}
That $A$ is an $L^p$~operator algebra
follows from Lemma~\ref{L_4Y17_SumIsLp}.
That $A$ has unique $L^p$~operator matrix norms
follows from Corollary~\ref{P_5124_MnUniq}
and Lemma~\ref{P_5124_UniqMatDSum}.
\end{proof}

The maps we will consider between
spatial semisimple finite dimensional
$L^p$~operator algebras are the
spatial homomorphisms.

\begin{dfn}\label{D_4326_SpHmFD}
Let $p \in [1, \infty) \setminus \{ 2 \}$,
and let
${\textdisp{ A = \bigoplus_{k = 1}^N \MP{d_k}{p} }}$
be a spatial semisimple finite dimensional
$L^p$~operator algebra.
Let $B$ be a $\sm$-finitely representable unital $L^p$~operator
algebra, and let $\ph \colon A \to B$ be a \hm.
We say that $\ph$ is
{\emph{spatial}}
if for $k = 1, 2, \ldots, N$,
the restriction of $\ph$ to the summand ${\MP{d_k}{p}}$
is spatial in the sense of Definition~\ref{D_4326_SpHmMd}.
\end{dfn}

\begin{lem}\label{L_4326_CharSpHmFDalg}
Let $p \in [1, \infty) \setminus \{ 2 \}$,
and let $A$ be a spatial semisimple finite
dimensional $L^p$~operator algebra.
Let $B$ be a $\sm$-finitely representable
unital $L^p$~operator algebra,
and let $\ph \colon A \to B$ be a \hm.
Then $\ph$ is spatial \ifo{}
$\ph (1)$ is a hermitian idempotent
(Definition~\ref{D_6608_Herm})
and $\ph$ is contractive.
\end{lem}

\begin{proof}
We can assume without
loss of generality that $B$
is a unital subalgebra of $L (L^p (Y, \nu))$
for some \sfm{} $\YCN$.
Assume also that
${\textdisp{ A = \bigoplus_{l = 1}^N \MP{d_l}{p} }}$.
For $k = 1, 2, \ldots, N$, let
${\textdisp{ \io_k
 \colon M_{d_k}^p \to \bigoplus_{l = 1}^N \MP{d_l}{p} }}$
be the inclusion of the $k$-th summand into $A$. Let
$\rh \colon L^{\infty} (Y, \nu) \to L (L^p (Y, \nu))$
be the representation by multiplication operators.

Suppose that $\ph$ is spatial.
For $k = 1, 2, \ldots, N$, the
\hm{} $\ph|_{M_{d_k}}$ is spatial, so
$e_k = \ph (\io_k (1_{M_{d_k}}))$ is a hermitian
idempotent.
The idempotents $e_1, e_2, \ldots, e_N$
are clearly orthogonal, so
Corollary \ref{C_4Y17_Orth}(\ref{C_4Y17_Orth_Sum})
implies that ${\textdisp{ \ph (1) = \sum_{k = 1}^N e_k }}$
is a hermitian idempotent.

Corollary \ref{C_4Y17_Orth}(\ref{C_4Y17_Orth_Disj})
provides disjoint measurable sets
$E_1, E_2, \ldots, E_N \subset Y$
such that $e_k = \rh ({{\ch}}_{{E_k}})$
for $k = 1, 2, \ldots, N$.
Set ${\textdisp{ E = \bigcup_{k = 1}^N E_k }}$.
We can identify $L^p (Y, \nu)$
with the $L^p$~direct sum
${\textdisp{ L^p (Y \setminus E, \, \nu) \oplus_p
 \bigoplus_{k = 1}^N L^p (E_k, \nu) }}$.
For $l = 1, 2, \ldots, N$,
let $\ph_l \colon M_{d_l}^p \to L (L^p (Y, \nu))$
be
$\ph_l (a) = \rh ( {{\ch}}_{{E_l}}) \ph ( \io_l (a))
           \rh ( {{\ch}}_{{E_l}})$
for $a \in M_{d_l}^p$.
If $a_k \in M_{d_k}^p$ for $k = 1, 2, \ldots, N$,
then
\begin{align*}
\ph (a_1, a_2, \ldots, a_N)
& = \sum_{k = 1}^N \ph ( \io_k (a_k))
  = \sum_{k = 1}^N \ph ( \io_k (1_{M_{d_k}})) \ph ( \io_k (a_k))
     \ph ( \io_k (1_{M_{d_k}}))
\\
& = \sum_{k = 1}^N \rh ( {{\ch}}_{{E_k}}) \ph ( \io_k (a_k))
           \rh ( {{\ch}}_{{E_k}})
  = \sum_{k = 1}^N \ph_k (a_k).
\end{align*}
Since the sets $E_j$ are disjoint,
\[
\| \ph (a_1, a_2, \ldots, a_N) \|
 = \max_{1 \leq k \leq N} \| \ph_k (a_k) \|
 \leq \max_{1 \leq k \leq N} \|a_k \|
 = \| (a_1, a_2, \ldots, a_N) \|,
\]
so $\ph \colon A \to B$ is contractive.

Conversely assume that $\ph$ is contractive
and $\ph (1)$ is a hermitian idempotent.
For $k = 1, 2, \ldots, N$, it is obvious from
Definition~\ref{D_4326_DSumNorm}
and Lemma~\ref{L_6612_HI} that
$\io_k (1_{M_{d_k}})$ is a hermitian idempotent
in ${\textdisp{ \bigoplus_{k = 1}^N \MP{d_k}{p} }}$.
Therefore $\ph ( \io_k (1_{M_{d_k}}))$ is a
hermitian idempotent in~$B$ by
Lemma~\ref{L_4Y18_Funct}.
Also, $\ph \circ \io_k$ is
contractive because $\io_k$ and $\ph$ are.
So $\ph|_{M_{d_k}}$ is spatial.
\end{proof}

\begin{cor}\label{C_5208_CompSpatial}
Let $p \in [1, \infty) \setminus \{ 2 \}$.
Let $A$
be a spatial semisimple finite dimensional
$L^p$~operator algebra, let $B$ and $C$ be unital
$\sm$-finitely representable $L^p$~operator algebras,
let $\ph \colon A \to B$ be a spatial \hm,
and let $\ps \colon B \to C$ be a contractive \hm{}
such that $\ps (1)$ is a hermitian idempotent in~$C$.
Then $\ps \circ \ph$ is spatial.
\end{cor}

\begin{proof}
Lemma~\ref{L_4326_CharSpHmFDalg}
implies that $\ps \circ \ph$ is contractive.
It follows from
Lemma~\ref{L_4326_CharSpHmFDalg}
and Lemma~\ref{L_4Y18_Funct} that
$(\ps \circ \ph) (1)$ is a hermitian
idempotent in~$C$.
So $\ps \circ \ph$ is spatial by
Lemma~\ref{L_4326_CharSpHmFDalg}.
\end{proof}

\begin{cor}\label{C_5917_SpIsoSim}
Let $p \in [1, \infty) \setminus \{ 2 \}$.
Let $A$
be a spatial semisimple finite dimensional
$L^p$~operator algebra, let $B$ be a unital
$\sm$-finitely representable $L^p$~operator
algebra, and let $\ph, \ps \colon A \to B$
be isometrically similar \hm{s}.
Then $\ph$ is spatial \ifo{} $\ps$ is spatial.
\end{cor}

\begin{proof}
Use Lemma~\ref{L_4326_CharSpHmFDalg},
Proposition~\ref{P_5916_PropOfISHm}(\ref{P_5916_PropOfISHm_Ct}),
and
Lemma~\ref{P_5916_PropOfIS}.
\end{proof}

The following definition is standard,
but is given here for reference.

\begin{dfn}\label{D_5208_Mult}
Let ${\textdisp{A = \bigoplus_{j = 1}^M M_{c_j} }}$
be a finite direct sum of full matrix algebras.
\begin{enumerate}
\item\label{D_5208_Mult_1Sum}
Let $d \in \N$,
and let $\ph \colon A \to M_d$ be a \hm.
Then for $j = 1, 2, \ldots, M$
the {\emph{$j$-th partial multiplicity}} of $\ph$
is defined to be
\[
m_j (\ph) = \rank (\ph (1_{M_{c_j}} ) ) / c_j.
\]
\item\label{D_5208_Mult_Gen}
Let ${\textdisp{ B = \bigoplus_{k = 1}^N M_{d_k} }}$ be
another finite direct sum of full matrix algebras,
and let $\ph \colon A \to B$ be a \hm.
For $j = 1, 2, \ldots, M$ and $k = 1, 2, \ldots, N$,
we denote by $m_{k, j} (\ph)$
the $j$-th partial multiplicity of the composition of $\ph$
with the projection map $B \to M_{d_k}$.
We call $m (\ph) = ( m_{k, j} (\ph) )_{k, j}$
the {\emph{partial multiplicity matrix}} of~$\ph$.
We use analogous notation for direct sums indexed by finite sets
not of the form $\{ 1, 2, \ldots, M \}$.
\end{enumerate}
\end{dfn}

Next we define block diagonal homomorphisms
between finite direct sums of full matrix algebras.

\begin{dfn}\label{D_4326_BlockDiag}
Let $p \in [1, \infty) \setminus \{ 2 \}$.
Let ${\textdisp{ A = \bigoplus_{j = 1}^M M_{c_j} }}$
be a finite direct sum of full matrix algebras.
\begin{enumerate}
\item\label{D_4326_BlockDiag_1}
A unital homomorphism $\ph \colon A \to M_d$ is said to
be {\emph{block diagonal}}
if there exist $n \geq 1$
and $r (1), r (2), \ldots, r (n) \in \{1, 2, \ldots, M \}$,
satisfying ${\textdisp{ \sum_{k = 1}^n c_{r (k)} = d }}$, such that
\[
\ph (a_1, a_2, \ldots, a_M)
 = \left(\begin{matrix}
  a_{r (1)} & 0 & 0 & 0 & \cdots & 0 \\
  0  & a_{r (2)} & 0 & 0 & \cdots & 0 \\
 0 & 0 & a_{r (3)} & 0 & \cdots & 0 \\
 \vdots & \vdots & \vdots & \vdots&  & \vdots \\
 0 &  0 & 0 &  0 & \cdots & a_{r (n)}
\end{matrix} \right).
\]
\item\label{D_4326_BlockDiag_2}
A nonunital homomorphism  $\ph \colon A \to M_d$
is {\emph{block diagonal}} if its unitization
${\textdisp{ \bigoplus_{j = 1}^M
    M_{c_j} \oplus \mathbb{C} \to M_d }}$
is block diagonal.
\item\label{D_4326_BlockDiag_3}
Let ${\textdisp{ B = \bigoplus_{j = 1}^N M_{d_j} }}$ be
a finite direct sum of full matrix algebras.
Then a homomorphism $\ph \colon A \to B$ is
{\emph{block diagonal}} if for $k = 1, 2, \ldots, N$
the homomorphism $\ph_k \colon A \to M_{d_k}$,
given by the composition of $\ph$ and the projection map
$B \to M_{d_k}$, is block diagonal.
\end{enumerate}
\end{dfn}

We list some properties of block diagonal homomorphisms.

\begin{lem}\label{L_5208_Summary}
Let $p \in [1, \I) \setminus \{ 2 \}$.
\begin{enumerate}
\item\label{L_5208_Summary_CMult}
Let $A$, $B$, and $C$ be finite direct sums of full matrix algebras,
and let $\ph \colon A \to B$ and $\ps \colon B \to C$
be \hm{s}.
Then $m (\ps \circ \ph) = m (\ps) m (\ph)$.
\item\label{L_5208_Summary_BDImpSp}
Let $A$ and $B$ be
spatial semisimple finite dimensional $L^p$~operator algebras.
Then every block diagonal \hm{} $\ph \colon A \to B$ is spatial.
\item\label{L_5208_Summary_MDMult}
If $\ph$ is as in
Definition \ref{D_4326_BlockDiag}(\ref{D_4326_BlockDiag_1}), then
\[
m_j (\ph)
 = \card \big( \big\{k \in \{1, 2, \ldots, n \} \colon r (k) = j \big\}
  \big).
\]
\item\label{L_5208_Summary_Exist}
Let ${\textdisp{ A = \bigoplus_{j = 1}^M M_{c_j} }}$
and ${\textdisp{ B = \bigoplus_{k = 1}^N M_{d_k} }}$ be
finite direct sums of full matrix algebras,
and let $m$ be an $N \times M$ matrix with entries in~$\Nz$.
Then \tfae:
\begin{enumerate}
\item\label{L_5208_Summary_Exist_BD}
There exists a block diagonal \hm{} $\ph \colon A \to B$
such that $m (\ph) = m$.
\item\label{L_5208_Summary_Exist_Hom}
There exists a \hm{} $\ph \colon A \to B$
such that $m (\ph) = m$.
\item\label{L_5208_Summary_Exist_Ineq}
For $k = 1, 2, \ldots, N$,
we have ${\textdisp{ \sum_{j = 1}^M m_{k, j} c_j \leq d_k }}$.
\end{enumerate}
\item\label{L_5208_Summary_BDComp}
The composition of two block diagonal \hm{s} is block
diagonal.
\item\label{L_5208_Summary_DSum}
Let $A_1, A_2, B_1, B_2$ be finite direct sums of full matrix algebras,
and let $\ph_1 \colon A_1 \to B_1$ and $\ph_2 \colon A_2 \to B_2$
be block diagonal \hm{s}.
Then $\ph_1 \oplus \ph_2 \colon A_1 \oplus A_2 \to B_1 \oplus B_2$
is block diagonal.
\item\label{L_5208_Summary_TensMd}
Let ${\textdisp{ A = \bigoplus_{j = 1}^M M_{c_j} }}$
and ${\textdisp{ B = \bigoplus_{k= 1}^N M_{d_k} }}$ be
finite direct sums of full matrix algebras,
let $\ph \colon A \to B$
be a block diagonal \hm,
and let $r \in \N$.
Make the identifications
${\textdisp{ M_r \otimes A = \bigoplus_{j = 1}^M M_{r c_j} }}$
and
${\textdisp{ M_r \otimes B = \bigoplus_{k= 1}^N M_{r d_k} }}$,
by using on each summand the isomorphism
$\te_{\sm}$ of Definition~\ref{D_5205_MmMnMmn}
with $\sm$ taken to be the standard choice of bijection
as given there.
Then $\id_{M_r} \otimes \ph$ is block diagonal.
\item\label{L_5208_Summary_BDCC}
Let $A$ and $B$ be
spatial semisimple finite dimensional $L^p$~operator
algebras, and let $\ph \colon A \to B$ be
block diagonal.
Then $\ph$ is completely contractive.
\end{enumerate}
\end{lem}

\begin{proof}
We first prove~(\ref{L_5208_Summary_BDImpSp}).
Write ${\textdisp{ B = \bigoplus_{k = 1}^N M_{d_k}^p }}$,
and for
$k = 1, 2, \ldots, N$ let
$\pi_k \colon B \to M_{d_k}^p$ be the projection map.
Block diagonal maps to $M_{d_k}^p$ are clearly contractive,
so $\pi_k \circ \ph$ is contractive.
Thus $\ph$ is contractive by Lemma~\ref{L_5128_SupIn}.
For $k = 1, 2, \ldots, N$,
the matrix
$(\pi_k \circ \ph) (1)$ is diagonal with entries in $\{0, 1\}$.
So $(\pi_k \circ \ph) (1)$ is a hermitian idempotent
by Corollary~\ref{C_5128_SpIsDiag}.
Now $\ph (1)$ is a hermitian idempotent by Lemma~\ref{L_5128_SIDSum}.
Use Lemma~\ref{L_4326_CharSpHmFDalg}.

Part~(\ref{L_5208_Summary_BDCC})
follows from
part~(\ref{L_5208_Summary_TensMd})
and part~(\ref{L_5208_Summary_BDImpSp}).

Everything else is either well known or immediate.
\end{proof}

\begin{lem}\label{L_4326_BlkDiagHm}
Let $p \in [1, \infty) \setminus \{ 2 \}$,
let ${\textdisp{ A = \bigoplus_{j = 1}^L \MP{c_j}{p} }}$
and ${\textdisp{ B = \bigoplus_{k = 1}^N \MP{d_k}{p} }}$
be spatial semisimple finite dimensional $L^p$~operator
algebras, and let $\ph \colon A \to B$ be a \hm.
Then
$\ph$ is spatial \ifo{} $\ph$ is isometrically similar
to a block diagonal \hm.
\end{lem}

\begin{proof}
If $\ph$ is isometrically similar to a block diagonal \hm,
then $\ph$ is spatial by
Lemma \ref{L_5208_Summary}(\ref{L_5208_Summary_BDImpSp})
and Corollary~\ref{C_5917_SpIsoSim}.

Conversely, assume that $\ph$ is spatial. Since the projection
map $\pi_k \colon B \to M_{d_k}^p$ is contractive and
$\pi_k(1)$ is a hermitian idempotent,
Corollary~\ref{C_5208_CompSpatial}
implies that $\pi_k\circ\ph$ is spatial. Therefore, it is enough to
prove the claim when $B = \MP{d}{p}$ for some $d \in \N$.

For $j = 1, 2, \ldots, M$ let $\io_j \colon \MP{c_j}{p} \to A$
be the inclusion map.
Since $\ph \circ \io_j$ is spatial
(by Lemma~\ref{L_4326_CharSpHmFDalg} and
Corollary~\ref{C_5208_CompSpatial}),
it follows from Corollary~\ref{C_4Y17_Orth}
that there are disjoint subsets
\[
E_1, E_2, \ldots, E_M \subset \{ 1, 2, \ldots, d \}
\]
such that $(\ph \circ \io_j) (1_{\MP{c_j}{p}})$
is multiplication by $\ch_{E_j}$
for $j = 1, 2, \ldots, M$.
Let $\rh$ be the representation of $C ( \{ 1, 2, \ldots, d \} )$
on $l_d^{p}$ by multiplication operators.
Set $d_0 = 0$ and
choose a permutation $\sm$ of $\{ 1, 2, \ldots, d \}$
and numbers $d_1, d_2, \ldots, d_M \in \{ 1, 2, \ldots, d \}$
such that for $j = 1, 2, \ldots, M$ we have
$\sm (E_j) = (d_{j - 1}, \, d_j] \cap \Z$.
Let $s_0 \in M_d^p$ be the corresponding permutation
matrix, satisfying
$s_0 \rh ({{\ch}}_{{E_j}} ) s_0^{-1} = \rh ( {{\ch}}_{{\sm (E_j)}} )$
for $j = 1, 2, \ldots, M$.
Corollary~\ref{C_5917_SpIsoSim} implies
that the map $a \mapsto s_0 \ph (a) s_0^{- 1}$ is spatial.

For $1 \leq j \leq M$ make the obvious identification
\[
\rh ( {{\ch}}_{ (d_{j - 1}, \, d_j] \cap \Z} ) \MP{d}{p}
   \rh ( {{\ch}}_{ (d_{j - 1}, \, d_j] \cap \Z} )
  = \MP{d_j - d_{j - 1}}{p}.
\]
Since
\[
s_0 (\ph \circ \io_j) (1) s_0^{-1}
 = \rh ( {{\ch}}_{ (d_{j - 1}, \, d_j] \cap \Z} ),
\]
by Corollary~\ref{C_5125_CutAlg} the formula
$s_0 (\ph \circ \io_j) (\cdot) s_0^{-1}$ defines
a contractive unital \hm{}
$\ps_j \colon \MP{c_j}{p} \to \MP{d_j - d_{j - 1}}{p}$.
It follows from Lemma~\ref{L_4326_WhenSp}
that there is a complex permutation matrix
$s_j \in \MP{d_j - d_{j - 1}}{p}$
such that $a \mapsto s_j \ps_j (a) s_j^{-1}$
is a block diagonal \hm{}
from $\MP{c_j}{p}$ to $\MP{d_j - d_{j - 1}}{p}$ for
$1\leq j\leq M$.

Set $s = [\diag (s_1, s_2, \ldots, s_M, 1_{d - d_M})] \cdot s_0$,
which is a complex permutation matrix in $\MP{d}{p}$.
Since
\[
s_0 \ph (a_1, a_2, \ldots, a_M) s_0^{-1}
 = \diag \big( \ps_1 (a_1), \, \ps_2 (a_2), \, \ldots,
  \, \ps_M (a_M), \, 0_{d - d_M} \big)
\]
for
${\textdisp{ a = (a_1, a_2, \ldots, a_M)
 \in \bigoplus_{j = 1}^L \MP{c_j}{p} }}$,
it follows that
$a \mapsto s \ph (a) s^{- 1}$ is block diagonal.
\end{proof}

\section{Spatial $L^p$~AF algebras}\label{Sec_Init}

We define spatial $L^p$~AF algebras and show that any spatial
$L^p$~AF algebra is a separable nondegenerately representable
$L^p$~operator algebra.

\begin{dfn}\label{D_4327_SpSys}
Let $p \in [1, \infty) \setminus \{ 2 \}$.
A {\emph{spatial $L^p$~AF direct system}}
is a contractive direct system with index set~$\Nz$
(that is, a pair
$\big( (A_m)_{m \in \Nz}, \, ( \ph_{n, m} )_{0 \leq m \leq n} \big)$
as in Definition~\ref{D_4310_DirSys}),
which satisfies the following additional conditions:
\begin{enumerate}
\item \label{D_4327_SpSys_Algs}
For every $m \in \Nz$,
the algebra $A_m$ is a
spatial semisimple finite dimensional $L^p$~operator algebra
(Definition~\ref{D_4326_SFDLp}).
\item \label{D_4327_SpSys_Maps}
For all $m, n \in \Nz$ with $m \leq n$,
the map $\ph_{n, m}$ is a spatial \hm{}
(Definition~\ref{D_4326_SpHmFD}).
\end{enumerate}
We further say that a Banach algebra $A$
is a {\emph{spatial $L^p$~AF algebra}}
if it is isometrically isomorphic to the direct limit
of a spatial $L^p$~AF direct system.
\end{dfn}

\begin{dfn}\label{D_4615_MatNormsLpAF}
Let $p \in [1, \infty) \setminus \{ 2 \}$.
Let
$\big( (A_m)_{m \in \Nz}, \, ( \ph_{n, m} )_{0 \leq m \leq n} \big)$
be a spatial $L^p$~AF direct system
(Definition~\ref{D_4327_SpSys}).
We make ${\textdisp{ A = \dirlim_m (A_m, \ph_{n, m}) }}$
into a matricial $L^p$~operator algebra
via Lemma~\ref{L_4326_FDIsLp}
and Theorem~\ref{T_5125_LpMatDLim}.
\end{dfn}

The matrix norms on~$A$ a priori depend on how $A$ is
realized as a direct limit.
We will show in Theorem
\ref{C_6416_GenUniqMN} that in fact
they are independent of the realization.

\begin{lem}\label{L_6408_MatSys}
Let $p \in [1, \infty) \setminus \{ 2 \}$.
Let $r \in \N$ and let
$\big( (A_m)_{m \in \Nz}, \, ( \ph_{n, m} )_{0 \leq m \leq n} \big)$
be a spatial $L^p$~AF direct system.
Then
\[
\big( (M_r ( A_m) )_{m \in \Nz},
   \, ( \id_{M_r} \otimes \ph_{n, m} )_{0 \leq m \leq n} \big)
\]
is a spatial $L^p$~AF direct system.
\end{lem}

\begin{proof}
Using Definition~\ref{D_4Y19_LpMN}
and Definition~\ref{D_5205_MatNormMm},
for any $d \in \N$
we see that $M_r ( M_d^p )$
is isometrically isomorphic to $M_{r d}^p$,
via a map as in Definition~\ref{D_5205_MmMnMmn}.
Therefore $M_r ( A_m)$ is a
spatial semisimple finite dimensional $L^p$~operator algebra
for all $m \in \N$.
Lemma~\ref{L_4326_BlkDiagHm} implies that $\ph_{n, m}$
is isometrically similar to a block diagonal homomorphism.
It follows from
Lemma \ref{L_5208_Summary}(\ref{L_5208_Summary_TensMd})
and
Proposition \ref{P_5916_PropOfISHm}(\ref{P_5916_PropOfISHm_Mn})
that the maps $\id_{M_r} \otimes \ph_{n, m}$
are isometrically similar to block diagonal homomorphisms.
Now use Lemma~\ref{L_4326_BlkDiagHm}.
\end{proof}

\begin{cor}\label{C_6409_Mat}
Let $p \in [1, \infty) \setminus \{ 2 \}$.
Let
$\big( (A_m)_{m \in \Nz}, \, ( \ph_{n, m} )_{0 \leq m \leq n} \big)$
be a spatial $L^p$~AF direct system.
Let ${\textdisp{ A = \dirlim_m (A_m, \ph_{n, m}) }}$
be the direct limit,
equipped
with the matricial $L^p$~operator algebra structure
of Definition~\ref{D_4615_MatNormsLpAF}.
Let $r \in \N$.
Then $M_r (A)$ is a spatial $L^p$~AF algebra.
\end{cor}

\begin{proof}
This is immediate from Lemma~\ref{L_6408_MatSys}.
\end{proof}

\begin{lem}\label{L_6408_DSumSys}
Let $p \in [1, \infty) \setminus \{ 2 \}$.
Let $N \in \N$ and
for $k = 1, 2, \ldots, N$
let
$\big( (A_m^{(k)})_{m \in \Nz},
 \, ( \ph_{n, m}^{(k)} )_{0 \leq m \leq n} \big)$
be a spatial $L^p$~AF direct system
(Definition~\ref{D_4327_SpSys}).
Then
\[
\left( \left( \bigoplus_{k = 1}^N A_m^{(k)} \right)_{m \in \Nz},
 \, \left( \bigoplus_{k = 1}^N \ph_{n, m}^{(k)}
          \right)_{0 \leq m \leq n}
  \right)
\]
is a spatial $L^p$~AF direct system.
\end{lem}

\begin{proof}
Obviously ${\textdisp{\bigoplus_{k = 1}^N A_m^{(k)} }}$
is a spatial semisimple finite dimensional $L^p$~operator algebra
for all $m \in \N$.
By Lemma~\ref{L_4326_BlkDiagHm},
a direct system
of spatial semisimple finite dimensional $L^p$~operator algebras
is a spatial $L^p$~AF direct system
\ifo{}
its maps are all isometrically similar to block diagonal maps.
It follows from Lemma \ref{L_5208_Summary}(\ref{L_5208_Summary_DSum})
that the direct sum of maps
isometrically similar to block diagonal maps
is again isometrically similar to a block diagonal map.
\end{proof}

\begin{cor}\label{C_6409_DSum}
Let $p \in [1, \infty) \setminus \{ 2 \}$.
Then the direct sum of
finitely many spatial $L^p$~AF algebras
is again a spatial $L^p$~AF algebra.
\end{cor}

\begin{proof}
This is immediate from Lemma~\ref{L_6408_DSumSys}.
\end{proof}

\begin{dfn}\label{D_5204_CcN}
Let $A$ be a Banach algebra,
and let $e = (e_n)_{n \in \N}$ be a sequence
of idempotents in~$A$ which is nondecreasing,
that is, for $n \in \N$ we have $e_n \leq e_{n + 1}$
in the sense of Definition~\ref{D_4Y17_Idemp}.
Set $e_0 = 0$
(by convention),
and let $\te_e \colon C_{\mathrm{c}} (\N) \to A$
be the unique \hm{}
such that $\te_e ({{\ch}}_{ \{ n \} } ) = e_n - e_{n - 1}$
for all $n \in \N$.
We equip $C_{\mathrm{c}} (\N)$ with the norm
$\| \cdot \|_{\infty}$, and when we refer to $\| \te_e \|$,
or demand that $\te_e$ be contractive or bounded,
we use this norm.
\end{dfn}

\begin{prp}\label{P_5204_SpCcN}
Let $p \in [1, \infty) \setminus \{ 2 \}$.
Let $A$ be a separable $L^p$~operator algebra,
and let $e = (e_n)_{n \in \N}$ and $\te_e$ be as
in Definition~\ref{D_5204_CcN}.
Assume that this
sequence is an approximate identity for~$A$,
and that $\te_e$ is contractive.
Then there are a \sft{} \msp{} $\YCN$,
with $L^p (Y, \nu)$ separable,
and an isometric nondegenerate \rpn{}
$\pi \colon A \to L (L^p (Y, \nu))$,
such that $\pi (e_n)$ is a hermitian idempotent
in $L (L^p (Y, \nu))$
for all $n \in \N$.
\end{prp}

\begin{proof}
Use Proposition~\ref{P_4Y19_SepImpSepRep}
to find a \sfm{} $\XBM$
such that $L^p (X, \mu)$
is separable and an isometric representation
$\rh$ of $A$ on $L^p (X, \mu)$.
It is clear from contractivity of $\te_e$ and Lemma~\ref{L_6612_HI}
that if $n \in \N$
then $\| e_n \| = 1$ and
(taking $e_0 = 0$)
that $e_{n - 1}$ is a hermitian idempotent in $e_n A e_n$.

Apply Lemma~\ref{L_6808_IncIdem} to the idempotents
$\rh (e_n)$ for $n \in \N$.
In the rest of the proof,
we use the notation of Lemma~\ref{L_6808_IncIdem}.
Set $E = e L^p (X, \mu)$.
Then $s$ is an invertible isometry
from $L^p (Y, \nu)$ to~$E$.
Moreover, the map $a \mapsto \rh (a) |_E$
defines a \hm{} from $A$ to $L (E)$,
with $\| \rh (a) |_E \| = \| \rh (a) \|$
by Lemma~\ref{L_5125_NormCut}.
Now the \rpn{} $\pi \colon A \to L (L^p (Y, \nu) )$,
defined by $\pi (a) = s^{-1} [ \rh (a) |_E ] s$,
is nondegenerate and isometric.
Moreover, for $n \in \N$,
the operator $\pi (e_n)$ is multiplication by the characteristic
function of ${\textdisp{ \bigcup_{k = 1}^n Y_n }}$
and is hence a hermitian idempotent in $L (L^p (Y, \nu) )$.
\end{proof}

Under the hypotheses of Proposition~\ref{P_5204_SpCcN},
an $L^p$~operator algebra has a canonical norm on
its unitization.

\begin{prp}\label{P_5209_NormOnUnit}
Let $p \in [1, \infty) \setminus \{ 2 \}$.
Let $A$ be a separable nonunital $L^p$~operator
algebra, and let $e = (e_n)_{n \in \N}$ be as in
Definition~\ref{D_5204_CcN}.
Assume that this sequence is
an approximate identity for~$A$, and that the
\hm{} $\te_e$ of Definition~\ref{D_5204_CcN} is contractive.
Then there is a unique norm $\| \cdot \|$
on the unitization $A^{+}$ of~$A$
satisfying the following conditions:
\begin{enumerate}
\item\label{P_5209_NormOnUnit_Agree}
$\| \cdot \|$ agrees with the given norm on $A \subset A^{+}$.
\item\label{P_5209_NormOnUnit_Eq}
$\| \cdot \|$ is equivalent to the usual norm on
the unitization.
\item\label{P_5209_NormOnUnit_Lp}
$A^{+}$ is an $L^p$~operator algebra.
\item\label{P_5209_NormOnUnit_Spatial}
Identify $C ( \N \cup \{ \infty \})$
with $C_0 (\N)^{+}$,
and give it the usual supremum norm on
$C ( \N \cup \{ \infty \})$.
Let $\te_e^{+} \colon C ( \N \cup \{ \infty \}) \to A^{+}$
be the unitization of~$\te_e$.
Then $\te_e^{+}$ is contractive.
\end{enumerate}
\end{prp}

\begin{proof}
We first prove existence.
Let $\pi \colon A \to L (L^p (Y, \nu))$
be as in Proposition~\ref{P_5204_SpCcN}.
Extend this \hm{} to a
\hm{} $\pi^{+} \colon A^{+} \to L (L^p (Y, \nu))$.
Then $\pi^{+}$ is injective because $A$ is not unital.
Define $\| a \| = \| \pi^{+} (a) \|$ for $a \in A^{+}$.
Conditions (\ref{P_5209_NormOnUnit_Agree}),
(\ref{P_5209_NormOnUnit_Eq}),
and~(\ref{P_5209_NormOnUnit_Lp}) are immediate.
It remains to prove condition~(\ref{P_5209_NormOnUnit_Spatial}).

By density,
it suffices to prove that for $n \in \N$,
if we set
\[
K = \{ n + 1, \, n + 2, \, \ldots, \, \infty \},
\]
and take
any function $f \in C ( \N \cup \{ \infty \})$ vanishing on~$K$
and any $\ld \in \C$,
then
\begin{equation}\label{Eq_5208_MaxN}
\| (\pi^{+} \circ \te_e^{+}) (f + \ld {{\ch}}_{{K}} )\|
 \leq \max ( \| f \|, \, | \ld | ).
\end{equation}
Since Proposition~\ref{P_5204_SpCcN} implies that $\pi (e_n)$
is a hermitian idempotent, by Lemma~\ref{L_4326_CharSpI} there
is a \mb{} subset $E \subset Y$ such that $\pi (e_n)$ is multiplication
by~$\ch_E$ on $L^p (Y, \nu))$.
Then $(\pi^{+} \circ \te_e^{+}) (f)$
acts on $L^p (E, \nu)$ and is zero on $L^p (Y \setminus E, \, \nu)$,
while $(\pi^{+} \circ \te_e^{+}) (\ld{{\ch}}_K)$
is multiplication by $\ld$ on $L^p (Y \setminus E, \, \nu)$
and zero on $L^p (E, \nu)$.
So (\ref{Eq_5208_MaxN}) holds.

Now we prove uniqueness.
Let $\| \cdot \|$ be a norm as in the statement.
For $n \in \N$, $a \in e_n A e_n$, and $\ld \in \C$
we prove that
\begin{equation}\label{Eq_5209_Known}
\| a + \ld \cdot 1 \|
 = \max ( \| a + \ld e_n \|, \, | \ld | ).
\end{equation}
Since the right hand side of~(\ref{Eq_5209_Known})
depends only on the norm on~$A$,
and since ${\textdisp{ \bigcup_{n = 1}^{\infty} e_n A e_n }}$ is
dense in~$A$, uniqueness will follow.

It follows from~(\ref{P_5209_NormOnUnit_Spatial})
that $e_n$ is a hermitian idempotent in~$A^{+}$.
Also, $e_n$ commutes with $a + \ld \cdot 1$
and $\| 1 - e_n \| = \| \te_e^+ ({{\ch}}_{{K}}) \| \leq 1$.
So Lemma~\ref{L_5209_CommuteSp}
implies that
\[
\| a + \ld \cdot 1 \|
 = \max \big( \| e_n (a + \ld \cdot 1) e_n \|,
   \, \| (1 - e_n) (a + \ld \cdot 1) (1 - e_n) \| \big)
 = \max ( \| a + \ld e_n \|, \, | \ld | ),
\]
which is~(\ref{Eq_5209_Known}).
\end{proof}

\begin{prp}\label{P_4327_IsLpAlg}
Let $p \in [1, \infty) \setminus \{ 2 \}$, and let
${\textdisp{ A = \dirlim_m (A_m, \ph_{n, m}) }}$ be a spatial
$L^p$~AF algebra,
expressed as a direct limit as in Definition~\ref{D_4327_SpSys},
and with canonical maps $\ph_n \colon A_n \to A$ for $n \in \N$.
Then $A$ is a separable
nondegenerately representable
$L^p$~operator algebra.
Moreover, $e = ( \ph_n (1_{A_n} ))_{n \in \N}$
is a nondecreasing
approximate identity of idempotents
such that
the corresponding \hm{}
$\te_e$ of Definition~\ref{D_5204_CcN} is contractive.
\end{prp}

\begin{proof}
Theorem~\ref{T_4310_LpDLim} and Lemma~\ref{L_4326_FDIsLp}
imply that $A$ is an $L^p$~operator algebra.
Separability is obvious.
We prove the statement about the approximate identity.
By Proposition~\ref{P_5204_SpCcN}, this will imply that
${\textdisp{ A = \dirlim_m (A_m, \ph_{n, m}) }}$
is nondegenerately representable.

For
$n \in \N$, write $f_n$
for the identity of~$A_n$,
and set $e_n = \ph_n (f_n)$.
It is clear that $\| e_n \| \leq 1$
(with equality unless $A_n = 0$),
and that $e_n \leq e_{n + 1}$.

Set $e = (e_n)_{n \in \N}$,
as in the statement of the theorem.
Then we have
\[
\limi{n} e_n \ph_m (a)
    = \limi{n} \ph_m (a) e_n = \ph_m (a)
\]
for every $m \in \N$ and $a \in A_m$.
Since ${\textdisp{ \bigcup_{m \in \N}\ph_m (A_m) }}$
is dense in~$A$ and $\| e_n \| \leq 1$ for all $n \in \N$,
a standard $\frac{\ep}{3}$-argument shows that
$e$ is an approximate identity for~$A$.

It remains to prove that $\te_e$ is contractive.
We prove by induction on $n$ that,
with $\ph_{n, 0} (f_0)$ taken to be zero,
the idempotents $\ph_{n, j} (f_j) - \ph_{n, j - 1} (f_{j - 1})$
are hermitian for $j = 1, 2, \ldots, n$.
For $n = 1$, this is just the assertion that the
identity is a hermitian idempotent in~$A_1$.
If the statement is known for~$n$,
then for $j = 1, 2, \ldots, n$
we have
\[
\ph_{n + 1, \, j} (f_j) - \ph_{n + 1, \, j - 1} (f_{j - 1})
 = \ph_{n + 1, \, n}
  \big( \ph_{n, j} (f_j) - \ph_{n, j - 1} (f_{j - 1}) \big),
\]
which is a hermitian idempotent by
Lemma~\ref{L_4Y18_Funct}.
Also,
\[
\ph_{n + 1, \, n + 1} (f_{n + 1}) - \ph_{n + 1, \, n} (f_n)
 = 1_{A_{n + 1}} - \ph_{n + 1, \, n} (f_n)
\]
is hermitian because $\ph_{n + 1, \, n} (f_n)$ is hermitian.
This completes the induction.

Corollary \ref{C_4Y17_Orth}(\ref{C_4Y17_Orth_Norm})
now implies that $\te_e |_{ C ( \{ 1, 2, \ldots, n \} ) }$
is contractive for all $n \in \N$.
It follows that $\te_e$ is contractive.
\end{proof}

\begin{prp}\label{injectivedirectsystem}
Let $p \in [1, \infty) \setminus \{ 2 \}$,
and let $A$ be a spatial $L^p$~AF algebra.
Then $A$ is isometrically isomorphic to the direct
limit of a spatial $L^p$~AF direct system in which all
the connecting maps are injective.
\end{prp}

\begin{proof}
Let
$\big( (A_m)_{m \in \Nz}, \, ( \ph_{n, m} )_{m \leq n} \big)$
be a spatial $L^p$~AF direct system such that
$A$ is isometrically isomorphic to ${\textdisp{ \dirlim_m  A_m }}$.
Then for $m \in \N$ we can write
${\textdisp{ A_m = \bigoplus_{j = 1}^{N (m)} M_{d (m, j)}^p }}$
with
\[
N (m) \in \Nz
\andeqn
d (m, 1), \, d (m, 2), \, \ldots, \, d (m, N (m)) \in \N.
\]
Set
${\textdisp{
 J_m = \bigcup_{n = m + 1}^{\I} {\operatorname{Ker}} (\ph_{n, m}) }}$.
Then $J_m$ is a closed ideal in $A_m$,
and $A_m / J_m$ is a spatial semisimple finite
dimensional $L^p$~operator algebra by Lemma~\ref{L_7629_Quot}.

Let $m \in \Nz$.
For $n \in \Nz$ with $n \geq m$,
Corollary~\ref{C_5208_CompSpatial}
shows that the induced homomorphism $A_m \to A_n / J_n$ is spatial.
Lemma~\ref{L_4326_CharSpHmFDalg}
can then be used to show that the induced homomorphism
${\overline{\ph}}_{n, m} \colon A_m / J_m \to A_n / J_n$
is spatial.
Clearly ${\overline{\ph}}_{n, m}$ is injective.
Now
$\big( (A_m / J_m)_{m \in \Nz}, \,
   (\overline{\ph}_{n, m})_{0 \leq m \leq n} \big)$
is a spatial $L^p$~AF direct system whose direct limit is
isometrically isomorphic to~$A$.
\end{proof}

\begin{cor}\label{C_6416_GenUniqMN}
Let $p \in [1, \infty) \setminus \{ 2 \}$,
and let $A$ be a spatial $L^p$~AF algebra.
Then $A$ has unique $L^p$~operator matrix norms.
\end{cor}

\begin{proof}
By Proposition~\ref{injectivedirectsystem},
there is a spatial $L^p$~AF direct system
$\big( (A_m)_{m\in\Nz}, \, ( \ph_{n, m} )_{0 \leq m \leq n} \big)$
with injective maps such that $A$ is isometrically
isomorphic to ${\textdisp{ \dirlim_m A_m }}$.
For $m, n \in \Nz$ with $m \leq n$,
Lemma~\ref{L_4326_BlkDiagHm}
implies that $\ph_{n, m}$ is
isometrically similar to a block diagonal homomorphism.
An injective block diagonal homomorphism is isometric,
so $\ph_{n, m}$ is isometric.
The result now
follows from Lemma~\ref{L_4326_FDIsLp} and
Proposition~\ref{P_5125_LimHasUniq}.
\end{proof}

\section{Classification of spatial $L^p$~AF algebras}\label{Sec_ClassN}

\indent
In this section we prove our main result,
the classification of spatial $L^p$~AF algebras based on their
scaled preordered $K_0$ groups.
Moreover, as for AF~algebras, we
show that every countable scaled Riesz group
can be realized as the scaled preordered $K_0$ group
of a spatial $L^p$~AF algebra.
Given the theory already developed,
the proofs are now essentially the same as in the C*~algebra case.

The original C*~theory of AF~algebras
is mainly due to Bratteli~\cite{Bra},
Elliott~\cite{Ell},
and Effros, Handelman, and Shen~\cite{EHS}.
As references for the entire theory,
we rely on Chapter~3 of~\cite{Bl3}
and on~\cite{EE}.
For a more detailed discussion of Riesz groups than is needed here,
see~\cite{Gd0}.

We only state the classification theorem in terms of
scaled preordered K-theory.
We don't discuss the connection with Bratteli diagrams,
since the relation between Riesz groups and Bratteli diagrams
is well known and the generalization of the AF~algebra classification
to spatial $L^p$~AF algebras introduces nothing new here.

We begin by describing the relevant K-theoretic background.

\begin{dfn}\label{D_6324_PreOrd}
A {\emph{preordered abelian group}} is a pair $(G, G_{+})$ in which
$G$ is an abelian group and $G_{+}$ is a subset of $G$ such that
$0 \in G_{+}$ and $G_{+} + G_{+} \subset G_{+}$.
For $\et, \mu \in G$
we write $\et \leq \mu$ to mean that $\mu - \et \in G_{+}$.

A {\emph{scaled preordered abelian group}} is a triple $(G, G_{+}, \Sm)$
such that $(G, G_{+})$ is a preordered abelian group,
and $\Sm$ (the {\emph{scale}})
is a subset of~$G_{+}$
such that $0 \in \Sm$.

If $(H, H_{+})$ is another preordered abelian group,
and $f \colon G \to H$ is a \hm,
then $f$ is {\emph{positive}} if $f (G_{+}) \subset H_{+}$.
If $\Sm \subset G_{+}$ and $\Gm \subset H_{+}$ are scales,
we say that $f$ is {\emph{contractive}} if $f (\Sm) \subset \Gm$.
\end{dfn}

The definitions are weak
because they are supposed to accommodate $K_0 (A)$
for any Banach algebra~$A$.
For example
(using the notation of Definition~\ref{D_6324_K0} below),
for the algebras
$C_0 (\R^2)$,
${\mathcal{O}}_{\infty}$,
and
${\mathcal{O}}_{n}$
we get the following results
for $\big( K_0 (A), \, K_0 (A)_{+}, \, \Sm (A) \big)$:
\[
(\Z, \, \{ 0 \}, \, \{ 0 \}),
\, \, \, \, \, \,
(\Z, \Z, \Z),
\andeqn
(\Z / (n - 1) \Z, \, \Z / (n - 1) \Z, \, \Z / (n - 1) \Z).
\]
The scale need not be hereditary if $A$ does not have cancellation.

We will take the $K_0$-group of a Banach algebra to be as in
Section~5 of~\cite{Bl3}.
(We will make very little use of the $K_1$-group,
and we don't recall its definition.)
To start,
we recall one of the standard equivalence relations on idempotents.
It is called algebraic equivalence in Definition 4.2.1 of~\cite{Bl3}.

\begin{dfn}\label{D_5128_EqPj}
Let $A$ be a Banach algebra.
Let $e$ and $f$ be
idempotents in $A$.
We say that
$e$ is {\emph{algebraically Murray-von Neumann equivalent}}
to $f$, denoted
by $e \sim f$, if there exist $x, y \in A$ such that $x y = e$
and $y x = f$.
\end{dfn}

\begin{dfn}\label{D_6324_K0}
Let $A$ be a ring.
We define $M_{\infty} (A)$ to be
the (algebraic) direct limit of the matrix rings $M_n (A)$
under the embeddings $a \mapsto \diag (a, 0)$.
(See Definition 5.1.1 of~\cite{Bl3}.)
We define $V (A)$
to be the abelian semigroup
of algebraic Murray-von Neumann equivalence
classes of idempotents in $M_{\infty} (A)$.
(See Definition 5.1.2 of~\cite{Bl3} and the discussion afterwards.)

When $A$ is a Banach algebra,
we define
$\big( K_0 (A), \, K_0 (A)_{+}, \, \Sm (A) \big)$
as follows.
We take $K_0 (A)$ to be the usual $K_0$-group of~$A$,
as in,
for example, Definition 5.5.1 of~\cite{Bl3}.
(There is trouble
if one uses the definition there for more general rings.)
For $n \in \N$ and an idempotent $e \in M_n (A)$,
we write $[e]$ for its class in $K_0 (A)$.
We take
$K_0 (A)_{+}$ to be the image of $V (A)$ in $K_0 (A)$
under the map coming from 5.5.2 and Definition 5.3.1 of~\cite{Bl3}.
We take $\Sm (A)$ to be the image
under this map of the subset of $V (A)$
consisting of the classes of idempotents in $A \subset M_{\infty} (A)$.
\end{dfn}

We warn that $[e]$ is sometimes used for the class of $e$ in $V (A)$.
Since $V (A) \to K_0 (A)$ need not be injective,
this is not the same as the class of $e$ in $K_0 (A)$.

\begin{rmk}\label{R_6402_StdDfn}
We can rewrite the definitions of $K_0 (A)_{+}$ and $\Sm (A)$
as
\[
K_0 (A)_{+}
 = \big\{ [e] \colon
  {\mbox{$e$ is an idempotent in $M_{\infty} (A)$}} \big\}.
\]
and
\[
\Sm (A)
 = \big\{ [e] \colon {\mbox{$e$ is an idempotent in $A$}} \big\}.
\]
\end{rmk}

\begin{prp}\label{P_6324_KIsOrd}
Let $A$ be a Banach algebra.
Then $\big( K_0 (A), \, K_0 (A)_{+}, \, \Sm (A) \big)$
is a scaled preordered abelian group
in the sense of Definition~\ref{D_6324_PreOrd}.
\end{prp}

\begin{proof}
This is immediate.
\end{proof}

Direct limits of direct systems of scaled preordered abelian
groups are constructed in the obvious way.

\begin{lem}\label{directlimitofpreorderedabegroups}
Let $I$ be a directed set.
For every $i \in I$ let
$(G_i, (G_i)_{+}, \Sm_i)_{i \in I}$ be a scaled preordered
abelian group, and for $i, j \in I$ with $i \leq j$ let
$g_{j, i} \colon G_i \to G_j$ be a positive contractive
homomorphism.
Let $G$ be the direct limit ${\textdisp{ \dirlim_i G_i }}$
as abelian groups
and for $i \in I$ let $g_i \colon G_i \to G$ be the canonical map.
Set
\[
G_{+} = \bigcup_{i \in I} g_i ((G_i)_{+})
\andeqn
\Sm = \bigcup_{i \in I} g_i (\Sm_i).
\]
Then $(G, G_{+}, \Sm)$ is a scaled preordered abelian group
and $(G, G_{+}, \Sm)$ is the direct limit of
$(G_i, (G_i)_{+}, \Sm_i)_{i \in I}$
in the category of scaled preordered abelian groups and
positive contractive homomorphisms.
\end{lem}

\begin{proof}
Without the scales, see Proposition 1.15 in~\cite{Gd0}.
The additional work for scaled preordered abelian groups
is easy, and is omitted.
\end{proof}

\begin{thm}\label{T_6324_LimK0}
The assignment
$A \mapsto \big( K_0 (A), \, K_0 (A)_{+}, \, \Sm (A) \big)$
is a functor from Banach algebras and \hm{s}
to scaled preordered abelian groups and contractive positive \hm{s}
which commutes with direct limits in which
the maps are contractive.
\end{thm}

\begin{proof}
Functoriality of $K_0 (A)$ is
stated after Definition 5.5.1 of~\cite{Bl3}.
Functoriality of the other two parts is clear.
(Also see 5.2.1 of~\cite{Bl3}.)

The fact that $K_0 (A)$ commutes with direct limits
is Theorem~6.4 of~\cite{PhLp3}.
The statement for $K_0 (A)_{+}$
follows from that for $V (A)$,
which is 5.2.4 of~\cite{Bl3}.
The statement for $\Sm (A)$
follows by the same proof,
which is Propositions 4.5.1 and 4.5.2 of~\cite{Bl3}.
\end{proof}

For the part about the scale in the following definition, we
refer the reader to the beginning of Chapter~7 of~\cite{EE}.
Riesz groups are sometimes called dimension groups,
for example in the definition at the beginning
of Chapter~3 of~\cite{Gd0}.

\begin{dfn}\label{D_6324_Riesz}
Let $(G, G_{+})$ be a preordered abelian group.
We say that $G$ is an {\emph{unperforated ordered group}}
if:
\begin{enumerate}
\item\label{D_6324_Riesz_Gen}
$G_{+} - G_{+} = G$.
\item\label{D_6324_Riesz_IntZ}
$G_{+} \cap (- G_{+}) = \{ 0 \}$.
\item\label{D_6324_Riesz_Unperf}
Whenever $\et \in G$ and $n \in \N$
satisfy $n \et \in G_{+}$,
then $\et \in G_{+}$.
\setcounter{TmpEnumi}{\value{enumi}}
\end{enumerate}
We say that $(G, G_{+})$ is a {\emph{Riesz group}}
if, in addition:
\begin{enumerate}
\setcounter{enumi}{\value{TmpEnumi}}
\item\label{D_6324_Riesz_Interp}
Whenever $\et_1, \et_2, \mu_1, \mu_2 \in G$
satisfy $\et_j \leq \mu_k$ for $j, k \in \{ 1, 2 \}$,
then there exists $\ld \in G$ such that
$\et_j \leq \ld \leq \mu_k$ for $j, k \in \{ 1, 2 \}$.
\setcounter{TmpEnumi}{\value{enumi}}
\end{enumerate}
Let $(G, G_{+}, \Sm)$ be a scaled preordered abelian group.
We say that $G$ is a {\emph{scaled Riesz group}}
if $(G, G_{+})$ is a Riesz group,
and in addition:
\begin{enumerate}
\setcounter{enumi}{\value{TmpEnumi}}
\item\label{D_6324_Riesz_ScGen}
For every $\et \in G_{+}$ there are $n \in \N$
and $\mu_1, \mu_2, \ldots, \mu_n \in \Sm$
such that $\et = \mu_1 + \mu_2 + \cdots + \mu_n$.
\item\label{D_6324_PreOrd_Hered}
Whenever $\et, \mu \in G$
satisfy $0 \leq \et \leq \mu$ and $\mu \in \Sm$,
then $\et \in \Sm$.
\item\label{D_6324_Riesz_Up}
For all $\et, \mu \in \Sm$
there is $\ld \in \Sm$
such that $\et \leq \ld$ and $\mu \leq \ld$.
\end{enumerate}
\end{dfn}

We recall for reference some standard definitions and facts.
A few are restated for the $L^p$~case.

\begin{dfn}\label{D_6324_ZmOrder}
For $N \in \N$ we make $\Z^N$ a Riesz group by taking
\[
(\Z^N)_{+} = \big\{ (\et_1,\et_2,\ldots,\et_N) \in \Z^N \colon
  {\mbox{$\et_k \geq 0$ for $k = 1, 2, \ldots, N$}} \big\}.
\]
\end{dfn}

\begin{rmk}\label{R_6402_Scales}
The possible scales on $\Z^N$ are
exactly the following sets.
Take $d = (d_1, d_2, \ldots, d_N) \in \Z^N$
with $d_k > 0$ for $k = 1, 2, \ldots, N$,
and define
\[
[0, d] = \big\{ (\mu_1, \mu_2,\ldots, \mu_N) \in (\Z^N)_{+} \colon
   {\mbox{$\mu_k \leq d_k$ for $k = 1, 2, \ldots, N$}} \big\}.
\]
See page~43 of~\cite{EE}.
\end{rmk}

\begin{rmk}\label{D_6403_KThDSum}
Let $p \in [1, \infty)$.
Let ${\textdisp{ A = \bigoplus_{j = 1}^M \MP{c_j}{p} }}$
be a spatial semisimple finite dimensional $L^p$~operator algebra.
Set $c = (c_1, c_2, \ldots, c_M)$.
As in the C*~algebra case
(see pages 55--56 of~\cite{EE}),
using the notation from Definition~\ref{D_6324_ZmOrder}
and Remark~\ref{R_6402_Scales},
we have
\[
\big( K_0 (A), \, K_0 (A)_{+}, \, \Sm (A) \big)
 \cong \big( \Z^M, \, (\Z^M)_{+}, \, [0, c] \big).
\]
The map $K_0 (A) \to \Z^M$
sends the class of an idempotent
${\textdisp{(e_1, e_2, \ldots, e_M)
  \in \bigoplus_{j = 1}^M M_n (\MP{c_j}{p}) }}$
to
\[
\big(
 \rank (e_{1}), \, \rank (e_{2}), \, \ldots, \, \rank (e_{N}) \big).
\]
If ${\textdisp{ B = \bigoplus_{j = 1}^N M_{d_j}^p }}$
is another spatial semisimple finite dimensional $L^p$~operator algebra,
and $\ph \colon A \to B$ is a \hm,
then $\ph_{*} \colon \Z^M \to \Z^N$
is given by the partial multiplicity matrix $m (\ph)$
of Definition \ref{D_5208_Mult}(\ref{D_5208_Mult_Gen}).
\end{rmk}

\begin{lem}\label{findim}
Let ${\textdisp{ A = \bigoplus_{j = 1}^M M_{c_j}^p }}$
and ${\textdisp{ B = \bigoplus_{k = 1}^N M_{d_k}^p }}$
be spatial semisimple finite dimensional $L^p$~operator algebras.
Let $f$ a positive contractive
homomorphism from $\big( K_0 (A), \, K_0 (A)_{+}, \, \Sigma (A) \big)$
to $\big( K_0 (B), \, K_0 (B)_{+}, \, \Sigma (B) \big)$.
Then there
exists a spatial homomorphism $\ph \colon A \to B$
such that $\ph_* = f$.
Moreover, $\ph$ is unique
up to isometric similarity, and it can be chosen to be block diagonal.
\end{lem}

\begin{proof}
The homomorphism $f$ from
\[
\big( K_0 (A), \, K_0 (A)_{+}, \, \Sigma (A) \big)
 \cong \big( \mathbb{Z}^M, \, \mathbb{Z}^M_{+}, \, [0, [1_{A}]] \big)
\]
to
\[
\big( K_0 (B), \, K_0 (B)_{+}, \, \Sigma (B) \big)
 \cong \big( \mathbb{Z}^N, \, \mathbb{Z}^N_{+}, \, [0, [1_{B}]] \big)
\]
is given by an $N \times M$ matrix
$m = (m_{k, j})_{1 \leq k \leq N, \, 1 \leq j \leq M}$
with entries in~$\Z$.
(See Remark~\ref{D_6403_KThDSum}).
One checks that positivity implies that the entries are in~$\Nz$
and that contractivity implies that
${\textdisp{\sum_{j = 1}^M m_{k, j} c_j \leq d_k }}$ for
$k = 1, 2, \ldots, N$.
Lemma \ref{L_5208_Summary}(\ref{L_5208_Summary_Exist})
implies that there is a block diagonal \hm{}
$\ph \colon A \to B$
such that $m (\ph) = m$, and it is clear
that $\ph_* = f$ (Remark~\ref{D_6403_KThDSum}).
It follows from Lemma \ref{L_5208_Summary}(\ref{L_5208_Summary_BDImpSp})
that $\ph$ is spatial.

Now suppose
$\ps \colon A \to B$ is another spatial
homomorphism such that $\ps_* = f$.
Lemma~\ref{L_4326_BlkDiagHm} implies that
$\ps$ is isometrically similar
to a block diagonal \hm.
Therefore we may assume that both
$\ph$ and $\ps$ are block diagonal \hm{s}.
It is easy to check that
if $n \in \N$ then
two block diagonal \hm{s} from $A$ to $M_n$
with the same partial multiplicities
are similar via a permutation matrix,
and are thus isometrically similar.
It is now immediate that $m (\ph) = m (\ps)$
implies that $\ph$ and $\ps$ are isometrically similar.
\end{proof}

\begin{thm}\label{T_6324_Exist}
Let $(G, G_{+}, \Sm)$ be a countable scaled preordered
abelian group.
Then $(G, G_{+}, \Sm)$ is a scaled Riesz
group \ifo{} for $n \in \Nz$
there are scaled Riesz groups $(G_n, (G_n)_{+}, \Sm_n)$,
each isomorphic to a group as in Remark~\ref{R_6402_Scales},
and positive contractive \hm{s}
$f_{n + 1, \, n} \colon G_n \to G_{n + 1}$,
such that,
if for $m, n \in \Nz$ with $m \leq n$ we set
\[
f_{n, m} = f_{n, \, n - 1}  \circ f_{n - 1, \, n - 2}
     \circ \cdots \circ f_{m + 1, \, m},
\]
then
\[
(G, G_{+}, \Sm)
 \cong \dirlim \big( (G_n, (G_n)_{+}, \Sm_n)_{n \in \Nz},
   \,  (f_{n, m})_{0 \leq m \leq n} \big).
\]
\end{thm}

\begin{proof}
The statement for a countable preordered abelian group $(G, G_+)$
is Theorem~2.2 in~\cite{EHS}.
Using Lemma~7.1 of~\cite{EE} one shows that
if $(G, G_{+}, \Sm)$ is a scaled Riesz group then the
homomorphisms in the commutative diagram in Lemma 2.1 of~\cite{EHS}
can be chosen to be positive and contractive.
With this choice of homomorphisms
in Theorem 2.2 of~\cite{EHS},
one obtains the result for a scaled Riesz group.
\end{proof}

\begin{cor}\label{C_6811_AFIsRiesz}
Let $p \in [1, \infty) \setminus \{ 2 \}$,
and let $A$ be a spatial $L^p$~AF algebra.
Then $\big( K_0 (A), \, K_0 (A)_{+}, \, \Sm (A) \big)$
is a scaled Riesz group.
\end{cor}

\begin{proof}
Use Theorem~\ref{T_6324_LimK0}, Remark~\ref{D_6403_KThDSum},
and Theorem~\ref{T_6324_Exist}.
\end{proof}

For completeness,
we also state the result for~$K_1$.

\begin{prp}\label{P_6811_K1AF}
Let $p \in [1, \infty) \setminus \{ 2 \}$,
and let $A$ be a spatial $L^p$~AF algebra.
Then $K_1 (A) = 0$.
\end{prp}

\begin{proof}
By Definition~\ref{D_4327_SpSys},
there is a spatial $L^p$~AF direct system
$\big( (A_m)_{m \in \Nz}, \, ( \ph_{n, m} )_{0 \leq m \leq n} \big)$
such that ${\textdisp{ A \cong \dirlim_{m} A_m }}$.
For $m \in \Nz$, write
${\textdisp{ A_m = \bigoplus_{j = 1}^{N (m)} M_{c_{m,j}}^p }}$.
Since
${\textdisp{ K_1 (A_m) = \bigoplus_{j = 1}^{N (m)} K_1 (M_{c_j}^p) }}$,
and since $K_1 (M_n^p) = 0$ for every $n \in \N$
by Example 8.1.2(a) in \cite{Bl3},
we obtain $K_1 (A_m) = 0$.
Since $K_1$ commutes with Banach algebra direct limits
(Remark 8.1.5 in \cite{Bl3}),
the conclusion follows.
\end{proof}

\begin{lem}[Proposition 5.5.5 of~\cite{Bl3}]\label{L_5917_AppIdIdem}
Let $A$ be a Banach algebra which has an approximate identity
consisting of idempotents.
Then $K_0 (A)$ is naturally isomorphic to
the Grothendieck group of $V (A)$.
\end{lem}

We can now give the main classification results.

\begin{thm}\label{T_6324_Classification_Exist}
Let $p \in [1, \infty)$.
Let $(G, G_{+}, \Sm)$ be a countable scaled Riesz group.
Then there exists a spatial $L^p$~AF algebra $A$
such that
\[
\big( K_0 (A), \, K_0 (A)_{+}, \, \Sm (A) \big)
 \cong (G, G_{+}, \Sm).
\]
\end{thm}

\begin{proof}
Choose a direct system as in Theorem~\ref{T_6324_Exist}.
For $n \in \Nz$,
Remark~\ref{D_6403_KThDSum}
shows that there is
a spatial semisimple finite dimensional $L^p$~operator algebra $A_n$
such that
\[
\big( K_0 (A_n), \, K_0 (A_n)_{+}, \, \Sm (A_n) \big)
 \cong (G_n, (G_n)_{+}, \Sm_n).
\]
Lemma~\ref{findim} provides a block diagonal \hm{}
$\ph_{n + 1, \, n} \colon A_n \to A_{n + 1}$
such that $(\ph_{n + 1, \, n})_* = f_{n + 1, \, n}$.
For $m, n \in \Nz$ with $m \leq n$ set
\[
\ph_{n, m} = \ph_{n, \, n - 1}  \circ \ph_{n - 1, \, n - 2}
     \circ \cdots \circ \ph_{m + 1, \, m} \colon
    A_m \to A_n.
\]
Then $\ph_{n, m}$ is spatial by Corollary~\ref{C_5208_CompSpatial} and
Lemma \ref{L_5208_Summary}(\ref{L_5208_Summary_BDImpSp}).

Define
${\textdisp{ A = \dirlim \big( (A_n)_{n \in \Nz},
   \,  (\ph_{n, m})_{0 \leq m \leq n} \big) }}$.
Then $A$ is a spatial $L^p$~AF algebra,
and
\[
\big( K_0 (A), \, K_0 (A)_{+}, \, \Sm (A) \big)
 \cong (G, G_{+}, \Sm)
\]
by Theorem~\ref{T_6324_LimK0}.
\end{proof}

The proof of Elliott's Theorem has two steps,
the first of which is entirely about the category of
scaled Riesz groups and positive contractive homomorphisms
(and which is exactly the same in every category of algebras),
and the second of which transfers the result to algebras
in the appropriate category.
The first step is the following lemma.
Even though it is a key step in the proof,
and the proof appears in a number of books,
we haven't found an explicit statement of this result in
the literature.

\begin{lem}\label{L_6418_MakeK0}
For every $m \in \Nz$ let $(G_m, (G_m)_{+}, \Sm_m)$
and $(H_m, (H_m)_{+}, T_m)$
be scaled Riesz groups,
each isomorphic to a group as in Remark~\ref{R_6402_Scales},
and for $m, n \in \Nz$
with $m \leq n$
let $g_{n, m} \colon G_m \to G_n$
and $h_{n, m} \colon H_m \to H_n$
be positive contractive \hm{s}
satisfying $g_{n, m} \circ g_{m, k} = g_{n, k}$
and $h_{n, m} \circ h_{m, k} = h_{n, k}$
whenever $0 \leq k \leq m \leq n$.
Set
\[
(G, G_{+}, \Sm) = \dirlim_n (G_n, (G_n)_{+}, \Sm_n)
\andeqn
(H, H_{+}, T) = \dirlim_n (H_n, (H_n)_{+}, T_n),
\]
and let $f \colon G \to H$ be an isomorphism of scaled ordered groups.
Then there exist $m_0, m_1,  \ldots, n_0, n_1,  \ldots \in \Nz$
such that $m_0 < m_1 < \cdots$ and $n_0 < n_1  < \cdots$,
and positive contractive \hm{s} $r_k \colon G_{m_k} \to H_{n_k}$
and $s_k \colon H_{n_k} \to G_{m_{k + 1}}$ for $k \in \Nz$,
such that the following diagram commutes:
\begin{displaymath}
\xymatrix@R2.8pc@C2.8pc{
G_{m_0} \ar[r]^{g_{m_1, m_0}} \ar[d]_{r_0} &
G_{m_1} \ar[r]^{g_{m_2, m_1}} \ar[d]_{r_1} &
G_{m_2} \ar[r]^{g_{m_3, m_2}} \ar[d]_{r_2} &
G_{m_3} \ar[d]_{r_3} \ar[r]^{g_{m_4, m_3}} &
\cdots \cdots \ar[r] & G \ar[d]^f \\
H_{n_0} \ar[r]_{h_{n_1, n_0}} \ar[ur]^{s_0} &
H_{n_1} \ar[r]_{h_{n_2, n_1}} \ar[ur]^{s_1} &
H_{n_2} \ar[r]_{h_{n_3, n_2}} \ar[ur]^{s_2} &
H_{n_3} \ar[r]_{h_{n_4, n_3}} \ar[ur]^{s_3} &
 \cdots \cdots \ar[r] & H,}
\end{displaymath}
and such that $f$ is the direct limit of the maps $r_k$ for $k \in \Nz$
and $f^{-1}$ is the direct limit of the maps $s_k$ for $k \in \Nz$.
\end{lem}

\begin{proof}
This result is contained,
in slightly different language,
in the proof of Theorem 7.3.2 of~\cite{Bl3}
(starting with the second diagram there).
\end{proof}

We also want the one sided version of Lemma~\ref{L_6418_MakeK0}.

\begin{lem}\label{L_6418_OneSided}
Let $(G, G_{+}, \Sm)$ and $(H, H_{+}, T)$
be as in Lemma~\ref{L_6418_MakeK0},
and let $f \colon G \to H$ be a positive contractive \hm.
Then there exist $m_0, m_1, \ldots$, $n_0, n_1, \ldots \in \Nz$
such that $m_0 < m_1 < \cdots$ and $n_0 < n_1 < \cdots$,
and positive contractive \hm{s}
$r_k \colon G_{m_k} \to H_{n_k}$ for $k \in \Nz$,
such that the following diagram commutes:
\begin{displaymath}
\xymatrix@R2.8pc@C2.8pc{
G_{m_0} \ar[r]^{g_{m_1, m_0}} \ar[d]_{r_0} &
G_{m_1} \ar[r]^{g_{m_2, m_1}} \ar[d]_{r_1} &
G_{m_2} \ar[r]^{g_{m_3, m_2}} \ar[d]_{r_2} &
G_{m_3} \ar[d]_{r_3} \ar[r]^{g_{m_4, m_3}} &
\cdots \cdots \ar[r] & G \ar[d]^f \\
H_{n_0} \ar[r]_{h_{n_1, n_0}} &
H_{n_1} \ar[r]_{h_{n_2, n_1}} &
H_{n_2} \ar[r]_{h_{n_3, n_2}} &
H_{n_3} \ar[r]_{h_{n_4, n_3}} &
 \cdots \cdots \ar[r] & H,}
\end{displaymath}
and such that $f$ is the direct limit of the maps $r_k$ for $k \in \Nz$.
\end{lem}

\begin{proof}
The proof is very similar to,
but slightly simpler than,
that of Lemma~\ref{L_6418_MakeK0}.
\end{proof}

\begin{thm}\label{T_6324_Classification_Uniq}
Let $p \in [1, \infty)$.
Let $A$ and $B$ be spatial $L^p$~AF algebras,
and let $f \colon K_0 (A) \to K_0 (B)$
define an isomorphism from
$\big( K_0 (A), \, K_0 (A)_{+}, \, \Sm (A) \big)$
to $\big( K_0 (B), \, K_0 (B)_{+}, \, \Sm (B) \big)$.
Then there is a completely isometric isomorphism
$\ph \colon A \to B$ such that $\ph_* = f$.
\end{thm}

\begin{proof}
By definition,
we can write $A$ and $B$
as direct limits of spatial $L^p$~AF direct systems,
\[
A = \dirlim_m
 \big( (A_m)_{m \in \Nz}, ( \af_{n, m})_{0 \leq m \leq n} \big)
\andeqn
B = \dirlim_m
 \big( (B_m)_{m \in \Nz}, ( \bt_{n, m})_{0 \leq m \leq n} \big).
\]
For $m \in \Nz$
let $\af_{m} \colon A_m \to A$
and $\bt_{m} \colon B_m \to B$
be the canonical maps.
Apply Lemma~\ref{L_6418_MakeK0}
with
\[
(G_m, (G_m)_{+}, \Sm_m)
 = \big( K_0 (A_m), \, K_0 (A_m)_{+}, \, \Sm (A_m) \big)
\]
and
\[
(H_m, (H_m)_{+}, T_m)
 = \big( K_0 (B_m), \, K_0 (B_m)_{+}, \, \Sm (B_m) \big)
\]
for $m \in \Nz$
(see Remark~\ref{D_6403_KThDSum}),
with $g_{m, n} = (\af_{n, m})_*$
and $h_{m, n} = (\bt_{n, m})_*$
whenever $m, n \in \Nz$ with $n \geq m$,
and with
\[
(G, G_{+}, \Sm)
 = \big( K_0 (A), \, K_0 (A)_{+}, \, \Sm (A) \big)
\]
and
\[
(H, H_{+}, T)
 = \big( K_0 (B), \, K_0 (B)_{+}, \, \Sm (B) \big)
\]
(justified by Theorem~\ref{T_6324_LimK0}).
Let
\[
m_0 < m_1 < \cdots,
\, \, \, \, \, \,
n_0 < n_1 < \cdots,
\, \, \, \, \, \,
r_k \colon G_{m_k} \to H_{n_k},
\andeqn
s_k \colon H_{n_k} \to G_{m_{k + 1}}
\]
be as in Lemma~\ref{L_6418_MakeK0},
making the diagram there commute.

We construct by induction on $k$
spatial \hm{s}
\[
\ph_k \colon A_{m_k} \to B_{n_k}
\andeqn
\ps_k \colon B_{n_k} \to A_{m_{k + 1}}
\]
such that $(\ph_k)_* = r_k$
and $(\ps_k)_* = s_k$
for $k \in \Nz$,
and such that the diagram
\begin{equation}\label{Eq_4620_Diag}
\xymatrix@R2.8pc@C2.8pc{
A_{m_0} \ar[r]^{\af_{m_1, m_0}} \ar[d]_{\ph_0} &
A_{m_1} \ar[r]^{\af_{m_2, m_1}} \ar[d]_{\ph_1} &
A_{m_2} \ar[r]^{\af_{m_3, m_2}} \ar[d]_{\ph_2} &
A_{m_3} \ar[d]_{\ph_3} \ar[r]^{\af_{m_4, m_3}} &
\cdots \cdots \ar[r] & A \\
B_{n_0} \ar[r]^{\bt_{n_1, n_0}} \ar[ur]^{\ps_0} &
B_{n_1} \ar[r]^{\bt_{n_2, n_1}} \ar[ur]^{\ps_1} &
B_{n_2} \ar[r]^{\bt_{n_3, n_2}} \ar[ur]^{\ps_2} &
B_{n_3} \ar[r]^{\bt_{n_4, n_3}} \ar[ur]^{\ps_3} &
 \cdots \cdots \ar[r] & B}
\end{equation}
commutes.

For the initial step,
use the existence statement in Lemma~\ref{findim}
to choose a spatial homomorphism
$\ph_0 \colon A_{m_0} \to B_{n_0}$ such that
$(\ph_0)_* = r_0$.
Use the existence statement in Lemma~\ref{findim}
to choose a spatial homomorphism
$\ps_0^{(0)} \colon B_{n_0} \to A_{m_{1}}$
such that $\big( \ps_0^{(0)} \big)_* = s_0$, and
use Corollary \ref{C_5208_CompSpatial}  and the uniqueness
statement in Lemma~\ref{findim}
to choose an invertible isometry $t \in A_{m_1}$
such that
\[
t \big( \ps_0^{(0)} \circ \ph_0 \big) (a) t^{-1} = \af_{m_1, m_0} (a)
\]
for all $a \in A_{m_0}$.
Define $\ps_0$ by
$\ps_0 (b) = t \ps_0^{(0)} (b) t^{-1}$ for $b \in B_{n_0}$.
For the induction step,
suppose we have $\ph_k$ and~$\ps_k$.
Use the existence statement in Lemma~\ref{findim}
to choose a spatial homomorphism
$\ph_{k + 1}^{(0)} \colon A_{m_{k + 1}} \to B_{n_{k+ 1}}$
such that $\big( \ph_{k + 1}^{(0)} \big)_* = r_{k + 1}$.
Use the uniqueness statement in Lemma~\ref{findim}
to choose an invertible isometry $v \in B_{n_{k+1}}$
such that
\[
v \big( \ph_{k + 1}^{(0)} \circ \ps_{k} \big) (b) v^{-1}
  = \bt_{n_{k + 1}, n_{k}} (b)
\]
for all $b \in B_{n_{k}}$,
and define $\ph_{k + 1}$ by
$\ph_{k + 1} (a) = v \ph_{k + 1}^{(0)} (a) v^{-1}$
for $a \in A_{m_{k + 1}}$.
The construction of $\ps_{k + 1}$ is now the same as
the construction of $\ps_0$ in the initial step.

For $k \in \Nz$,
the map $\ph_k$ is completely contractive
by Lemma~\ref{L_4326_BlkDiagHm},
Lemma \ref{L_5208_Summary}(\ref{L_5208_Summary_BDCC}),
and Proposition~\ref{P_5916_PropOfISHm}(\ref{P_5916_PropOfISHm_Cmp}).
Commutativity of the diagram~(\ref{Eq_4620_Diag})
therefore implies the existence
of a contractive \hm{}
$\ph \colon A \to B$
such that $\ph \circ \af_{m_k} = \bt_{n_k} \circ \ph_k$
for all $k \in \Nz$,
and $\ph$ must in fact be completely contractive.
Similarly,
we get a completely contractive \hm{}
$\ps \colon B \to A$
such that $\ps \circ \bt_{n_k} = \af_{m_{k + 1}} \circ \ps_k$
for all $k \in \Nz$.
Using the universal property of direct limits,
we find that $\ph \circ \ps = \id_B$
and $\ps \circ \ph = \id_A$.
Therefore $\ph$ and $\ps$ are completely isometric.
It is clear that $\ph_* = f$.
\end{proof}

\begin{thm}\label{T_6420_ExistHom}
Let $p \in [1, \infty)$.
Let $A$ and $B$ be spatial $L^p$~AF algebras,
and let $f \colon K_0 (A) \to K_0 (B)$
define a positive contractive homomorphism from
$\big( K_0 (A), \, K_0 (A)_{+}, \, \Sm (A) \big)$
to $\big( K_0 (B), \, K_0 (B)_{+}, \, \Sm (B) \big)$.
Then there is a completely contractive homomorphism
$\ph \colon A \to B$ such that $\ph_* = f$.
\end{thm}

\begin{proof}
The proof is a one sided version of the proof of
Theorem~\ref{T_6324_Classification_Uniq},
using Lemma~\ref{L_6418_OneSided}
in place of Lemma~\ref{L_6418_MakeK0}.
\end{proof}

\begin{thm}\label{T_6417_FullClass}
Let $p \in [1, \infty)$.
Let $A$ and $B$ be spatial $L^p$~AF algebras.
Then \tfae:
\begin{enumerate}
\item\label{T_6417_FullClass_KTh}
$\big( K_0 (A), \, K_0 (A)_{+}, \, \Sm (A) \big)
 \cong \big( K_0 (B), \, K_0 (B)_{+}, \, \Sm (B) \big)$.
\item\label{T_6417_FullClass_Ring}
$A \cong B$ as rings.
\item\label{T_6417_FullClass_BIso}
$A$ is isomorphic to $B$ (not necessarily isometrically)
as Banach algebras.
\item\label{T_6417_FullClass_Isometric}
$A$ is isometrically isomorphic to $B$
as Banach algebras.
\item\label{T_6417_FullClass_CI}
$A$ is completely isometrically isomorphic to $B$
as matrix normed Banach algebras.
\end{enumerate}
\end{thm}

\begin{proof}
It is trivial that
(\ref{T_6417_FullClass_CI}) implies~(\ref{T_6417_FullClass_Isometric}),
that
(\ref{T_6417_FullClass_Isometric})
implies~(\ref{T_6417_FullClass_BIso}),
and that
(\ref{T_6417_FullClass_BIso}) implies~(\ref{T_6417_FullClass_Ring}).
Since $V (A)$ depends only on the ring structure of~$A$,
Lemma~\ref{L_5917_AppIdIdem}
and Proposition~\ref{P_4327_IsLpAlg}
show that $K_0 (A)$ depends only on the ring structure of~$A$.
It now follows directly from the definitions
that $K_0 (A)_{+}$ and $\Sm (A)$
depend only on the ring structure of~$A$.
Thus (\ref{T_6417_FullClass_Ring})
implies~(\ref{T_6417_FullClass_KTh}).
The implication from~(\ref{T_6417_FullClass_KTh})
to~(\ref{T_6417_FullClass_CI})
is Theorem~\ref{T_6324_Classification_Uniq}.
\end{proof}

\end{document}